\documentclass[11pt]{article}
\usepackage[utf8]{inputenc} 
\usepackage{amstext,amsmath,latexsym,amsbsy,amssymb}
\usepackage{url}
\usepackage[colorlinks,linkcolor=magenta,citecolor=blue, pagebackref=true,backref=true]{hyperref}
\renewcommand*{\backrefalt}[4]{%
    \ifcase #1 \footnotesize{(Not cited.)}%
    \or        \footnotesize{(Cited on page~#2.)}%
    \else      \footnotesize{(Cited on pages~#2.)}%
    \fi}
\usepackage{amsfonts}
\usepackage[final]{graphicx}
\usepackage{amsfonts,amsthm}
\usepackage{epsf}
\usepackage{graphics}
\usepackage{psfrag}
\usepackage{times}
\usepackage{epsfig}
\usepackage{subfig}
\usepackage[titletoc]{appendix}
\usepackage{multirow}
\usepackage{hhline}
\usepackage{bbm}
\usepackage{color}
\usepackage{bigints}
\usepackage{algorithmic}
\usepackage{algorithm}

\long\def\comment#1{}
\definecolor{battleshipgrey}{rgb}{0.52, 0.52, 0.51}
\definecolor{darkgray}{rgb}{0.66, 0.66, 0.66}
\definecolor{darkgreen}{rgb}{0.0, 0.2, 0.13}
\definecolor{darkspringgreen}{rgb}{0.09, 0.45, 0.27}
\definecolor{dukeblue}{rgb}{0.0, 0.0, 0.61}
\definecolor{olivedrab7}{rgb}{0.24, 0.2, 0.12}
\definecolor{darkblue}{rgb}{0.0, 0.0, 0.55}
\definecolor{darkscarlet}{rgb}{0.34, 0.01, 0.1}
\definecolor{candyapplered}{rgb}{1.0, 0.03, 0.0}
\definecolor{ao(english)}{rgb}{0.0, 0.5, 0.0}
\definecolor{applegreen}{rgb}{0.55, 0.71, 0.0}

\renewcommand\vec[1]{\ensuremath\boldsymbol{#1}}

\newcommand{\Ocal}{\ensuremath{\mathcal{O}}}
\newcommand{\Ecal}{\ensuremath{\mathcal{E}}}
\newcommand{\Gcal}{\ensuremath{\mathcal{G}}}
\newcommand{\Pcal}{\ensuremath{\mathcal{P}}}

\newcommand{\Mcal}{\ensuremath{\mathcal{M}}}
\newcommand{\hba}{\ensuremath{\overline{h}}}
\newcommand{\vba}{\ensuremath{\overline{V}}}
\newcommand{\Psiba}{\ensuremath{\overline{\Psi}}}
\newcommand{\Scal}{\ensuremath{\mathcal{S}}}
\newcommand{\norm}[1]{\left\lVert#1\right\rVert}
\newcommand{\firstmer}{\ensuremath{G'}}
\newcommand{\trunc}{\ensuremath{G''}}
\newcommand{\secmer}{\ensuremath{\widetilde{G}}}

\DeclareMathOperator*{\argmin}{arg\,min}
\comment{
\setlength{\topmargin}{0 in}
\setlength{\textwidth}{5.5 in}
\setlength{\textheight}{8 in}
\setlength{\oddsidemargin}{0.5 in}
}

\theoremstyle{plain}

\newtheorem{theorem}{Theorem}

\numberwithin{theorem}{section}

\newtheorem{proposition}{Proposition}

\numberwithin{proposition}{section}

\newtheorem{lemma}{Lemma}

\numberwithin{lemma}{section}

\newtheorem{definition}{Definition}

\numberwithin{definition}{section}

\numberwithin{condition}{section}

\numberwithin{problem}{section}

\newtheorem{corollary}{Corollary}

\numberwithin{corollary}{section}

\numberwithin{assumption}{section}

\numberwithin{example}{section}

\numberwithin{conjecture}{section}

\theoremstyle{definition}

\numberwithin{remark}{section}

\renewenvironment{abstract}
 {\small
  \begin{center}
  \bfseries \abstractname\vspace{-.5em}\vspace{0pt}
  \end{center}
  \list{}{%
    \setlength{\leftmargin}{15mm}
    \setlength{\rightmargin}{\leftmargin}%
  }%
  \item\relax}
 {\endlist}

\date{\normalsize\today} 

\begin{document}
 
\begin{center}

{\bf{\LARGE{
On posterior contraction of parameters and \\
interpretability in Bayesian mixture modeling
}}}

\vspace*{.2in}
 {\large{
 \begin{tabular}{ccc}
  Aritra Guha$^{\star}$ & Nhat Ho$^{\dagger}$ &  XuanLong Nguyen$^{\star}$
 \end{tabular}
}}

 \vspace*{.2in}

 \begin{tabular}{c}
 Department of Statistics, University of Michigan$^\star$\\
 Department of EECS, University of California, Berkeley$^\dagger$
 \end{tabular}

\vspace*{.2in}

\today

\vspace*{.2in}

\begin{abstract}
We study posterior contraction behaviors for parameters of interest in the context of Bayesian mixture modeling, where the number of mixing components is unknown while the model itself may or may not be correctly specified. Two representative types of prior specification will be considered: one requires explicitly a prior distribution on the number of mixture components, while the other places a nonparametric prior on the space of mixing distributions. The former is shown to yield an optimal rate of posterior contraction on the model parameters under minimal conditions, while the latter can be utilized to consistently recover the unknown number of mixture components, with the help of a fast probabilistic post-processing procedure. We then turn the study of these Bayesian procedures to the realistic settings of model misspecification.  It will be shown that the modeling choice of kernel density functions plays perhaps the most impactful roles in determining the posterior contraction rates in the misspecified situations. Drawing on concrete posterior contraction rates established in this paper we wish to highlight some aspects about the interesting tradeoffs between model expressiveness and interpretability that a statistical modeler must negotiate in the rich world of mixture modeling.

\comment{
In Bayesian estimation of finite mixture models with unknown number of components, it is a common practice to use an infinite mixture model with Dirichlet process prior for the mixing component weights. However, with this prior, the convergence rate of the posterior distribution for finite mixtures is far from optimal. Also, the number of components cannot be estimated consistently. In this paper, we consider schemes which can overcome these limitations. We discuss a posterior-processing scheme which estimates the number of components consistently. In addition, we also show that the prior which uses symmetric Dirichlet weights for the mixing component weights of finite mixture models, and puts a discrete measure on the number of components obtains an optimal $\sqrt{n}$ convergence rate for the estimation of mixing measure relative to the Wasserstein metric. Furthermore, we also obtain convergence rates for the posterior with the above mentioned prior on mixtures in misspecified regimes.}

\end{abstract}
\end{center}

\section{Introduction}
\label{Section:introduction}
Mixture models are one of the most useful tools in a statistician's toolbox for analyzing heterogeneous data populations. They can be a powerful black-box modeling device to approximate the most complex forms of density functions. Perhaps more importantly, they help the statistician  express the data population's heterogeneous patterns and interpret them in a useful way~\cite{McLachlan-Basford-88,Lindsay-1995,Mengersen-etal-2001}. The following are common, generic and meaningful questions a practitioner of mixture modeling may ask: (I) how many mixture components are needed to express the underlying latent subpopulations, (II) how efficiently can one estimate the parameters representing these components and, (III) what happens to a mixture model based statistical procedure when the model is actually misspecified? 



How to determine the number of mixture components is a question that has long fascinated mixture modelers. Many proposed solutions approached this as a model selection problem. The number of model parameters, hence the number of mixture components, may be selected by optimizing with respect to some regularized loss function; see, e.g., \cite{Lindsay-1995,Kass-Raftery-bayes_factor-95,Dacunha_order_97} and the references therein. A Bayesian approach to regularization is to place explicitly a prior distribution on the number of mixture components~\cite{Nobile-94,Green-Richardson-97,Nobile-Fearnside-07,Miller-2016}. A convenient aspect of separating out the modeling and inference questions considered in (I) and (II) is that once the number of parameters is determined, the model parameters concerned by question (II) can be estimated and assessed via any standard parametric estimation methods. 

In a number of modern applications of mixture modeling to heterogeneous data, such as in topic modeling, the number of mixture components (the topics) may be very large and not necessarily a meaningful quantity~\cite{Blei-etal-03,Tang-etal-icml14}. In such situations, it may be appealing for the modeler to consider a nonparametric approach, where both (I) and (II) are considered concurrently. The object of inference is now the mixing measure which encapsulates all unknowns about the mixture density function. There were numerous works exemplifying this approach~\cite{Leroux-92,Figuereido-93,Ishwaran-James-Sun-01}. In particular, the field of Bayesian nonparametrics (BNP) has offered a wealth of prior distributions on the mixing measure based on which one can arrive at the posterior distribution of any quantity of interest related to the mixing measure~\cite{Hjort-etal-10}.

A common choice of such priors is the Dirichlet process~\cite{Ferguson-73,Blackwell-MacQueen,Sethuraman}, resulting in the famous Dirichlet process mixture models~\cite{Antoniak-74,Lo-84,Escobar-West}. Dirichlet process (DP) and its variants have also been adopted as a building block for more sophisticated hierarchical modeling, thanks to the ease with which computational procedures for posterior inference via Markov Chain Monte Carlo can be implemented~\cite{Teh-etal-06,Rodriguez-etal-08}. Moreover, there is a well-established asymptotic theory on how such Bayesian nonparametric mixture models result in asymptotically optimal estimation procedures for the population density. See, for instance,~\cite{Ghosal-Ghosh-Ramamoorthi-99,Ghosal-vanderVaart-07b,Shen-etal-13} for theoretical results specifically on DP mixtures, and~\cite{Ghosal-Ghosh-vanderVaart-00,Shen-Wasserman-01,Walker-Lijoi-Prunster-07} for general BNP models. The rich development in both algorithms and theory in the past decades has contributed to the widespread adoption of these models in a vast array of application domains.

For quite some time there was a misconception among quite a few practitioners in various application domains, a misconception that may have initially contributed to their enthusiasm for Bayesian nonparametric modeling, that the use of such nonparametric models eliminates altogether the need for determining the number of mixture components, because the learning of such a quantity is "automatic" from the posterior samples of the mixing measure. The implicit presumption here is that a consistent estimate of the mixing measure may be equated with a consistent estimate of the number of mixture components. This is not correct, as has been noted, for instance, by~\cite{Leroux-92} in the context of mixing measure estimation. More recently,~\cite{Miller-2014} explicitly demonstrated that the common practice of drawing inference about the number of mixture components via the DP mixture, specifically by reading off the number of support points in the Dirichlet's posterior sample, leads to an asymptotically inconsistent estimate. 

Despite this inconsistency result, it will be shown in this paper that it is still possible to obtain a consistent estimate of the number of mixture components using samples from a Dirichlet process mixtures, or any Bayesian nonparametric mixtures, by applying a simple and fast post-processing procedure on samples drawn from the DP mixture's posterior. On the other hand, the parametric approach of placing an explicit prior on the number of components yields both a consistent estimate of the number mixture component, and more notably, an optimal posterior contraction rate for component parameters, under a minimal set of conditions. It is worth emphasizing that all these results are possible only under the assumption that the model is well-specified, i.e., the true but unknown population density lies in the support of the induced prior distribution on the mixture densities.

As George Box has said, "all models are wrong", but more relevant to us, all mixture models are misspecified in some way. The statistician has a number of modeling decision to make when it comes to mixture models, including the selection of the class of kernel densities, and the support of the space of mixing measures. The significance of question (III) comes to the fore, because if the posterior contraction behavior of model parameters is very slow due to specific modeling choices, one has to be cautious about the interpretability of the parameters of interest. A very slow posterior contraction rate in theory implies that a given data set probably has relatively very slow influence on the movement of mass from the prior to the posterior distribution.

In this paper we study Bayesian estimation of model parameters with both well-specified and misspecified mixture models.
There are two sets of results. The first results resolve several outstanding gaps that remain in the existing theory and current practice of Bayesian parameter estimation, given that the mixture model is well-specified. 
The second set of results describes posterior contraction properties of such procedures when the mixture model is misspecified. We proceed to describe these results, related works and  implications to the mixture modeling practice.

\subsection{Well-specified regimes}
Consider discrete mixing measures $G=\sum_{i=1}^k p_i \delta_{\theta_i}$. 
Here, $\vec{p}=(p_1,\dots,p_k)$ is a vector of mixing weights, while atoms $\{\theta_i\}_{i=1}^k$ are elements in a given compact space $\Theta \in \mathbb{R}^d$. Mixing measure $G$ is combined with a likelihood function $f(\cdot|\theta)$ with respect to Lebesgue measure $\mu$ to yield a mixture density: $p_G(\cdot)=\int f(\cdot|\theta)\mathrm{d}G(\theta)=\sum_{i=1}^k p_i f(\cdot|\theta_i)$.
When $k<\infty$, we call this a \textit{finite mixture model} with $k$ components. We write $k=\infty$ to denote an \textit{infinite mixture model}. The atoms $\theta_i$'s are representatives of the underlying subpopulations.

Assume that $X_1,\ldots, X_n$ are i.i.d. samples from a mixture density $p_{G_0}(x) = \int f(x|\theta) \textrm{d}G_0(\theta)$, where $G_0$ is a discrete mixing measure with \emph{unknown} number of support points $k_0 <\infty$ residing in $\Theta$. In the overfitted setting, i.e., an upper bound $k_0 \leq \overline{k}$ is given so that one may work with an overfitted mixture with $\overline{k}$ mixture components, Chen~\cite{Chen-95} showed that the mixing measure $G_0$ can be estimated at a rate $n^{-1/4}$ under the $L_1$ metric, provided that the kernel $f$ satisfies a second-order identifiability condition -- this is a linear independence property on the collection of kernel function $f$ and its first and second order derivatives with respect to $\theta$. 

Asymptotic analysis of Bayesian estimation of the mixing measure that arises in both finite and infinite mixtures, where the convergence is assessed under Wasserstein distance metrics, was first investigated by Nguyen~\cite{Nguyen-13}. Convergence rates of the mixing measure under a Wasserstein distance can be directly translated to the convergence rates of the parameters in the mixture model. Under the same (second-order) identifiability condition, it can be shown that either maximum likelihood estimation method or a Bayesian method with a non-informative (e.g., uniform) prior yields a $(\log n/n)^{1/4}$ rate of convergence~\cite{Ho-Nguyen-EJS-16,Nguyen-13,Ishwaran-James-Sun-01}. Note, however, that $n^{-1/4}$ is not the optimal \emph{pointwise} rate of convergence. Heinrich and Kahn~\cite{Kahn_2018} showed that a distance based estimation method can achieve $n^{-1/2}$ rate of convergence under $W_1$ metric, even though their method may not be easy to implement in practice.~\cite{Ho-Nguyen-Ritov} described a minimum Hellinger distance estimator that achieves the same optimal rate of parameter estimation.

An important question in the Bayesian analysis is whether there exists a suitable prior specification for mixture models according to which the posterior distribution on the mixing measure can be shown to contract toward the true mixing measure at the same fast rate $n^{-1/2}$. Rousseau and Mengersen~\cite{Rousseau-Mengersen-11} provided an interesting result in this regard, which states that for overfitted mixtures with a suitable Dirichlet prior on the mixing weights $\vec{p}$, assuming that an upper bound to the number of mixture component is given, in addition to a second-order type identifiability condition, then the posterior contraction to the true mixing measure can be established by the fact that the mixing weights associated with all redundant atoms of mixing measure $G$ vanish at the rate close to the optimal $n^{-1/2}$. 

In our first main result given in Theorem~\ref{theorem:convergence_rate_well_specified}, we show that an alternative and relatively common choice of prior also yields optimal rates of convergence of the mixing measure (up to a logarithmic term), in addition to correctly recovering the number of mixture components, under considerably weaker conditions. In particular, we study the mixture of finite mixture (MFM) prior, which places an explicit prior distribution on the number of components $k$ and a (conditional) Dirichlet prior on the weights $\vec{p}$, given each value of $k$. This prior has been investigated by Miller and Harrison~\cite{Miller-2016}. Compared to the method of~\cite{Rousseau-Mengersen-11}, no upper bound on the true number of mixture components is needed. In addition, only first-order identifiability condition is required for the kernel density $f$, allowing our results to apply to popular mixture models such as location-scale Gaussian mixtures. We also note that the MFM prior is one instance in a class of modeling proposals, e.g.,~\cite{Nobile-94,Green-Richardson-97,Nobile-Fearnside-07} for which the established convergence behavior continues to hold. In other words, from an asymptotic standpoint, all is good on the parametric Bayesian front.

\comment{
There have been a large number of works proposing inference methods for this type of model ~\cite{Nobile-94,Nobile-Fearnside-07,Green-Richardson-97}. Miller \& Harrison ~\cite{Miller-2016} device an inference scheme using ideas analogous to the Stick-breaking construction of Sethuraman-Tiwari ~\cite{Sethuraman-Tiwari-81}. In Section~\ref{Well-specified MTM contraction} of this paper, we show that this scheme not only achieves contraction properties in the number of components, but also achieves optimal contraction rates for parameter estimation, for finite mixture models.
}

Our second main result, given in Theorem~\ref{theorem: merge-truncate-merge consistency}, is concerned with a Bayesian nonparametric modeling practice. A Bayesian nonparametric prior on mixing measures places zero mass on measures with finite support points, so the BNP model is misspecified with respect to the number of mixture components. Indeed, when $G_0$ has only finite support the true density $p_{G_0}$ lies at the boundary of the support of the class of densities produced by the BNP prior. Despite the inconsistency results mentioned earlier on the number of mixture components produced by Dirichlet process mixtures, we will show that this situation can be easily corrected by applying a post-processing procedure to the samples generated from the posterior distribution arising from the DP mixtures, or any sufficiently well-behaved Bayesian nonparametric mixture models. By "well-behaved" we mean any BNP mixtures under which the posterior contraction rate on the mixing measure can be guaranteed by an upper bound using a Wasserstein metric~\cite{Nguyen-13}.

Our post-processing procedure is simple, and motivated by the observation that a posterior sample of the mixing measure tends to produce a large number of atoms with very small and vanishing weights~\cite{Green-Richardson-clustering,Miller-2014}. Such atoms can be ignored by a suitable truncation procedure. In addition, similar atoms in the metric space $\Theta$ can also be merged in a systematic and probabilistic way. Our procedure, named Merge-Truncate-Merge algorithm, is guaranteed to not only produce a consistent estimate of the number of mixture components but also retain the posterior contraction rates of the original posterior samples for the mixing measure. Theorem~\ref{theorem: merge-truncate-merge consistency} provides a theoretical basis for the heuristics employed in practice in dealing with mixtures with unknown number of components \cite{Green-Richardson-clustering,Nobile-Fearnside-07}.

\subsection{Misspecified regimes} 
There are several ways a mixture model can be misspecified: either in the kernel density function $f$, or the mixing measure $G$, or both. Thus, in the misspecified setting, we assume that the data samples $X_1,\ldots,X_n$ are i.i.d. samples from a mixture density $p_{G_0,f_0}$, namely,
$p_{G_0,f_0}(x) = \int f_0(x|\theta) G_0(\textrm{d}\theta)$, where both $G_0$ and $f_0$ are unknown. The statistician draws inference from a mixture model $p_{G,f}$, still denoted by $p_{G}$ for short, where $G$ is a mixing measure with support on compact $\Theta$, and $f$ is a chosen kernel density function. In particular, a Bayesian procedure proceeds by placing a prior on the mixing measure $G$ and obtaining the posterior distribution on $G$ given the $n$-data sample. In general, the true data generating density $p_{G_0}$ lies outside the support of the induced prior on $p_G$. We study  the posterior behavior of $G$ as the sample size $n$ tends to infinity.


The behavior of Bayesian procedures under model misspecification has been investigated in the foundational work of~\cite{Kleijn-2006, Kleijn-Bernstein-Von-mises-2012}. 
This body of work focuses primarily on density estimation. In particular, assuming that the true data generating distribution's density lies outside the support of a Bayesian prior, then the posterior distribution on the model density can be shown to contract to an element of the prior's support, which is obtained by a Kullback-Leibler (KL) projection of the true density into the prior's support~\cite{Kleijn-2006}. 

It can be established that the posterior of $p_G$ contracts to a density $p_{G_*}$, where $G_*$ is a probability measure on $\Theta$ such that $p_{G_*}$ is the (unique) minimizer of the Kullback-Leilber distance $K(p_{G_0, f_0}, p_{G})$ among all probability measure $G$ on $\Theta$. This mere fact is readily deduced from the theory of~\cite{Kleijn-2006}, but the outstanding and relevant issue is whether the posterior contraction behavior carries over to that of $G$, and if so, at what rate. In general, $G_*$ may not be unique, so posterior contraction of $G$ cannot be established. Under identifiability, $G_*$ is unique, but still $G_* \neq G_0$.

This leads to the question about interpretability when the model is misspecified. Specifically, when $f\neq f_0$, it may be unclear how one can interpret the parameters that represent mixing measure $G$, unless $f$ can be assumed to be a reasonable approximation of $f_0$. Mixing measure $G$, too, may be misspecified, when the true support of $G_0$ may not lie entirely in $\Theta$. In practice, it is a perennial challenge to explicate the relationship between $G_*$ and the unknown $G_0$. In theory, it is mathematically an interesting question to characterize this relationship, if some assumption can be made on the true $G_0$ and $f_0$, but this is beyond the scope of this paper. Regardless of the truth about this relationship, it is important for the statistician to know how impactful a particular modeling choice on $f$ and $G$ can affect the posterior contraction rates of the parameters of interest.

The main results that we shall present in Theorem~\ref{theorem:posterior_rate_Gaussian} and Theorem~\ref{theorem:posterior_rate_Laplace} are on the posterior contraction rates of the mixing measure $G$ toward the limit point $G_*$, under very mild conditions on the misspecification of $f$. In particular, we shall require that the tail behavior of function $f$ is not much heavier than that of $f_0$ (cf. condition (P.5) in Section~\ref{section:Posterior Contraction:mis-specified}). Specific posterior contraction rates of contraction for $G$ are derived when $f$ is either Gaussian or Laplace density kernel, two representatives for supersmooth and ordinary smooth classes of kernel densities~\cite{Fan-91}. A key step in our proofs lies in several inequalities which provide upper bound of Wasserstein distances on mixing measures in terms of weighted Hellinger distances, a quantity that plays a fundamental role in the asymptotic characterization of misspecified Bayesian models~\cite{Kleijn-2006}. 

It is interesting to highlight that the posterior contraction rate for the misspecified Gaussian location mixture is the same as that of well-specified setting, which is nonetheless extremely slow, in the order of $(1/\log n)^{1/2}$. On the other hand, using a misspecified Laplace location mixture results in some loss in the exponent $\gamma$ of the polynomial rate $n^{-\gamma}$. Although the contrast in contraction rates for the two families of kernels is quite similar to what is obtained for well-specified deconvolution problems for both frequentist methods~\cite{Fan-91,Zhang-90} and Bayesian methods~\cite{Nguyen-13,Gao-vdV-2016}, our results are given for misspecified models, which can be seen in a new light:
since the model is misspecified anyway, the statistician should be "free" to choose the kernel that can yield the most favorable posterior contraction for the parameters of his/ her model. In that regard, Laplace location mixtures should always be preferred to a Gaussian location mixtures. Of course it is not always advisable to use a heavy-tailed kernel density function, as dictated by condition (P.5).

The relatively slow posterior contraction rate for $G$ is due to the fact that the limiting measure $G_*$ in general may have infinite support, regardless of whether the true $G_0$ has finite support or not. From a practical standpoint, it is difficult to interpret the estimate of $G$ if $G_*$ has infinite support. However, if $G_*$ happens to have a finite number of support points, which is bounded by a known constant, say $\overline{k}$, then by placing a suitable prior on $G$ to reflect this knowledge we show that the posterior of $G$ contracts to $G_*$ at a relatively fast rate $(\log n/n)^{1/4}$. This is the same rate obtained under the well-identified setting for overfitted mixtures.

\subsection{Further remarks}
The posterior contraction theorems in this paper provide an opportunity to re-examine several aspects of the fascinating picture about the tension between a model's expressiveness and its interpretability. They remind us once again about the tradeoffs a modeler must negotiate for a given inferential goal and the information available at hand. We enumerate a few such insights:

\begin{itemize}
\item [(1)] "One size does not fit all": Even though the family of mixture models as a whole can be excellent at inferring about population heterogeneity and at density estimation as a black-box device, a specific mixture model specification cannot do a good job at both. For instance, a Dirichlet process mixture of Gaussian kernels may yield an asymptotically optimal density estimation machine but it performs poorly when it comes to learning of parameters. 

\item [(2)] "Finite versus infinite": If the number of mixture components is known to be small and an object of interest, then employing an explicit prior on this quantity results in the optimal posterior contraction rate for the model parameters and thus is a preferred method. When this quantity is known to be high or not a meaningful object of inference, Bayesian nonparametric mixtures provide a more attractive alternative as it can flexibly adapt to complex forms of densities. Regardless, one can still consistently recover the true number of mixture components using a nonparametric approach.

\item [(3)] "Some forms of misspecification are more useful than others". When the mixture model is misspecified, careful design choices regarding the (mispecified) kernel density and the support of the mixing measure can significantly speed up the posterior contraction behavior of model parameters. For instance, a heavy-tailed and ordinary smooth kernel such as the Laplace, instead of the Gaussian kernel, is shown to be especially amenable to efficient parameter estimation.

\end{itemize}

The remainder of the paper is organized as follows. Section \ref{section:preliminary} provides necessary backgrounds about mixture models, Wasserstein distances and several key notions of strong identifiability. 
Section \ref{section:Posterior Contraction:well-specified} presents posterior contraction theorems for well-mispecified mixture models for both parametric and nonparametric Bayesian models. 
Section \ref{section:Posterior Contraction:mis-specified} presents posterior contraction theorems when the mixture model is misspecified. In Section \ref{section:simulation}, we provide illustrations of the Merge-Truncate-Merge algorithm via a simulation study. Proofs of key results are provided in Section \ref{section:proofs} while the remaining proofs are deferred to the Appendices.

\paragraph{Notation} Given two densities $p, q$ (with respect to the Lebesgue measure $\mu$), the total variation distance is given by $V(p,q)={\displaystyle (1/2)\int {|p(x)-q(x)|}\textrm{d}\mu(x)}$.
Additionally, the squared Hellinger distance is given by  $h^{2}(p,q)=\\
{\displaystyle (1/2) \int (\sqrt{p(x)}-\sqrt{q(x)})^{2}\textrm{d}\mu(x)}$. Furthermore, the Kullback-Leibler (KL) divergence is given by $K(p,q)={\displaystyle \int \log(p(x)/q(x))p(x)\textrm{d}\mu(x)}$ and the squared KL divergence is given by $K_{2}(p,q)={\displaystyle \int \log(p(x)/q(x))^{2}p(x)\textrm{d}\mu(x)}$.
For a measurable function $f$, let $Q f$ denote the integral $\int f dQ$. For any $\kappa=(\kappa_{1},\ldots,\kappa_{d}) \in \mathbb{N}^{d}$, we 
denote $\dfrac{\partial^{|\kappa|}{f}}{\partial{\theta^{\kappa}}}(x|\theta) =
\dfrac{\partial^{|\kappa|}{f}}{\partial{\theta_{1}^{\kappa_{1}}\ldots
\partial{\theta_{d}^{\kappa_{d}}}}}(x|\theta)$ where $\theta=(\theta_{1},
\ldots,\theta_{d})$. For any metric $d$ on $\Theta$, we define the open ball of $d$-radius $\epsilon$ around $\theta_0 \in \Theta$ as $B_d(\epsilon,\theta_0)$. 
We use  $D(\epsilon, \Omega, \tilde{d})$ to denote the maximal $\epsilon$-packing number for a general set $\Omega$ under a general metric $\tilde{d}$ on $\Omega$. 
Additionally, the expression $a_{n} \gtrsim b_{n}$ will be used 
to denote the inequality up to a constant multiple where the value of the constant is 
independent of $n$. We also denote $a_{n} \asymp b_{n}$ if both $a_{n} \gtrsim b_{n}$ 
and $a_{n} \lesssim b_{n}$ hold. Furthermore, we denote $A^{c}$ as the complement of set $A$ for any set $A$ while $B(x,r)$ denotes the ball, with respect to the $l_2$ norm, of radius $r > 0$ centered at $x \in \mathbb{R}^{d}$. Finally, we use $\text{Diam}(\Theta)= \sup \{\|\theta_1-\theta_2\|: \theta_1,\theta_2 \in \Theta\}$  to denote the diameter of a given parameter space $\Theta$ relative to the $l_{2}$ norm, $\|\cdot\|$, for elements in $\mathbb{R}^{d}$.


\section{Preliminaries} 
\label{section:preliminary}
We recall the notion of Wasserstein distance for mixing measures, along with the notions of strong identifiability and uniform Lipschitz continuity conditions that prove useful in Section \ref{section:Posterior Contraction:well-specified}.

\paragraph{Mixture model}
Throughout the paper, we assume that $X_{1}, \ldots, X_{n}$ are i.i.d. samples from a true but unknown distribution $P_{G_0}$ with given density function 
 \begin{align}
p_{G_{0}} := \int f(x|\theta) dG_{0}(\theta) = \sum \limits_{i=1}^{k_{0}}{p_{i}^{0}f(x|\theta_{i}^{0})} \nonumber
 \end{align}
where $G_{0}=\sum \limits_{i=1}^{k_{0}}{p_{i}^{0}\delta_{\theta_{i}^{0}}}$ is a true but unknown mixing distribution with exactly $k_{0}$ number of support points, for some unknown $k_0$. Also, $\left\{f(x|\theta),\theta \in \Theta \subset \mathbb{R}^{d} \right\}$ is a given family of probability densities (or equivalently kernels) with respect to a sigma-finite measure $\mu$ on $\mathcal{X}$ where $d \geq 1$. Furthermore, $\Theta$ is a chosen parameter space, where we empirically believe that the true parameters belong to. In a well-specified setting, all support points of $G_0$ reside in $\Theta$, but this may not be the case in a misspecified setting. 
 
Regarding the space of mixing measures, let $\Ecal_{k}:=\Ecal_{k}(\Theta)$ and $\Ocal_{k}:=\Ocal_{k}(\Theta)$ respectively denote the space of all mixing measures with exactly and at most $k$ support points, all in $\Theta$. Additionally, denote $\Gcal : = \Gcal(\Theta) = \mathop{\cup} \limits_{k \in \mathbb{N}_{+}}{\Ecal_{k}}$ the set of all discrete measures with finite supports on $\Theta$. Moreover, $\overline{\Gcal}(\Theta)$ denotes the space of all discrete measures (including those with countably infinite supports) on $\Theta$. Finally, $\Pcal(\Theta)$ stands for the space of all probability measures on $\Theta$.

\paragraph{Wasserstein distance}
As in~\cite{Nguyen-13,Ho-Nguyen-EJS-16} it is useful to analyze the identifiability and convergence
of parameter estimation in mixture models using the notion of Wasserstein distance, which
can be defined as the optimal cost of moving masses transforming one probability
measure to another~\cite{Villani-09}. Given two discrete measures $G=\mathop {\sum }\limits_{i=1}^{k}{p_{i}\delta_{\theta_{i}}}$ 
and
$G' = \sum_{i=1}^{k'}p'_i \delta_{\theta'_i}$, a coupling between $\vec{p}$ and
$\vec{p'}$ is a joint distribution $\vec{q}$ on $[1\ldots,k]\times [1,\ldots, k']$, which
is expressed as a matrix
$\vec{q}=(q_{ij})_{1 \leq i \leq k,1\ \leq j \leq k'} \in [0,1]^{k \times k'}$
with marginal probabilities
$\mathop {\sum }\limits_{i=1}^{k}{q_{ij}}=p_{j}'$ and $\mathop {\sum  }\limits_{j=1}^{k'}{q_{ij}}=p_{i}$ for any $i=1,2,\ldots,k$ and $j=1,2,\ldots,k'$.
We 
use $\mathcal{Q}(\vec{p},\vec{p'})$ to denote the space of all such couplings. For any $r \geq 1$, the $r$-th order Wasserstein distance between
$G$ and $G'$ is given by
\begin{eqnarray}
W_{r}(G,G') & = &
 \inf_{\vec{q} \in \mathcal{Q}(\vec{p},\vec{p'})}\biggr ({\mathop {\sum }\limits_{i,j}{q_{ij}\|\theta_{i}-\theta_{j}'\|^{r}}}
\biggr )^{1/r}, \nonumber
\end{eqnarray}
where $\|\cdot\|$  denotes the $l_{2}$ norm for
elements in $\mathbb{R}^{d}$. It is simple to see that if a sequence of probability measures $G_{n} \in \Ocal_{k_{0}}$
converges to $G_{0} \in \mathcal{E}_{k_{0}}$ under the $W_{r}$ metric
at a rate $\omega_{n} = o(1)$ for some $r \geq 1$ then there exists a subsequence of $G_{n}$ such that the set of
atoms of $G_{n}$ converges to the $k_{0}$ atoms of $G_{0}$, up to a permutation of the atoms, at the same rate $\omega_{n}$.

\paragraph{Strong identifiability and uniform Lipschitz continuity} 
The key assumptions that will be used to analyze the posterior contraction of mixing measures include uniform Lipschitz condition and strong identifiability condition. The uniform Lipschitz condition can be formulated as follows~\cite{Ho-Nguyen-EJS-16}.

\begin{definition} \label{definition:uniform_Lipschitz_condition}
We say the family of densities $\left\{f(x|\theta), \theta \in \Theta \right\}$ is uniformly Lipschitz up to
the order $r$, for some $r \geq 1$, if $f$ as a function of $\theta$ is differentiable up to the order $r$ and its partial derivatives with respect to $\theta$
satisfy the following inequality
\begin{eqnarray}
\sum_{|\kappa| = r} \biggr | \bigg (\frac{\partial^{|\kappa|} f}{\partial \theta^\kappa}
(x|\theta_1) -
 \frac{\partial^{|\kappa|} f}{\partial \theta^\kappa}(x|\theta_2) \biggr ) \gamma^\kappa\biggr |
\leq C \|\theta_1-\theta_2\|^\delta \|\gamma\| \nonumber
\end{eqnarray}
for any $\gamma \in \mathbb{R}^{d}$ and for some positive constants $\delta$ and $C$ independent of $x$ and $\theta_{1},\theta_{2} \in \Theta$. Here,  $\gamma^\kappa=\prod \limits_{i=1}^{d}{\gamma_{i}^{\kappa_{i}}}$ where $\kappa=(\kappa_{1},\ldots,\kappa_{d})$.
\end{definition}
The first order uniform Lipschitz condition is satisfied by many popular classes of density functions, including Gaussian, Student's t, and skew-normal family. Now, strong  identifiability condition of the $r^{th}$ order is formulated as follows, 
\begin{definition} \label{definition:strong_order_identifiability}
For any $r \geq 1$, we say that the family $\left\{f(x|\theta), \theta \in \Theta \right\}$ (or in short, $f$) is
\emph{identifiable in the order $r$}, for some $r \geq 1$, if $f(x|\theta)$ is differentiable up to the order $r$ in $\theta$
and the following holds
\begin{itemize}
\item[A1.] For any $k \geq 1$, given $k$ different elements
$\theta_{1}, \ldots, \theta_{k} \in \Theta$.
If we have $\alpha_{\eta}^{(i)}$ such that for almost all $x$
\begin{eqnarray}
\sum \limits_{l=0}^{r}{\sum \limits_{|\eta|=l}{\sum \limits_{i=1}^{k}{\alpha_{\eta}^{(i)}\dfrac{\partial^{|\eta|}{f}}{\partial{\theta^{\eta}}}(x|\theta_{i})}}}=0 \nonumber
\end{eqnarray}
then $\alpha_{\eta}^{(i)}=0$ for all $1 \leq i \leq k$ and $|\eta| \leq r$.
\end{itemize}
\end{definition}
Many commonly used families of density functions satisfy the first order identifiability condition, 
including location-scale Gaussian distributions and location-scale Student's t-distributions. Technically speaking, strong identifiability conditions are useful in providing the guarantee that we have some sort of lower bounds of Hellinger distance between mixing densities in terms of Wasserstein metric between mixing measures. For example, if $f$ is identifiable in the first order, we have the following inequality~\cite{Ho-Nguyen-EJS-16}
\begin{eqnarray}
h(p_{G},p_{G_{0}}) \gtrsim W_{1}(G,G_{0}) \label{eq:lower_bound_first_order_identifiability}
\end{eqnarray}
for any $G \in \mathcal{E}_{k_{0}}$. It implies that for any estimation method that yields the convergence rate $n^{-1/2}$ for density $p_{G_{0}}$ under the Hellinger distance, the induced rate of convergence for the mixing measure $G_{0}$ is $n^{-1/2}$ under $W_{1}$ distance.


\section{Posterior contraction under well-specified regimes}
\label{section:Posterior Contraction:well-specified}
In this section, we assume that the mixture model is well-specified, i.e., the data are i.i.d. samples from the mixture density $p_{G_0}$, where mixing measure $G_0$ has $k_0$ support points in compact parameter space $\Theta \subset \mathbb{R}^d$. Within this section, we assume further that the true but unknown number of components $k_0$ is finite. 
A Bayesian modeler places a prior distribution $\Pi$ on a suitable subspace of $\overline{\Gcal}(\Theta)$. Then, the posterior distribution over $G$ is given by:
\begin{eqnarray}
\label{Bayes}
\Pi(G \in B \bigr|X_1,\dots,X_n) =\frac{\int_B \prod_{i=1}^n p_G(X_i) \mathrm{d}\Pi(G)}{\int_{\overline{\Gcal}(\Theta)} \prod_{i=1}^n p_G(X_i) \mathrm{d}\Pi(G)}
\end{eqnarray}
We are interested in the posterior contraction behavior of $G$ toward $G_0$, in addition to recovering the true number of mixture components $k_0$.

\subsection{Prior results}
The customary prior specification for a finite mixture is to use a Dirichlet distribution on the mixing weights and another standard prior distribution on the atoms of the mixing measure.
Let $H$ be a distribution with full support on $\Theta$. Thus, for a mixture of $k$ components, the full Bayesian mixture model specification takes the form:
\begin{eqnarray} \label{eq:Exact_fit_model}
\vec{p} = (p_{1}, \ldots, p_{k})  & \sim & \text{Dirichlet}_{k}(\gamma/k,\ldots,\gamma/k), \nonumber \\
\theta_{1}, \ldots, \theta_{k} & \overset{iid} \sim & H, \nonumber \\
X_1,\ldots, X_n \; | \; G = \sum_{i=1}^{k} p_i \delta_{\theta_i} & \stackrel{iid}{\sim} & p_G.
\end{eqnarray}

Suppose for a moment that $k_0$ is known, we can set $k=k_0$ in the above model specification. Thus we would be in an \textit{exact-fitted} setting. Provided that $f$ satisfies both first-order identifiability condition and the uniform Lipschitz continuity condition, $H$ is approximately uniform on $\Theta$, then according to~\cite{Nguyen-13,Ho-Nguyen-EJS-16} it can be established that as $n$ tends to infinity,
\begin{eqnarray}
\label{eq:posterior_results_exact_fit}
\Pi\biggr(G \in {\Ecal_{k_0}}(\Theta): W_1(G,G_0) \gtrsim (\log n/n)^{1/2} \biggr| X_1, \ldots, X_n\biggr) \overset{p_{G_0}} \longrightarrow 0.
\end{eqnarray} 
The $(\log n/n)^{1/2}$ rate of posterior contraction is optimal up to a logarithmic term. 

When $k_0$ is unknown, there may be a number of ways for the modeler to proceed. Suppose that an upper bound of $k_0$ is given, say $k_0 < \overline{k}$. Then by setting $k=\overline{k}$ in the above model specification, we have a Bayesian \textit{overfitted} mixture model. Provided that $f$ satisfies the second-order identifability condition and the uniform Lipschitz continuity condition, $H$ is again approximately uniform distribution on $\Theta$, then it can be established that~\cite{Nguyen-13,Ho-Nguyen-EJS-16}:
\begin{eqnarray}
\label{eq:posterior_results_overfit}
\Pi\biggr(G \in {\Ocal_{\overline{k}}}(\Theta): W_2(G,G_0) \gtrsim (\log n/n)^{1/4} \biggr| X_1, \ldots, X_n\biggr) \overset{p_{G_0}} \longrightarrow 0.
\end{eqnarray} 
This result does not provide any guarantee about whether the true number of mixture components $k_0$ can be recovered. The rate (upper bound) $(\log n/n)^{1/4}$ under $W_2$ metric implies that under the posterior distribution the redundant mixing weights of $G$ contracts toward zero at the rate $(\log n/n)^{1/2}$, but the posterior contraction to each of the $k_0$ atoms of $G_0$ occurs at the rate $(\log n/n)^{1/4}$ only. 

Interestingly, it can be shown by Rousseau and Mengersen~\cite{Rousseau-Mengersen-11} that with a more judicious choice of prior distribution on the mixing weights, one can achieve a near-optimal posterior contraction behavior. Specifically, they continued to employ the Dirichlet prior, but they required the Dirichlet's hyperparameters set to be sufficiently small: $\gamma/k \leq d/2$ in \eqref{eq:Exact_fit_model} 
where $k=\overline{k}$, $d$ is the dimension of the parameter space $\Theta$. Then, under some conditions on kernel $f$ approximately comparable to the second-order identifiability and the uniform Lipschitz continuity condition defined in the previous section, they showed that for any $\epsilon > 0$, as $n$ tends to infinity
\begin{eqnarray}
\label{RM}
\Pi\biggr( \exists  I \subset \{1,\ldots,k\}, |I| = k-k_0 \ \text{s.t.} \ \sum_{i \in I} p_i < n^{-1/2+ \epsilon} \biggr| X_1,\ldots,X_n \biggr)  \overset{p_{G_0}} \longrightarrow 1.
\end{eqnarray}

For a more precise statement along with the complete list of sufficient conditions leading to claim~\eqref{RM}, we refer the reader to the original theorem of~\cite{Rousseau-Mengersen-11}.
Although their theorem is concerned with only the behavior of the redundant mixing weights $p_i$, where $i \in I$, which vanish at a near-optimal rate $n^{-1/2+\epsilon}$, it can be deduced from their proof that the posterior contraction for the true atoms of $G_0$ occurs at this near-optimal rate as well.~\cite{Rousseau-Mengersen-11} also showed that this performance may not hold if the Dirichlet's hyperparameters are set to be sufficiently large. Along this line, concerning the recovery of the number of mixture components $k_0$,~\cite{Chambaz-Rousseau-order-08} 
demonstrated the convergence of the posterior mode of the number of components to the true number of components $k_0$ at a rate $n^{-\rho}$, 
where $\rho$ depends on $\overline{k}-k_0$, the number of redundant components forced upon by our model specification.

\comment{Even though these results 
provide an encouragement for the use of over-fitted framework for 
model~\eqref{eq:Exact_fit_model} under mixtures, they require the 
knowledge of an upper bound for the number of components. There exists 
a large number of scenarios, for example, in topic modeling~\cite{Blei-et-al} 
where the true number of topics can be fairly large (see for example~\cite{Yurochkin-Guha-Nguyen-17} and the references therein). As a 
consequence, imposing an upper bound for the true number of components 
would not be effective under that setup. 
}

\subsection{Optimal posterior contraction via a parametric Bayesian mixture}
\label{Well-specified MTM contraction}

\comment{
A natural alternative choice of priors for density estimation of 
mixture models has been the Dirichlet process prior which places 
positive mass on infinite number of components. Because of suitable 
properties such as exchangeability and effectiveness in absence of 
knowledge of true number of components, the Dirichlet process prior has 
been a very popular tool in Bayesian nonparametrics as can be seen in~\cite{Antoniak-74,MacEachern-Muller-98}. 
Good rates for posterior of mixture densities with this prior, as shown 
in Ghosal and van der Vaart~\cite{Ghosal-vanderVaart-01}, also 
strengthen the argument in favor of it. However, Dirichlet Process 
mixtures are known to be inconsistent for the number of components in 
the frequentist sense, as shown by Miller and Harrison \cite{Miller-2014}. 
That inconsistency gives rise to an open problem in the literature 
about the effective Bayesian framework to determine the true number of 
components.}

We will show that optimal posterior contraction rates for mixture model parameters can be achieved by a natural Bayesian extension on the prior specification, even when the upper bound on the number of mixture component $k$ is unknown. The modeling idea is simple and truly Bayesian in spirit: since $k_0$ is unknown, let $K$ be a natural-valued random variable representing the number of mixture components. We endow $K$ with a suitable prior distribution $q_K$ on the positive integers. Conditioning on $K=k$, for each $k$, the model is specified as before:
\begin{eqnarray} \label{eq:MFM_model}
K & \sim & q_K, \\
\vec{p} = (p_{1}, \ldots, p_{k}) | K=k  & \sim & \text{Dirichlet}_{k}(\gamma/k,\ldots,\gamma/k), 
 \nonumber \\
\theta_{1}, \ldots, \theta_{k}\; | \;K=k & \overset{iid} \sim & H, \nonumber \\
X_1,\ldots,X_n\; | \; G = \sum_{i=1}^{k} p_i \delta_{\theta_i} & \stackrel{iid}{\sim} & p_G \; .
\end{eqnarray}

This prior specification is called \textit{mixture of finite mixtures} (MFM) model~\cite{Green-Richardson-97, Stephen_2000, Miller-2016}. 
%
In the sequel we show that the application of the MFM prior leads to the optimal posterior contraction rates for the model parameters. Interestingly, such guarantees can be established under very mild conditions on the kernel density $f$: only the uniform Lipschitz continuity and the first-order identifiability conditions will be required. The first-order identifiability condition is the minimal condition for which the optimal posterior contraction rate can be established, since this condition is also necessary for exact-fitted mixture models to receive the $n^{-1/2}$ posterior contraction rate. We proceed to state such conditions.

\begin{itemize}
\item[(P.1)] The parameter space $\Theta$ is compact, while
kernel density $f$ is first-order identifiable 
and admits the uniform Lipschitz property up to the first order.
\item[(P.2)] The base distribution $H$ is absolutely continuous with respect to the Lebesgue measure $\mu$ on $\mathbb{R}^{d}$ and admits a density function $g(\cdot)$.
Additionally, $H$ is approximately uniform, i.e., $\min_{\theta \in \Theta} g(\theta) > c_0 > 0$.
\item[(P.3)] There exists $\epsilon_{0}>0$ such that 
${\displaystyle \int (p_{G_{0}}(x))^{2}/p_{G}(x)d\mu(x) 
\leq M(\epsilon_{0}})$ as long as $W_{1}(G,G_{0}) 
\leq \epsilon_{0}$ for any $G \in \Ocal_{k_{0}}$ 
where $M(\epsilon_{0})$ depends only on $\epsilon_{0}$, $G_{0}$, and $\Theta$.
\item[(P.4)] The prior $q_{K}$ places positive mass on the set of natural numbers, i.e., $q_{K}(k) > 0$ for all $k \in \mathbb{N}$. 
\end{itemize}
\begin{theorem} \label{theorem:convergence_rate_well_specified}  
Under assumptions (P.1), (P.2), (P.3), and (P.4) on MFM, we have that
\begin{itemize}
\item[(a)] $\Pi(K= k_0|X_1,\ldots,X_n) \to 1$ a.s. under $P_{G_0}$.
\item[(b)] Moreover, \begin{eqnarray}
\Pi\biggr(G \in \overline{\mathcal{G}}(\Theta): W_{1}(G,G_{0}) 
		\lesssim (\log n/n)^{1/2}\biggr|
		X_{1},\ldots,X_{n}\biggr) \to 1 \nonumber
\end{eqnarray}
in $P_{G_{0}}$-probability. 

\end{itemize}
\end{theorem}
The proof of Theorem~\ref{theorem:convergence_rate_well_specified} is 
deferred to Section~\ref{proof:theorem:convergence_rate_well_specified}. 
We make several remarks regarding the conditions required in the theorem. It is worth stating up front that these conditions are almost minimal in order for the optimal posterior contraction to be guaranteed, and are substantially weaker than previous works (as discussed above). Assumption (P.1) is crucial in establishing that the 
Hellinger distance $h(p_{G},p_{G_{0}})\geq C_0 W_1(G,G_0)$ where $C_0$ is some positive constant depending only on $G_{0}$ and $\Theta$. Assumption (P.2) and (P.4) are standard conditions on the support of the prior so that posterior consistency can be guaranteed for any unknown $G_0$ with unknown number of support atoms residing on $\Theta$. Finally, the role of (P.3) is to help control the growing rate of KL neighborhood, which is central in the analysis of posterior convergence rate of mixing measures. This assumption is held for various choices of kernel $f$, including location families and location-scale families. Therefore, the assumptions (P.1), (P.2),(P.3) and (P.4) are fairly general and satisfied by most common choice of kernel densities.

\paragraph{Further remarks} Theorem~\ref{theorem:convergence_rate_well_specified} provides a positive endorsement for employing the MFM prior when the number of mixture components is unknown, but is otherwise believed to be finite and an important quantity of inferential interest. The papers of~\cite{Green-Richardson-97,Miller-2016} discuss additional favorable properties of this class of models.
However, when the true number of mixture components is large, posterior inference with the MFM may still be inefficient in practice. This is because much of the computational effort needs to be expended for the model selection phase, so that the number of mixture components can be reliably ascertained. Only then does the fast asymptotic rate of parameter estimation come meaningfully into effect.

\comment{
Finally, the strength of the MFM scheme also lies in the fact that contraction 
properties hold for any general distribution $q_{K}$ and $H$. Note that,
the contraction result in Theorem \ref{theorem:convergence_rate_well_specified} 
makes use of the fact that the parameter space $\Theta$ is compact. 
Choosing the appropriate $\Theta$, therefore, is based greatly on 
prior knowledge of the data. However, it is possible that the true mixing measure does not have all its atoms in the 
specified parameter space $\Theta$. Under this scenario of model misspecification, posterior contraction to 
the true parameters cannot be guaranteed, of course. We will investigate this setting in  Section~\ref{section:Posterior Contraction:mis-specified}.
}


\subsection{A posteriori processing for BNP mixtures}
\label{MTM Procedure}

Instead of placing a prior distribution explicitly on the number of mixture components when this quantity is unknown, another predominant approach is to place a Bayesian nonparametric prior on the mixing measure $G$, resulting in infinite mixture models. Bayesian nonparametric models such as Dirichlet process mixtures and the variants have remarkably extended the reach of mixture modeling into a vast array of applications, especially those areas where the number of mixture components in the modeling is very large and difficult to fathom, or when it is a quantity of only tangential interest. For instance, in topic modeling applications of web-based text corpora, one may be interested in the most "popular" topics, the number of topics is totally meaningless~\cite{Blei-etal-03,Teh-etal-06,Nguyen-admixture,Yurochkin-Guha-Nguyen-17}. DP mixtures and variants can also serve as an asymptotically optimal device for estimating the population density, under standard conditions on the true density's smoothness, see, e.g.,~\cite{Ghosal-vanderVaart-01,Ghosal-vanderVaart-07b,Shen-etal-13,Scricciolo_2014}. 

Since a nonparametric Bayesian prior such as the Dirichlet process places zero probability on mixing measures with finite number of supporting atoms, the Dirichlet process mixture's posterior is inconsistent on the number of mixture components, provided the true number of mixture components is finite~\cite{Miller-2014}. It is well known in practice that Dirichlet process mixtures tend to produce many small extraneous components around the "true" clusters, making them challenging to use to draw conclusion about the true number of mixture components when this becomes a quantity of interest~\cite{MacEachern-Muller,Green-Richardson-clustering}. In this section we describe a simple posteriori processing algorithm that consistently estimates the number of components for any general Bayesian prior, even without the exact knowledge of its structure as long as the posterior for that prior contracts at some known rate to the true $G_0$.

Our starting point is the availability of a mixing measure sample $G$ that is drawn from the posterior distribution $\Pi(G | X_1,\ldots, X_n)$, where $X_1,\ldots, X_n$ are i.i.d. samples of the mixing density $p_{G_0}$. Under certain conditions on the kernel density $f$, it can be established that for some Wasserstein metric $W_r$, as $n\rightarrow \infty$
\begin{eqnarray}
\label{eq:posterior_results}
\Pi\biggr(G \in \overline{\Gcal}(\Theta): W_r(G,G_0) \leq \delta \omega_n \biggr| X_1, \ldots, X_n\biggr) \overset{p_{G_0}} \longrightarrow 1
 \end{eqnarray} 
for \emph{all} constant $\delta > 0$, while $\omega_n = o(1)$ is a vanishing rate. Thus, $\omega_n$ can be taken to be (slightly) slower than actual rate of posterior contraction of the mixing measure. 
Concrete examples of the posterior contraction rates in infinite and (overfitted) finite mixtures are given in~\cite{Nguyen-13,Gao-vdV-2016,Ho-Nguyen-EJS-16}.

The posterior processing algorithm operates on an instance of mixing measure $G$, by suitably merging and truncating atoms that provide the support for $G$. The only inputs to the algorithm, which we call \textit{Merge-Truncate-Merge} (MTM) algorithm is $G$, in addition to the upper bound of posterior contraction rate $\omega_n$, and a tuning parameter $c>0$. The tuning parameter $c$ is useful in practice, as we shall explain, but in theory the algorithm "works" for any constant $c>0$. Thus, the method is almost "automatic" as it does not require any additional knowledge about the kernel density $f$ or the space of support $\Theta$ for the atoms. It is also simple and fast. We shall show that the outcome of the algorithm is a consistent estimate of both the number of mixing components and the mixing measure. The latter admits a posterior contraction rate's upper bound $\omega_n$ as well.

The detailed pseudocode of MTM algorithm is summarized in Algorithm~\ref{algo:merge-truncate-merge}. At a high level, it consists of two main stages. The first stage involves a probabilistic procedure for merging atoms that may be clustered near one another. The second stage involves a deterministic procedure for truncating extraneous atoms and merging them suitably with the remaining ones in a systematic way. The driving force of the algorithm lies in the asymptotic bound on the Wasserstein distance, i.e.,
$W_r(G,G_0) \leq c\omega_n$ with high probability. When $c\omega_n$ is sufficiently small, there may be many atoms that concentrate around each of the supporting atoms of $G_0$. Although $G_0$ is not known, such clustering atoms may be merged into one, by our first stage of probabilistic merging scheme. The second stage (truncate-merge) is also necessary in order to obtain a consistent estimate of $k_0$, because there remain distant atoms which carry a relatively small amount of mass. They will need to be suitably truncated and merged with the other more heavily supported atoms. In other words, our method can be viewed as a formal procedure of the common practices employed by numerous practitioners.
\comment{
in this section 
essentially consists of three steps. Step $1$ merges atoms close to 
one another using a probabilistic index assigning scheme, followed by a merge scheme. The index assigning scheme ensures that, with high probability the atoms of $G$ are merged to those atoms of $G$ which are close to atoms of $G_0$ themselves, rather than to outlying atoms of $G$. Step $2$ eliminates the merged atoms which have 
small weights. At this point, all atoms are sufficiently far apart 
from one another and have mass above a certain threshold. However, 
there still exist certain atoms which transport mass to a common 
atom of $G_0$ according to the optimal transportation plan. A 
consistent estimation scheme for the number of components needs to 
eliminate such redundancies in estimating the number of components 
of $G_0$ by its empirical counterpart. Step $3$ takes care of that 
by merging atoms which have low cost of mass transfer between them. 
}

\setcounter{algorithm}{0}
\begin{algorithm}
\caption{Merge-Truncate-Merge Algorithm}
\label{algo:merge-truncate-merge}
\begin{algorithmic}[1]
\algsetup{linenosize=\small}
\REQUIRE Posterior sample 
$G=\sum_{i } p_{i}\delta_{\theta_{i}}$ from~\eqref{eq:posterior_results}, rate $\omega_n$, constant $c$.
\ENSURE Discrete measure $\secmer$ and its number of supporting atoms $\tilde{k}$. 

\COMMENT{\textbf {Stage 1: Merge procedure:}}
\STATE Reorder atoms $\{\theta_1,\theta_2,\dots\}$ by simple random sampling without replacement with corresponding weights $\{p_1,p_2,\dots\}$. 
\begin{itemize}
\item[] let $\tau_1, \tau_2,\dots$ denote the new indices, and set $\mathcal{E} = \{\tau_j \}_{j}$ as the existing set of atoms.
\end{itemize}

\STATE Sequentially for each index $\tau_{j} \in \mathcal{E}$, if there exists an index $\tau_{i} < \tau_{j}$ such that 
$\| \theta_{\tau_i} - \theta_{\tau_j} \| \leq \omega_n$, then:
\begin{itemize}
        \item [] update $p_{\tau_i} = p_{\tau_i} + p_{\tau_j}$, and remove $\tau_j$ from $\mathcal{E}$.
    \end{itemize}
\STATE Collect $\firstmer = \sum_{j: \ \tau_j \in \mathcal{E}} p_{\tau_j} \delta_{\theta_{\tau_j}}$.
\begin{itemize}
\item [] write $G'$ as $\sum_{i = 1}^{k} q_{i}\delta_{\phi_i}$ so that $q_{1} \geq q_{2} \geq \dots$.
\end{itemize}

\item[]

\COMMENT{\textbf{Stage 2: Truncate-Merge procedure:}}
\STATE Set $\mathcal{A}=\{i : q_i > (c\omega_n)^r\}$, 
$\mathcal{N}=\{i : q_i \leq ( c\omega_n)^r\}$.

\STATE For each index $i \ \in \mathcal{A}$, if there is $j \in \mathcal{A}$ such that $j < i$ and $q_i \| \phi_i - \phi_j\|^{r} \leq (c\omega_n)^r$, then 
\begin{itemize}
    \item [] remove $i$ from $\mathcal{A}$ and add it to $\mathcal{N}$.
\end{itemize}

\STATE For each $i \in \mathcal{N}$, find atom $\phi_j$ among $j \in \mathcal{A}$ that is nearest to $\phi_i$
\begin{itemize}
\item [] update $q_j = q_j + q_i$.
\end{itemize}
\STATE Return $\secmer = \sum_{j \in \mathcal{A}} {q}_j \delta_{\phi_j}$ and $\tilde{k}= |\mathcal{A}|$.
\end{algorithmic}
\end{algorithm}

We proceed to present the theoretical guarantee for the outcome of Algorithm \ref{algo:merge-truncate-merge}.

\begin{theorem}
\label{theorem: merge-truncate-merge consistency}
Let $G$ be a posterior sample from posterior distribution of any Bayesian procedure, namely, $\Pi(\cdot | X_1,\ldots,X_n)$ according to which the upper bound~\eqref{eq:posterior_results} holds for all $\delta>0$. 
Let $\secmer$ and $\tilde{k}$ be the outcome of Algorithm~\ref{algo:merge-truncate-merge} applied to $G$, for an arbitrary constant $c>0$. 
Then the following hold as $n \rightarrow \infty$.
\begin{enumerate}
\item[(a)] $\Pi(\tilde{k} = k_0 |X_{1},\ldots,X_{n}) 
\to 1$ in $P_{G_0}$-probability.
\item[(b)]  For all $ \delta>0$,
      $\Pi\biggr( G \in \overline{\Gcal}(\Theta): W_r(\tilde{G},G_0) \leq \delta \omega_n \biggr| X_1, \ldots, X_n\biggr)  \longrightarrow 1$ in  $P_{G_0}$-probability.
\end{enumerate}
\end{theorem}
We add several comments concerning this theorem.
\begin{itemize}
    \item [(i)] The proof of this theorem is deferred to Section~\ref{ssub:proof::MTM consistency}, where we clarify carefully the roles played by each step of the MTM algorithm.
    \item  [(ii)] Although it is beyond the scope of this paper to study the practical viability of the MTM algorithm, for interested readers we present a brief illustration of the algorithm via simulations in Section~\ref{section:simulation}.
    \item [(iii)] In practice, one may not have a mixing measure $G$ sampled from the posterior $\Pi(\cdot | X_1,\ldots, X_n)$ but a sample from $G$ itself, say $F_n$. Then one can apply the MTM algorithm to $F_n$ instead. Assume that $F_n$ is sufficiently close to $G$, in the sense that $W_r(F_n,G) \lesssim W_r(G,G_0)$, it is straightforward to extend the above theorem to cover this scenario.
    \end{itemize}
    
\paragraph{Further remarks}  At this point, one may look forward to some guidance regarding the modeling choices of parametrics versus nonparametrics. Even in the tight arena of Bayesian mixture modeling, the jury may still be out. The results in this section seems to provide a stronger theoretical support for the former, when it comes to the efficiency of parameter estimation and the corresponding model interpretation. 

However, as we will see in the next section, when the mixture model is misspecified, the fast posterior contraction rate offered by the use of the MFM prior is no longer valid. On the other hand, Bayesian nonparametric models are more versatile in adapting to complex forms of population densities. In many modern applications it is not meaningful to estimate the number of mixing components, only the most "significant" ones in a sense suitably defined. Perhaps a more meaningful question concerning a Bayesian nonparametric mixture model is whether it is capable of learning selected  mixture components in an efficient way.

\section{Posterior contraction under model misspecification}
\label{section:Posterior Contraction:mis-specified}
In this section, we study the posterior contraction behavior of the mixing measure under the realistic scenarios of model misspecification. There are several ways a mixture model can be misspecified, due to the misspecification of the kernel density function $f$, or the support of the mixing measure $G$, or both. 
From here on, we shall assume that the data population follows a mixture distribution composed of unknown kernel density $f_0$ and unknown mixing measure $G_0$ --- thus, in this section the true density shall be denoted by $p_{G_0,f_0}$ to highlight the possibility of misspecification. 

To avoid heavy subscripting, we continue to use $p_{G}$ instead of $p_{G,f}$ to represent the density function of the mixture model that we operate on. The kernel density $f$ is selected by the modeler. Additionally, $G$ is endowed with a suitable prior $\Pi$ on the space of mixing measures with support belonging to compact parameter space $\Theta$. By Bayes rule (Eq.~\eqref{Bayes}) one obtains the posterior distribution $\Pi(G|X_1,\ldots, X_n)$, where the $n$-i.i.d. sample $X_1,\ldots, X_n$ are generated by $p_{G_0,f_0}$.  It is possible that $f\neq f_0$. It is also possible that the support of $G_0$ does not reside within $\Theta$. In practice, the statistical modeler would hope that the kernel choice of $f$ is not too different from the true but unknown $f_0$. Otherwise, it would be unclear how one can interpret the parameters that represent the mixing measure $G$.
Our goal is to investigate the posterior contraction of $\Pi(G|X_1,\ldots,X_n)$ in such situations, as sample size $n$ tends to infinity. The theory is applicable for a broad class of prior specification on the mixing measures on $\Theta$, including the MFM prior and a nonparametric Bayesian prior such as the Dirichlet process.

A fundamental quantity that arises in the theory of Bayesian misspecification for density estimation is the minimizer of the Kullback-Leibler (KL) distance from the true population density to a density function residing in the support of the induced prior on the space of densities $p_{G}$, which we shall assume to exist (cf.~\cite{Kleijn-2006}). Moreover, assume that the KL minimizer can be expressed as a mixture density $p_{G_*}$, where $G_*$ is a probability measure on $\Theta$. We may write
\begin{eqnarray}
G_{*} \in \mathop{\arg \min} \limits_{G \in \Pcal(\Theta)}
		{K(p_{G_0,f_0}, p_{G})}. \label{eq:KL_minimum}
\end{eqnarray}
We will see in the sequel that the existence of the KL minimizer $p_{G^*}$ entails its uniqueness. In general, however, $G_{*}$ may be non-unique. Thus, define
\begin{eqnarray}
\Mcal^{*}: = \left\{G_{*} \in \Pcal (\Theta): \
		G_{*} \in \mathop{\arg \min} \limits_{G \in 
		\Pcal(\Theta)}{K(p_{G_0,f_0}, p_{G})}\right\}. \nonumber
\end{eqnarray}

It is challenging to characterize the set $\Mcal^*$ in general. However, a very useful technical property can be shown as follows:
\begin{lemma} \label{lemma:misspecified_optimal}
For any $G \in \Pcal(\Theta)$ and $G_{*} \in \Mcal^{*}$, 
it holds that 
${\displaystyle \int \dfrac{p_{G}(x)}{p_{G_{*}}(x)}
		p_{G_0,f_0}(x) \mathrm{d}x \leq 1}$.
\end{lemma}
By exploiting the fact that the class of mixture densities is a convex set, the proof of this lemma is similar to that of Lemma 2.3 of~\cite{Kleijn-2006}, so it is omitted.
This leads quickly to the following fact.
\begin{lemma} \label{lemma:misspecified_optimal_unique}
For any two elements $G_{1,*}, G_{2,*} \in \Mcal^{*}$, $p_{G_{1,*}}(x)=p_{G_{2,*}}(x)$ for almost all $x \in \mathcal{X}$.
\end{lemma}
In other words, the mixture density $p_{G_*}$ is uniquely identifiable. Under a standard identifiability condition of the kernel $f$, which is satisfied by the examples considered in this section, it follows that $G_*$ is unique. Due to the model misspecification, in general $G_* \neq G_0$. The best we can hope for is that the posterior distribution of the mixing measure $G$ contracts toward $G_*$ as $n$ tends to infinity.
\comment{
The proof of Lemma~\ref{lemma:misspecified_optimal_unique} is 
deferred to Section~\ref{subsection:proof_lemma_misspecified_optimal} 
in Appendix A. A consequence of Lemma~\ref{lemma:misspecified_optimal_unique} 
leads to the partition of $\Mcal^{*}$ into the union of 
$\Mcal_{k}= \left\{G_{*} \in \Mcal^{*}: 
\ G_{*} \ \text{has} \ k \ \text{atoms}\right\}$. 
We denote $k_{*} : = k_{*}(\Mcal^{*})$ the minimum number $k \in [1,\infty]$ 
such that $\Mcal_{k}$ is non-empty. 
If $k_{*} < \infty$, it implies that $\Mcal_{k_{*}}$ 
will have exactly one element $G_{*}$ as long as $f$ is identifiable. 
Additionally, $\Mcal_{k}$ is non-empty for all $k_{*}<k<\infty$ 
while $\Mcal_{\infty}$ may contain various elements. When $k_{*} = \infty$, 
it is clear that $\Mcal_{k} = \emptyset$ for all $k<\infty$. 
The set $\Mcal_{k_{*}}$ may contain various elements $G_{*}$ 
with infinite number of elements. For the simplicity of the argument later, 
we will consider only location family of densities $f$ under the 
infinite setting of $k_{*}$. 
Under that assumption, $\Mcal_{k_{*}}$ will have a unique element $G_{*}$. 
}
The goal of the remaining of this section is to study the posterior contraction behavior of the (misspecified) mixing measure $G$ towards the unique $G_*$.

Following the theoretical framework of~\cite{Kleijn-2006} and~\cite{Nguyen-13}, the posterior contraction behavior of the mixing measure $G$ can be obtained by studying the relationship of a weighted version of Hellinger distance and corresponding Wasserstein distances between $G$ and the limiting point $G_*$. In particular, for a fixed pair of mixture densities $p_{G_0,f_0}$ and $p_{G_*}$, the weighted Hellinger $\overline{h}$ between two mixture densities is defined as follows~\cite{Kleijn-2006}.
\begin{definition} \label{definition:weighted_Hellinger}
For $G_{1}, G_{2} \in \Pcal(\Theta)$, 
\begin{eqnarray}
\hba^{2}(p_{G_{1}},p_{G_{2}}) 
		: = \dfrac{1}{2} \int \left(\sqrt{p_{G_{1}}(x)}-\sqrt{p_{G_{2}}(x)}
		\right)^{2}\dfrac{p_{G_{0},f_0}(x)}{p_{G_{*}}(x)}dx. \nonumber
\end{eqnarray}
\end{definition}
It is clear that when $G_{*}=G_{0}$ and $f=f_0$, the weighted Hellinger distance reduces to the standard Hellinger distance. In general they are different due to misspecification.
According to Lemma~\ref{lemma:misspecified_optimal}, we have 
$\hba(p_{G_{1}},p_{G_{2}}) \leq 1$ for all $G_{1}, G_{2} \in \Pcal(\Theta)$. 

\paragraph{Choices of prior on mixing measures}
As in the previous section, we work with two representative priors on the mixing measure: the MFM prior and the Dirichlet process prior. Both prior choices may contribute to the model misspecification, if the true mixing measure $G_0$ lies outside of the support of the prior distribution. 

Recall the MFM prior specification given in Eq.~\eqref{eq:MFM_model}. We also need a stronger condition on $q_{K}$:
\begin{itemize}
\item[(P.4')] The prior 
distribution $q_{K}$ on the number of components satisfies
$q_k \gtrsim k^{-\alpha_0}$ for some $\alpha_0 >1$.
\end{itemize}
Note that the assumption with prior on the number of components $q_{K}$ is mild and satisfied by many distributions, such as Poisson distribution. 
In order to obtain posterior contraction rates, one needs to make sure the prior places sufficient mass on the (unknown) limiting point of interest. For the MFM prior, such a condition is guaranteed by the following lemma.

\begin{lemma}
\label{lemma:Prior_mass_MFM}
Let $\Pi$ denote the prior for generating $G$ based on MFM~\eqref{eq:MFM_model}, where $H$ admits condition (P.2) and $q_K$ admits (P.4'). Fix $r\geq 1$. Then the following holds, for any $G_* \in \Pcal(\Theta)$
\begin{eqnarray}
\label{eq:Prior_mass_MFM}
& & \Pi \left( W_r^r(G,G_*) \leq (2^r +1)\epsilon^r \right) \nonumber \\
		& & \hspace{ 6 em} \gtrsim  \frac{\gamma \Gamma(\gamma) D!q_D}{D} 
		\left(c_0 \left(\frac{\epsilon}
		{\text{Diam}(\Theta)}\right)^{d} \right)^{D} \left( \frac{1}{D} 
		\left(\frac{\epsilon}{\text{Diam}(\Theta)}\right)^r
		\right)^{\gamma(D-1)/D}
\end{eqnarray}
for all $\epsilon$ sufficiently small so that $D(\epsilon, \Theta, \|.\|)>\gamma$. Here, $D = D(\epsilon, \Theta, \|.\|)$ and $q_D$ stand for the maximal $\epsilon$-packing number for $\Theta$ under $\|.\|$ norm and the prior weight $\Pi(K = D)$, respectively. 
\end{lemma}

The proof of Lemma~\ref{lemma:Prior_mass_MFM} is provided in Section~\ref{subsection:proof_lemma_prior_mass_MFM}. Alternatively, for a Dirichlet process prior, $G$ is distributed a priori according to a Dirichlet measure with concentration parameter $\gamma > 0$ and base measure $H$ satisfying condition (P.2). An analogous concentration bound for such a prior is given in Lemma 5 of~\cite{Nguyen-13}.

It is somewhat interesting to note that the difference in the choices of prior under misspecification does not affect the posterior contraction bounds that we can establish. In particular, as we have seen for the definition, $G_*$ does not depend on a specific choice of prior distribution (only its support). Due to misspecification, $G_*$ may have infinite support, even if the true $G_0$ has a finite number of support points. When $G_*$ has infinite support, the posterior contraction toward $G_*$ becomes considerably slower compared to the well-specified setting. In addition to the structure of $G_*$, we will see in the sequel that the modeler's specific choice of kernel density $f$ proves to be especially impactful on the rate of posterior contraction.

\subsection{Gaussian location mixtures}
\label{subsubec:Gaussian}
Consider a class of kernel densities that belong to the supersmooth 
location family of density functions.
A particular example that we focus on in this section is a class of Gaussian 
distributions with some fixed covariance matrix $\Sigma$. More precisely, 
$f$ has the following form:
\begin{eqnarray}
\left\{f(\cdot|\theta), \theta \in \Theta \subset \mathbb{R}^d: 
		f(x|\theta) :=   \frac{\exp(-(x-\theta)^{\top}
		\Sigma^{-1}(x-\theta)/2)}{|2 \pi \Sigma|^{- 1/2}}\right\}, \label{eqn:location_Gaussian_formulation}
\end{eqnarray}
where $|\cdot|$ stands for matrix determinant. Note that, Gaussian kernel is perhaps the most popular choice in mixture modeling. 

With the Gaussian 
location kernel, it is possible to obtain a lower bound on the Hellinger 
distance between the mixture densities in terms of the Wasserstein distance 
between corresponding mixing measures~\cite{Nguyen-13}. 
More useful in the misspecified setting is a key lower bound for the weighted Hellinger distance in terms of the Wasserstein metric, which is given as follows. We shall require a technical condition relating $f$ to the true $f_0$ and $G_0$:

\begin{itemize}
\item[(P.5)] The support of $G_0$, namely, $\textrm{supp} (G_0)$ is a bounded subset of $\mathbb{R}^d$. Moreover, there are some constants $C_0,C_1, \alpha >0$ such that for any $R > 0$,
\[\sup_{x \in \mathbb{R}^d, \theta \in \Theta, \theta_0 \in \textrm{supp}(G_0)} \frac{f(x|\theta)}{f_0(x|\theta_0)} \mathbbm{1}_{\|x\|_2 \leq R} \leq C_1 \exp (C_0 R^\alpha).\]
\end{itemize}

\begin{proposition} 
\label{proposition:lower_bound_weighted_Hellinger_Wasserstein_infinite_Gaussian}
Let $f$ be a Gaussian kernel given by~\eqref{eqn:location_Gaussian_formulation}, $\Theta$ a bounded subset of $\mathbb{R}^d$. Moreover, assume that $f$ satisfies condition (P.5) for $\alpha\leq2$. 
Then, there exists $\epsilon_0 > 0$ depending on $\Theta$ and $\Sigma$, such that for any $G,G' \in \Pcal(\Theta)$, whenever $\hba(p_{G},p_{G'}) \leq \epsilon_0$, the following inequality holds
\begin{eqnarray}
\hba(p_{G},p_{G'}) 
		\geq C \exp\biggr(-(1+  8\lambda_{\max}(\lambda_{\min}^{-1}+C_0))/
		W_{2}^{2}(G,G')\biggr). \nonumber
\end{eqnarray} 
Here, $\lambda_{\max}$ and $\lambda_{\min}$ are respectively the maximum and minimum eigenvalue of $\Sigma$. $C$ is a constant depending on the parameter space $\Theta$, the dimension $d$, the covariance matrix $\Sigma$, $G_0$ and $C_1$ in condition (P.5). 
\end{proposition}
The proof of  Proposition~\ref{proposition:lower_bound_weighted_Hellinger_Wasserstein_infinite_Gaussian} 
is provided in Section~\ref{subsection:proof_proposition_lower_bound_hellinger_Wasserstein _infinite_Gaussian}. 
\comment{
Note that, Proposition~\ref{proposition:lower_bound_weighted_Hellinger_Wasserstein_infinite_Gaussian} 
holds generally for absolutely continuous 
probability measures $G$ and $G'$ that have 
their support on a bounded set $\Theta \subset \mathbb{R}^d$ 
and are sufficiently close in weighted 
Hellinger distance. For our purpose, however, 
we make use of the special case of $G,G' \in \mathcal{G}(\Theta)$ 
to prove the following posterior convergence 
rate of $G$ when $f$ is a location family of 
Gaussian mixture distributions. }
We are ready to prove the first main result of this section.

\begin{theorem} \label{theorem:posterior_rate_Gaussian}
Assume that $f$ satisfies condition specified in Prop.~\ref{proposition:lower_bound_weighted_Hellinger_Wasserstein_infinite_Gaussian}, and
$\Pi$ is an MFM prior on $\Pcal(\Theta)$ specified in Lemma~\ref{lemma:Prior_mass_MFM}. Then, as $n$ tends to infinity,
\begin{eqnarray}
\Pi\biggr(G \in \overline{\mathcal{G}}(\Theta) : W_2(G,G_*) 
		\lesssim \left(\frac{\log \log n}{\log n}\right)^{1/2} \biggr| 
		X_1, \ldots, X_n\biggr) \to 1 \nonumber
\end{eqnarray}
in $p_{G_0,f_0}$-probability.
\end{theorem}
The proof of Theorem~\ref{theorem:posterior_rate_Gaussian} is given in Section~\ref{ssub:proof_theorem_posterior_rate_Gaussian}. The same posterior contraction behavior holds if we replace MFM prior by the Dirichlet Process prior with no change in the proof, except that Lemma 5 of~\cite{Nguyen-13} is used in place of Lemma~\ref{lemma:Prior_mass_MFM}. 

\subsection{Laplace location mixtures}
\label{subsubec:Laplace}
Next, we consider a class of multivariate Laplace kernel, a representative in the family of ordinary smooth density functions. It was shown by~\cite{Nguyen-13} that under a Dirichlet process location mixture with a Laplace kernel, assume the model is well-specified, the posterior contraction rate of mixing measures to $G_0$ is of order $n^{-\gamma}$ for some constant $\gamma > 0$. Under the current misspecification setting, we will be able to derive contraction rates toward $G_*$ in 
the order of $n^{-\gamma'}$ for some constant $\gamma'$ dependent on $\gamma$.
The density of location Laplace distributions is given by :
\begin{eqnarray}
\label{eqn:location_Laplace_formulation}
f(x| \theta) 
		= \frac{2}{\lambda(2\pi)^{d/2}} \dfrac{K_{(d/2)-1} 
		\left(\sqrt{2/\lambda}\sqrt{ (x-\theta)^{\top} \Sigma^{-1} 
		(x-\theta)} \right)}{\left(\sqrt{\lambda/2}\sqrt{(x-\theta)^{\top}
		 \Sigma^{-1}(x-\theta)}\right)^{(d/2)-1}},
\end{eqnarray}
where $\Sigma$ and $\lambda>0$ are respectively fixed covariance matrix 
and scale parameter such that $|\Sigma | = 1$. Here, $K_{v}$ is a Bessel 
function of the second kind of order $v$. As discussed in~\cite{Laplace-Eltoft-2006}, 
$K_m(x) \sim \sqrt{\frac{\pi}{2x}} \exp(-x)$ as $|x| \rightarrow \infty$. 
Therefore, there exists $\tilde{R}$ such that as long as 
$\|x-\mu\| > \tilde{R}$,  we have
\begin{eqnarray}
f(x| \theta) 
		\asymp \dfrac{\exp\left(- \sqrt{\frac{2}{\lambda}} \|x-\theta\|
		_{\Sigma^{-1}}\right)}{(\|x-\theta\|_{\Sigma^{-1}})^{(d - 1)/2}}, \nonumber
\end{eqnarray}
where we use the shorthand notation 
$\|y\|_{\Sigma^{-1}}= \sqrt{y^{\top}\Sigma^{-1}y}$. To ease the ensuing presentation, we denote 
\begin{align*}
\tau(\alpha) : = \frac{\sqrt{2/(\lambda\lambda_{\max})}}{\left(\sqrt{2/(\lambda\lambda_{\min})}+ 
\sqrt{2/(\lambda\lambda_{\max})} +C_0\right )^{1/\alpha}}.
\end{align*}
The following proposition provides a key lower bound of weighted 
Hellinger distance in terms of the Wasserstein metric.

\begin{proposition}
\label{proposition:lower_bound_weighted_Hellinger_Wasserstein_infinite_Laplace}
Let $f$ be a Laplace kernel given by~\eqref{eqn:location_Laplace_formulation} 
for fixed $\Sigma$ and $\lambda$ such that $|\Sigma | = 1$. Moreover, $f$ satisfies condition (P.5) for some $\alpha \geq 1$.
Then, there exists $\epsilon_0>0$ depending on $\Theta$, $\lambda$ 
and $\Sigma$, such that for any $G,G' \in \Pcal(\Theta)$, whenever $\hba(p_{G},p_{G'}) \leq \epsilon_0$, 
the following inequality holds
\comment{
\begin{eqnarray}
 \exp\left( - \tau(\alpha)\log\left( \frac{1}{\hba(p_G,p_{G'})}\right)^{1/\alpha}\right) \geq C \frac{W_2^{2/m}(G,G')}{{\log(1/W_2(G,G'))^{d\alpha/2}}}
		\nonumber
\end{eqnarray}
}
\begin{eqnarray*}
\left(\log \frac{1}{\hba(p_G,p_{G'})} 
		\right)^{d/(2\alpha)} \exp\left( - \tau(\alpha)\left( \log \frac{1}{\hba(p_G,p_{G'})}\right)^{1/\alpha}\right) \geq  C {W_2^{2/m}(G,G')}  .
\end{eqnarray*}
for any positive constant $m < 4/ (4 + 5d)$. Here, $\lambda_{\max}$ and 
$\lambda_{\min}$ are respectively the maximum and minimum 
eigenvalue of $\Sigma$. The constant $C$ depends on the parameter space $\Theta$, the dimension $d$, the covariance matrix $\Sigma$, the scale parameter $\lambda$, $G_0$ and $C_1$ in (P.5).
\end{proposition}

\comment{Here is what I have, for $\alpha > 0$
\begin{eqnarray*}
\left(\log \frac{1}{\hba(p_G,p_{G'})} 
		\right)^{d/(2\alpha)} \exp\left( - \tau(\alpha)\left( \log \frac{1}{\hba(p_G,p_{G'})}\right)^{1/\alpha}\right) \gtrsim  {W_2^{2/m}(G,G')} .
\end{eqnarray*}}

The proof of Proposition~\ref{proposition:lower_bound_weighted_Hellinger_Wasserstein_infinite_Laplace} 
is provided in Section~\ref{subsection:proof_proposition_lower_bound_hellinger_Wasserstein _infinite_Laplace}. Given the above result, the posterior contraction rate for mixing 
measures $G$ in the location family of Laplace mixture 
distributions can be obtained from the following result:
\begin{theorem} \label{theorem:posterior_rate_Laplace}
Assume that $f$ is given by equation~\eqref{eqn:location_Laplace_formulation} 
for fixed $\Sigma$ and $\lambda$ such that $|\Sigma | = 1$. Additionally, assume that $f$ satisfies condition specified in Prop.~\ref{proposition:lower_bound_weighted_Hellinger_Wasserstein_infinite_Laplace}, and $\Pi$ an MFM prior on $\Pcal(\Theta)$ specified in Lemma~\ref{lemma:Prior_mass_MFM}. Then, as $n$ tends to infinity,
\begin{eqnarray}
\Pi\biggr(G \in \overline{\mathcal{G}}(\Theta) : W_2(G,G_*) 
		\lesssim \exp\left( - \frac{m \tau(\alpha)}{2} \biggr (\frac{\log n-\log\log n}{2(d+2)} \biggr )^{1/\alpha}\right)
		\biggr| X_1, \ldots, X_n\biggr) \to 1 \nonumber
\end{eqnarray}
in $p_{G_0,f_0}$-probability for any positive constant $m < 4/(4+5d)$.
\end{theorem}

The proof of Theorem~\ref{theorem:posterior_rate_Laplace} is 
straightforward using the result in Proposition~\ref{proposition:lower_bound_weighted_Hellinger_Wasserstein_infinite_Laplace} 
and analogous to the proof argument of Theorem \ref{theorem:posterior_rate_Gaussian};  
therefore, it is omitted. Note that, identical to the Gaussian kernel case, a similar contraction behavior also holds for the Laplace kernel with the Dirichlet Process Prior. The proof can be obtained similar to the MFM prior by invoking Lemma 5 of~\cite{Nguyen-13} instead of Lemma~\ref{lemma:Prior_mass_MFM}.

\paragraph{Remarks} (i) It is worth noting that compared to the well-specified setting, the posterior contraction upper bound obtained for Gaussian location mixtures remains the same slow logarithmic rate $(\log\log n/\log n)^{1/2}$. For Laplace mixtures, when the truth $f_0$ satisfies condition (P.5) with  $\alpha \leq 1$, the posterior contraction upper bound obtained under misspecification remains a polynomial rate of the form $n^{-\gamma'}$ modulo a logarithmic term. Due to misspecification there is a loss of a constant factor in the exponent $\gamma'$, which is dependent on the shape of the kernel density as it is captured by the term $\tau(\alpha)$.  


(ii) Although Gaussian mixtures have proved to be an asymptotically optimal density estimation device under suitable and mild conditions (cf.~\cite{Ghosal-vanderVaart-07b}), the results obtained in this section suggest that it is not a suitable choice for mixture modeling under model misspecification, even if the true $G_0$ has finite number of support points, if the primary interest is in the quality of model parameter estimates. Mixtures of heavy-tailed and ordinary smooth kernel densities such as the Laplace prove to be more amenable to efficient parameter estimation. Thus, the modeler is advised to select for $f$, say, a Laplace kernel over a supersmooth kernel such as Gaussian kernel, provided that condition (P.5) is valid.

(iii) It is interesting to consider the scenario where the true kernel $f_0$ happens to be a Gaussian kernel: if we use the either a well-specified or a misspecified Gaussian kernel to fit the data, the posterior contraction bound is the extremely slow $(\log \log n/\log n)^{1/2}$ accordingly to Theorem~\ref{theorem:posterior_rate_Gaussian}. This rate may be too slow to be practical. If the statistician is too impatient get to the truth $G_0$, because sample size $n$ is not sufficiently large, he may well decide to select a Laplace kernel $f$ instead. Despite the intentional misspecification, he might be comforted by the fact that the posterior distribution of $G$ contracts at an exponentially faster rate to a $G_*$ given by Theorem~\ref{theorem:posterior_rate_Laplace} for $\alpha=2$. 

\subsection{When $G_*$ has finite support}  
\label{ssection:finite k}
The source of the deterioriation in the statistical efficiency of parameter estimation under model mispecification is ultimately due to the increased complexity of the limiting point $G_*$. Even if the true $G_0$ has a finite number of support points, this is not the case for $G_*$ in general. Unfortunately, it is very difficult to gain concrete information about $G_*$ both in practice and in theory, due to the lack of knowledge about the true $p_{G_0,f_0}$. When some precious information about $G_*$ is available, specifically, suppose that we happen to know $G_*$ has a bounded number of support points $k_*$ such that $k_* < \overline{k}$ for some known $\overline{k}$.
Then it is possible to devise a new prior specification on the mixing measure $G$ so that one can gain a considerably improved posterior contraction rate toward $G_*$. We will show that it is possible to obtain the contraction rate of the order $(\log n/n)^{1/4}$ under $W_2$ metric --- this is the same rate of posterior contraction one would get with overfitted mixtures in the well-specified regime.

In order to analyze the convergence rate of mixing measure 
under that setting of $k_{*}$, we introduce a relevant notion of 
integral Lipschitz property, which is a generalized form of the uniform 
Lipschitz property for the misspecification scenarios.

\begin{definition} \label{definition: integral_Lipschitz_first_order}
For any given $r \geq 1$, we say that the family of densities $f$ admits the 
\emph{ integral Lipschitz} property up to the order $r$ with 
respect to two mixing measures $G_0$ and $G_*$ , if $f$ as a function 
of $\theta$ is differentiable up to the order $r$ and its partial derivatives 
with respect to $\theta$ satisfy the following inequality
\begin{eqnarray}
\sum_{|\kappa| = r} \biggr |\biggr(\frac{\partial^{|\kappa|} f}{\partial \theta^\kappa}
		(x|\theta_1) - \frac{\partial^{|\kappa|} f}{\partial \theta^\kappa}
		(x|\theta_2) \biggr ) \gamma^\kappa\biggr |
		\leq C(x) \|\theta_1-\theta_2\|^\delta \|\gamma\|^{r} \nonumber
\end{eqnarray}
for any $\gamma \in \mathbb{R}^{d}$ and for some positive constants 
$\delta$ independent of $x$ and $\theta_{1}, \theta_{2} \in \Theta$. Here, $C(x)$ is some function such that ${\displaystyle \int C(x)\dfrac{p_{G_0,f_0}(x)}{p_{G_*}(x)} \mathrm{d}x} < \infty$.
\end{definition}
It is clear that when $f$ has  integral Lipschitz property up to the 
order $r$, for some $r \geq 1$, with respect to $G_{0}$ and $G_{*}$, 
then it will admit uniform Lipschitz property up to the order $r$. We can 
verify that the first order intergral Lipschitz property is satisfied 
by many popular kernels, including location-scale Gaussian distribution 
and location-scale Cauchy distribution. 

In the following we shall work with the MFM prior~\eqref{eq:MFM_model}. Moreover,
\begin{itemize}
    \item[(M.0)] $q_K$ places positive masses on 
$K \in \left\{1,\ldots,\overline{k}\right\}$ and $0$ mass 
elsewhere, where $\overline{k} \gg k_{*}$ is a fixed number. 
\end{itemize} 
Given that $k_{*}$ is finite, 
we obtain a key lower bound of weighted Hellinger 
distance in terms of the Wasserstein metric under strong identiability of $f$:
\begin{proposition} \label{proposition:weighted_Hellinger_Wasserstein_metric}
Assume that $f$ is second order identifiable and admits uniform 
integral Lipschitz property up to the second order. Then, for any 
$G \in \Ocal_{\overline{k}}$, the following inequality holds
\begin{eqnarray}
\hba(p_{G},p_{G_{*}}) 
		\gtrsim W_{2}^2(G,G_{*}). \nonumber
\end{eqnarray}
\end{proposition}
The proof of Proposition~\ref{proposition:weighted_Hellinger_Wasserstein_metric} 
is in Section~\ref{subsection:proof_proposition_weighted_hellinger}. 
Before stating the final theorem of this section, we will 
need following assumptions:
\begin{itemize}
\item[(M.1)] The assumptions of Proposition~\ref{proposition:weighted_Hellinger_Wasserstein_metric} hold, i.e., 
$f$ is second order identifiable and admits uniform integral 
Lipschitz property up to the second order. 
\item[(M.2)] There exists $\epsilon_{0}>0$ such that 
${\displaystyle \int (p_{G_0,f_0}(x))p_{G_{*}}(x)/p_{G}(x)d\mu(x) 
\leq M^*(\epsilon_0)}$ whenever we have $W_{1}(G,G_{*}) \leq 
\epsilon_{0}$ for any $G \in \Ocal_{k_{*}}$ where 
$M^*(\epsilon_0)$ depends only on $\epsilon_{0}$, $G_{*}$, $G_{0}$, and $\Theta$.
\item[(M.3)]The parameter $\gamma$ in Dirichlet distribution in MFM 
satisfies $\gamma < \overline{k}$. Additionally, the base distribution $H$ satisfies Assumption (P.2).
\end{itemize}

\begin{theorem} \label{theorem:convergence_rate_misspecified}  
Assume $k_0<\infty$, and assumptions (M.0),(M.1),(M.2) and (M.3) hold. 
Then we have that,
\begin{eqnarray}
\Pi\biggr(G \in \overline{\mathcal{G}}(\Theta): W_{2}(G,G_{*}) 
		\lesssim (\log n/n)^{1/4} \biggr|X_{1},\ldots,X_{n}
		\biggr) \to 1 \nonumber
\end{eqnarray}
in $p_{G_0,f_0}$-probability. 
\end{theorem}
The proof of Theorem~\ref{theorem:convergence_rate_misspecified} 
is deferred to Section~\ref{subs:proof_misspecified_finite_k }. 

\paragraph{Further remarks}
The above theorem raises a promising prospect for combating model misspecification, by having the modeler to fit the data to an \emph{underfitted} 
mixture model $p_{G}$. Unfortunately, this theorem does not address this scenario, under which the limiting mixing measure would correspond to the KL minimizer
\[G_{**} = \mathop{\arg \min} \limits_{G \in \Ocal_{\overline{k}}(\Theta)}
		{K(p_{G_0,f_0}, p_{G})}.\]
		for some $\overline{k} < \infty$, provided that this quantity exists (compare this with $G_*$ given in~\eqref{eq:KL_minimum}). Due to the lack of convexity of the class of mixture densities with bounded number of mixture components, the theory developed in this section (tracing back to the work of~\cite{Kleijn-2006}) is not applicable. Thus, posterior contraction behaviors in an underfitted mixture models remain an interesting open question.
\comment{
Here, 
we put a few remark with~\ref{theorem:convergence_rate_misspecified}. 
First, Assumption (M.1) is to ensure that we have the lower bound of 
weighted Hellinger distance between in terms of second order 
Wasserstein metric, i.e., 
$\bar{h}^2(p_{G},p_{G_{*}}) \geq Cr^4$ as long as the Wasserstein metric $W_{2}(G,G_{*}) \geq r$ for any $r>0$ 
where $C$ is some positive constant depending only on $G_{0}$ and $\Theta$. 
Second, the high level idea of assumption (M.2) is to control the 
growing rate of KL neighborhood around $p_{G_{*}}$, which plays a key 
role in the analysis of posterior convergence rate of mixing measures 
under misspecified setting (see Theorem~\ref{theorem:posterior_contraction_rate_misspecified} in Appendix B). 
We can validate that this assumption holds for 
various choices of kernel $f$, such as location Gaussian and location 
Laplace families. As a consequence, the assumptions (M.1) and (M.2) are 
mild and true for many popular distributions. Assumption (M.3) is more of a technical condition, and can always hold for large enough $\overline{k}$.
}
 

\section{Simulation studies}
\label{section:simulation}
\begin{figure}[t]
\centering
\begin{minipage}{.5\textwidth}
  \centering
  \includegraphics[width=.9\linewidth]{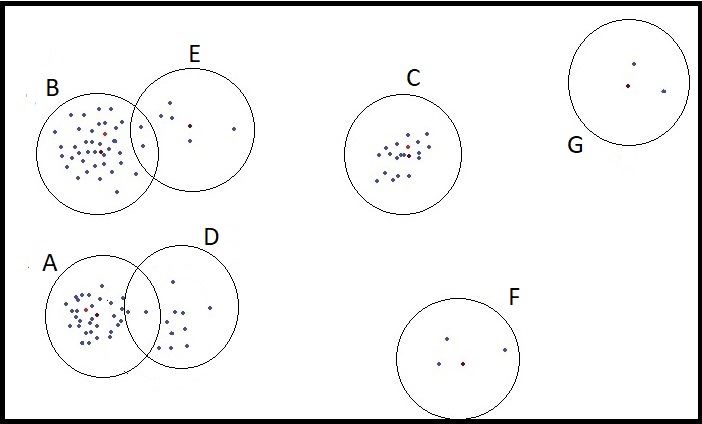}
  \captionof{figure}{Initial distribution $G$.}
  \label{fig:initial}
\end{minipage}%
\begin{minipage}{.5\textwidth}
  \centering
  \includegraphics[width=.9\linewidth]{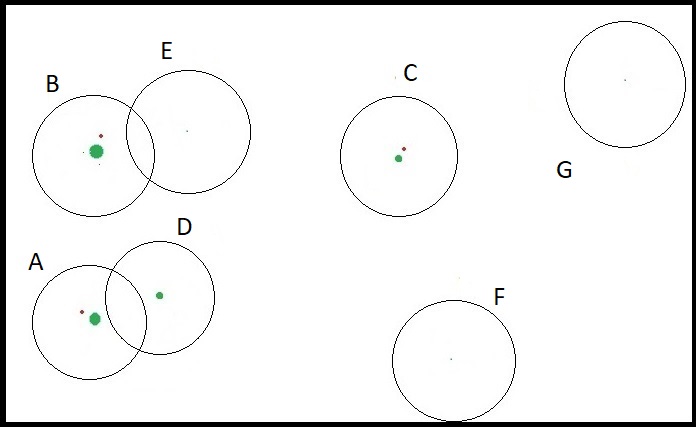}
  \captionof{figure}{After first stage-"merge".}
  \label{fig:merge_1}
\end{minipage}
\end{figure}

\begin{figure}[t]
\centering
\begin{minipage}{.5\textwidth}
  \centering
  \includegraphics[width=.9\linewidth]{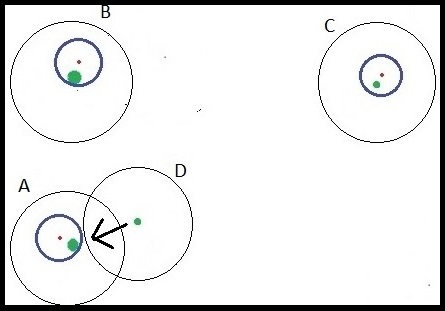}
  \captionof{figure}{After second stage-"truncation".}
  \label{fig:truncate}
\end{minipage}%
\begin{minipage}{.5\textwidth}
  \centering
  \includegraphics[width=.9\linewidth]{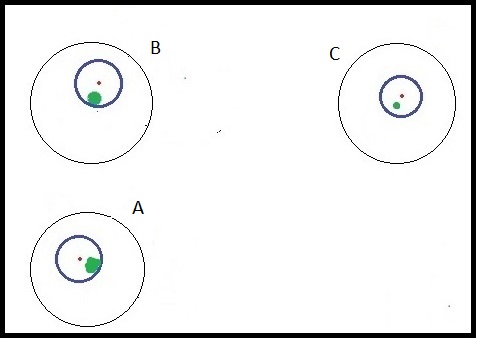}
  \captionof{figure}{After second stage-"merge".}
  \label{fig:merge_2}
\end{minipage}
\end{figure}

In this section we provide an illustration of the MTM algorithm's behavior via a simple simulation study.
Figures~\ref{fig:initial},~\ref{fig:merge_1},~\ref{fig:truncate} and~\ref{fig:merge_2} illustrate the different stages in the application of MTM algorithm~\ref{algo:merge-truncate-merge}. In each figure, green dots denote the atoms in the set of "remaining atoms" at each stage, with weights proportional to their sizes. 
Red dots denote the supporting atoms of the true mixing measure $G_0$. Black 
circles denote balls of radius $\omega_n$ around each of the "remaining atoms". Blue circles denote balls of radius $\frac{\omega_n}{4k_0}$ around 
the atoms of $G_0$. 

Starting with an input measure $G$ represented in Fig.~\ref{fig:initial}, the first stage of the algorithm (merge procedure, from line 1 to line 4) merges nearby atoms to produce $G'$, which is represented by Fig.~\ref{fig:merge_1}. 
There remains some atoms that carry very small mass, they are suitably truncated (via line 5 in the algorithm), and then merged accordingly (via line 6). Fig.~\ref{fig:truncate} and Fig.~\ref{fig:merge_2} represent the outcome after these two steps of the algorithm. Observe how the atoms in each of the blue circles are merged to produced a reasonably accurate estimate of the corresponding atom of $G_0$. The number of such circles gives the correct number of the supporting atoms of $G_0$.

Next, we illustrate the performance of the MTM algorithm as it is applied to the samples from a Dirichlet process mixture, given the data generated by mixtures of three location Gaussian distributions:
\begin{eqnarray}
p_{G_0}(\cdot) = \sum_{i=1}^3 p_i^0 \mathcal{N}(\cdot|\mu_i^0,\Sigma^0) \nonumber
\end{eqnarray}
where $\mathcal{N}(\cdot| \mu, \Sigma)$ is the Gaussian distribution 
with mean vector $\mu$ and covariance matrix $\Sigma$. For simulation purposes, we consider the following four different settings ($n$ is the sample size):
\begin{enumerate}
    \item Case A: $\mu_1^0=(0.8,0.8)$, $\mu_2^0=(0.8,-0.8)$, $\mu_3^0=(-0.8,0.8)$, $\Sigma^0= 0.05 I_3$, $n=500$. 
    \item Case B:  $\mu_1^0=(0.8,0.8)$, $\mu_2^0=(0.8,-0.8)$, $\mu_3^0=(-0.8,0.8)$, $\Sigma^0= 0.05 I_3$, $n=1500$.
    \item Case C: $\mu_1^0=(1.8,1.8)$, $\mu_2^0=(1.8,-1.8)$, $\mu_3^0=(-1.8,1.8)$, $\Sigma^0= 0.05 I_3$, $n=500$. 
    \item Case D: $\mu_1^0=(0.8,0.8)$, $\mu_2^0=(0.8,-0.8)$, $\mu_3^0=(-0.8,0.8)$, $\Sigma^0= 0.01 I_3$, $n=1500$. 
\end{enumerate}
Here, $I_3$ is the identity 
matrix of dimension $3$. 
Additionally, the weight vector for all these cases is chosen as $p^0=(p_1^0,p_2^0,p_3^0) = (0.4,0.3,0.3)$.

As mentioned above, a Dirichlet process prior with an uniform prior base measure $H$ in the region $[-6,6] \times[-6,6]$, along with concentration parameter $\alpha=1$. This choice of prior enables us to sample significantly larger numbers of components of the mixing measure than the true number of three components.

It is known that the contraction rate of mixing measures under location Gaussian DPMM is 
$\tilde{C}(\log(n)^{-1/2})$ with respect to the Wasserstein-$2$ norm, for some constant $\tilde{C}$ which depends on $\Sigma^0$(the covariance matrix), the location parameters $\mu_i^0$ and the weights $p_i^0$~\cite{Nguyen-13}. For our purpose, in order for $\omega_n$ to satisfy Equation~\eqref{eq:posterior_results}, 
we may choose any $\omega_n$ as long as $\frac{\omega_n} {\log(n)^{-1/2}} \to \infty$. We selected $\omega_n= 
\biggr(\frac{\log(\log(n))}{(\log(n))}\biggr) ^{1/2}$ for all our applications of the MTM algorithm.

The MTM algorithm is provably consistent (in the asymptotic sense) for all chosen constants $c> 0$. In practice for $n$ being fixed, the input $c$ to Algorithm~\ref{algo:merge-truncate-merge} should be chosen so that $\frac{\tilde{C}}{(\log(\log(n)))^{1/2}}\leq c$. Moreover, for finite $n$ it is not expected that the posterior probability for $k=k_0$ is close to 1. However, for identifying the number of components the posterior mode provides a reasonable estimate. In particular, $(1-\sum_{i=1}^3 \frac{c}{p_i^0})$ forms a useful lower bound on the posterior mass at the mode as identified in Equation~\eqref{eq:posterior mode mass}. To identify $k=k_0$ consistently using the posterior mode safely, one needs to choose $c< c_0$, with $c_0$ satisfying $(1-\sum_{i=1}^3 \frac{c_0}{p_i^0}) > 1/2$. The exact computation of the upper bound $c_0$ and the lower bound $\frac{\tilde{C}}{(\log(\log(n)))}$ for $c$ may be unrealistic but a reasonable estimate may be possible. Nonetheless, we simply considered a large range of $c$ and show there is a range where we can robustly identify the true number of components via the posterior mode. 

For the DP mixture's posterior computation, we make use of the non-conjugate split-merge sampler of Jain 
and Neal~\cite{Jain-Neal-07} with $(5,1,1,5)$ scheme, i.e., $5$ scans 
to reach the split launch state, $1$ split-merge move per iteration, 
$1$ Gibbs scan per iteration, and $5$ moves to reach the merge launch 
state. We run our experiments for two settings corresponding to sample sizes $500$ and $1500$.
The sampler had 2000 burn-in iterations followed by 18000 sample 
iterations (a total 20000), with each 10th iteration being counted. 
 
\begin{figure}[t]
\centering
\begin{minipage}{.45\textwidth}
  \centering
  \includegraphics[width=\linewidth,height=.9\linewidth]{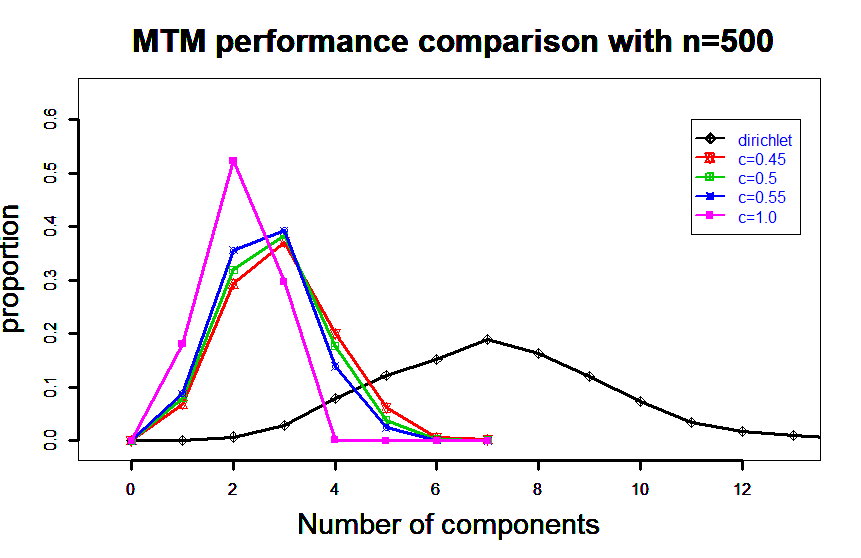}
  \captionof{figure}{Case A.}
  \label{fig:MTM-500}
\end{minipage}%
\begin{minipage}{.45\textwidth}
  \centering
  \includegraphics[width=\linewidth,height=.9\linewidth]{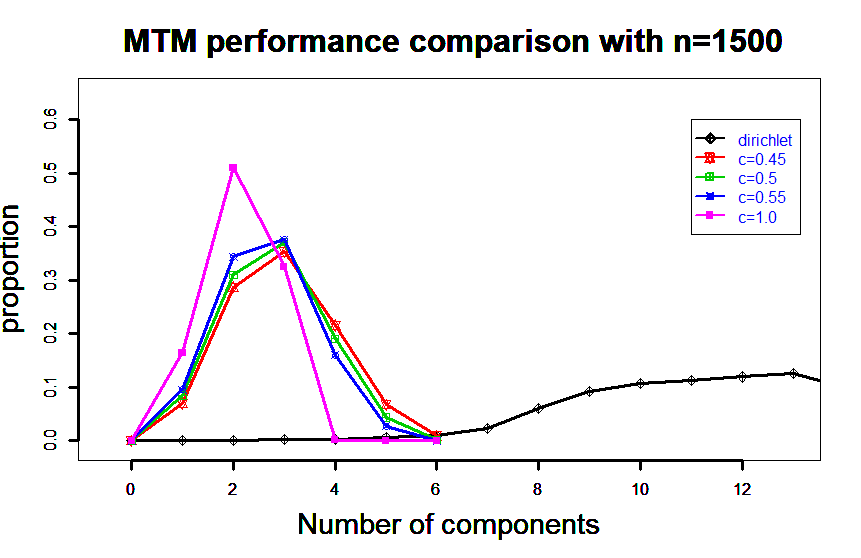}
  \captionof{figure}{Case B.}
  \label{fig:MTM-1500}
\end{minipage}
\end{figure}
 \begin{figure}[t]
\centering
\begin{minipage}{.45\textwidth}
  \centering
  \includegraphics[width=\linewidth,height=.9\linewidth]{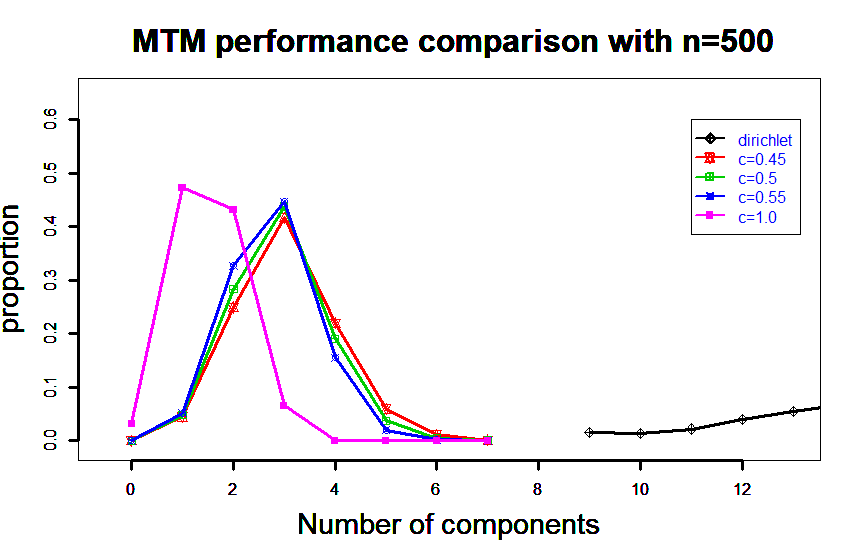}
  \captionof{figure}{Case C.}
  \label{fig:MTM-500-spaced}
\end{minipage}%
\begin{minipage}{.45\textwidth}
  \centering
  \includegraphics[width=\linewidth,height=.9\linewidth]{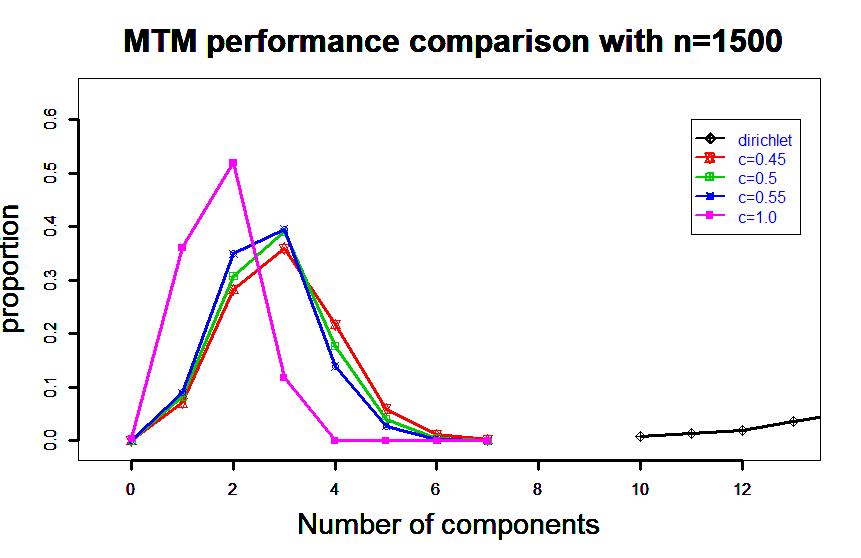}
  \captionof{figure}{Case D.}
  \label{fig:MTM-1500-small-variance}
\end{minipage}
\end{figure}

The experiments run for DP mixture-based sampler, followed by application of the MTM 
procedure for 4 different values of the tuning parameter $c$ in  
Algorithm~\ref{algo:merge-truncate-merge}, namely, for $c=0.45, 0.5, 0.55, 1.0$. 
The proportional frequencies are plotted in Figure~\ref{fig:MTM-500} 
and Figure~\ref{fig:MTM-1500} respectively, along with the 
proportional frequencies for DP mixture. The uniform base measure for the 
Dirichlet Process prior is chosen so as to enable easier creation of 
newer components in the split-merge scheme. As a consequence the DP mixture's posterior yields quite bad results as far as the number of mixture components is concerned. However, even under that case, we can 
recover the true number of components by considering the mode of the 
frequency distribution after an application of the MTM algorithm on the posterior samples, with 
appropriate constant $c$. It is expected, however, that a large choice 
of $c$ would underestimate the number of components. This is also
what is observed from the simulations, where the procedure breaks down when $c=1.0$. 

We perform the experiments under four different settings of data populations. In particular, figure~\ref{fig:MTM-500-spaced} consists of data generated from mixture of Gaussians with more widely spread location parameter values. In this case, it is expected that the convergence to the true number of components via Algorithm~\ref{algo:merge-truncate-merge} will be faster for the posterior mode, in comparison to the situation where the location parameters are closer together. This is indeed what is observed in our simulations. The value of the covariance matrix $\Sigma^0$, on the other hand does not seem to noticeably affect the results. This is again expected, since the prior support $[-6,6]\times[-6,6]$ is quite large in comparison to the eigenvalues of the covariance matrix chosen.




\section{Proofs of key results}
\label{section:proofs}

In this section, we provide proofs for several key results in the paper.
\subsection{Proof of Theorem~\ref{theorem:convergence_rate_well_specified}}
\label{proof:theorem:convergence_rate_well_specified}
The proof of the theorem consists of two key parts. First, we recall a general framework for establishing posterior contraction of mixing measures. Then we proceed to apply this framework to analyze the specific setting of the MFM model.
\subsubsection{General framework}
To establish convergence rates of mixing measures under the setting of 
MFM, we utilize the general framework of posterior contraction of mixing 
measures under well-specified setting from~\cite{Nguyen-13}. 
To state such results formally, we will need to introduce several 
key definitions in harmony with the notations in our paper. 
Let $G$ be endowed with a prior distribution $\Pi$ on a measure space of discrete probability measures in $\overline{\Gcal}(\Theta)$. Fix $G_0 \in \Pcal(\Theta)$.
For any set $\Scal \subset \overline{ \Gcal}(\Theta)$, 
we define the Hellinger information of the $W_{1}$ metric for subset $\Scal$ by the following function
\begin{align*}
\Psi_{ \Scal}(r) 
		: = \inf \limits_{G \in \Scal: \ W_{1}(G, G_{0}) \geq r/2}
		{h^{2}(p_{G}, p_{ G_{0}})}.
\end{align*}
Note that, the choice of first order Wasserstein metric in the 
above formulation is due to the lower bound of Hellinger distance 
between mixing densities in terms of first order Wasserstein 
distance between their corresponding mixing measures 
in~\eqref{eq:lower_bound_first_order_identifiability}. 
Now, for any mixing measure $G_{1} \in \overline{ \Gcal}(\Theta)$ 
and $r >0$, we define a Wasserstein ball centered at $G_{1}$ 
under $W_{1}$ metric as follows
\begin{eqnarray}
B_{W_{1}}(G_{1},r) 
		= \left\{G \in \overline{\Gcal}(\Theta): \ W_{1}(G,G_{1}) 
		\leq r \right\}. \nonumber
\end{eqnarray}
Furthermore, for any $M>0$, we define a Kullback-Leibler neighborhood of $G_{0}$ by
\begin{eqnarray}
B_{K}(\epsilon,M) 
		= \left\{G \in \overline{\Gcal}(\Theta): \  K(p_{G_{0}},p_{G}) 
		\leq \epsilon^{2}\log \left(\frac{M}{\epsilon} \right), 
		K_{2}(p_{G_{0}},p_{G}) \leq \epsilon^{2} 
		\left(\log \left(\frac{M}{\epsilon}\right)\right)^2
		\right\}. \nonumber
\end{eqnarray}

For the proof of Theorem~\ref{theorem:convergence_rate_well_specified}, we use a straightforward extension of Theorem 4 in~\cite{Nguyen-13}, adapted to the setting in this work.
\begin{theorem}\label{theorem:posterior_contraction_Nguyen}
Fix $G_{0} \in \overline{\Gcal}$. Assume the following:
\begin{itemize}
\item[(a)] The family of likelihood functions is finitely identifiable and satisfies 
$h(f(x|\theta_{i}),f(x|\theta_{j}')) 
\leq C_{1}\|\theta_{i}-\theta_{j}'\|^{\alpha}$ 
for any $\theta_{i}, \theta_{j}' \in \Theta$, 
for some constants $C_{1}>0$, $\alpha \geq 1$.
\item[(b)] There is a sequence of sets 
$\Gcal_{n} \subset \overline{\Gcal}$ for which 
\begin{eqnarray}
M(\Gcal_{n}, G, r) 
	= D \biggr(\dfrac{\Psi_{\Gcal_{n}}(r)^{1/2}}{2 \text{Diam}
	(\Theta)^{\alpha-1}\sqrt{C_{1}}}, \Gcal_{n} 
	\cap B_{W_{1}}(G,r/2), W_{1}\biggr). \nonumber
\end{eqnarray} 
\item[(c)] There is a sequence $\epsilon_{n} \to 0$ such that 
$n \epsilon_{n}^{2}$ is bounded away from 0 or 
tending to infinity, a constant $M>0$ sufficiently large, 
and a sequence $M_{n}$ such that
\begin{eqnarray}
&& \log D(\epsilon/2, \Gcal_{n} 
		\cap B_{W_{1}}(G_{0},2\epsilon) 
		\backslash B_{W_{1}}(G_{0},\epsilon), W_{1}) 
		+ 
		\sup \limits_{G \in \Gcal_{n}} 
		\log M(\Gcal_{n},G,r) \leq n\epsilon_{n}^{2}, 
		\ \forall \ \epsilon \geq \epsilon_{n}, 
		\label{eq:well_posterior_contraction_rate_first_condition}  \\
&& \dfrac{\Pi(\overline{\Gcal} 
		\backslash \Gcal_{n})}{\Pi(B_{K}(\epsilon_{n},M))} 
		= o\left(\exp \left(-2n\epsilon_{n}^{2} 
		\log \left(\frac{M}{\epsilon_n}\right) \right)\right)	
	\label{eq:well_posterior_contraction_rate_second_condition}  \\
&& \dfrac{\Pi(B_{W_{1}}(G_{0},2j\epsilon_{n}) 
		\backslash B_{W_{1}}(G_{0},j\epsilon_{n}))}
		{\Pi(B_{K}(\epsilon_{n},M))} \leq \exp\biggr(n
		\Psiba_{\Gcal_{n}}(j\epsilon_{n})/16\biggr), 
		\ \forall j \geq M_{n} \label{eq:well_posterior_contraction_rate_third_condition} \\
&& \exp\left(2n\epsilon_{n}^{2}\log 
		\left(\frac{M}{\epsilon_n}\right)\right)
		\sum \limits_{j \geq M_{n}}{\exp\left(-n\Psiba_{\Gcal_{n}}
		(j\epsilon_{n})/16\right)} \to 0. 
		\label{eq:well_posterior_contraction_rate_fourth_condition}
\end{eqnarray}
\end{itemize}
Then, we have that 
$\Pi(G \in \overline{\Gcal}: \ W_{1}(G,G_{0}) 
\geq M_{n}\epsilon_{n} | X_{1},\ldots,X_{n}) \to 0$ 
in $P_{G_{0}}$-probability.
\end{theorem}
\subsubsection{Posterior contraction under MFM}
Now, we apply the above result to establish the convergence rate of mixing measure under a well-specified MFM model. The constant $M$ for part (c) of Theorem~\ref{theorem:posterior_contraction_Nguyen} is chosen later. Also, let $\epsilon_{n} := \overline{M}(\log n/ n)^{1/2}$ where $\overline{M}$ 
is a sufficiently large constant that will be chosen later. Note that it is enough to show,
$\Pi\biggr(G \in\mathcal{G}(\Theta): W_{1}(G,G_{0}) 
		\gtrsim \dfrac{(\log n)^{1/2}}{n^{1/2}}\biggr|
		X_{1},\ldots,X_{n}\biggr) \to 0$, since $\Pi(G \in \overline{\mathcal{G}}(\Theta)\setminus\mathcal{G}(\Theta) |
		X_{1},\ldots,X_{n}) = 0$.

With $\epsilon_n$ chosen as above, we denote
$A_{n} := \Pi(G \in \mathcal{G}(\Theta) : W_1(G,G_0)\gtrsim \epsilon_{n}| 
X_{1},\ldots,X_{n})$. It is clear that
\begin{eqnarray}
A_{n} 
		& = & \sum \limits_{k=1}^{\infty}\Pi(G \in \Ocal_{k}(\Theta) : W_1(G,G_0) 
		\gtrsim \epsilon_{n}| X_{1},\ldots,X_{n})\Pi(K=k|X_{1},\ldots,X_{n}) \nonumber \\
		& \leq & \Pi(G \in \Ocal_{k_{0}}(\Theta): W_1(G,G_0) 
		\gtrsim \epsilon_{n}| X_{1},\ldots,X_{n})
		+ \Pi(K \neq k_{0}|X_{1},\ldots,X_{n}). \nonumber
\end{eqnarray}
Now, we divide our proof into the following key steps
\paragraph{Step 1:} $\Pi(K=k_{0}|X_{1},\ldots,X_{n}) \to 1$ a.s. $P_{G_{0}}$. 
As the model is identifiable, this result is the direct 
application of Doob's consistency theorem~\cite{Doob-1948}. 
\paragraph{Step 2:} $P_{G_{0}}\biggr(\Pi(G \in \Ocal_{k_{0}}(\Theta): 
W_1(G,G_0) \gtrsim \epsilon_{n}| X_{1},\ldots,X_{n})\biggr) \to 0$ 
as $n \to \infty$. The proof of this result is the application of 
Theorem~\ref{theorem:posterior_contraction_Nguyen}. In fact, as we 
focus on the posterior contraction of $G_{0}$ from $G \in 
\Ocal_{k_{0}}(\Theta)$, we denote the prior on $G \in \Ocal_{k_{0}}(\Theta)$ to be 
$\Pi = H \times Q$ where $Q \overset{d}{=} \text{Dir}(\gamma/
k_{0},\ldots,\gamma/k_{0})$. Now, we claim that
\begin{eqnarray}
\epsilon^{c_{H}} 
		& \lesssim H(\|\theta_{i}-\theta_{i}^{0}\| \leq \epsilon, 
		\ i=1,\ldots,k_{0}) \lesssim \epsilon^{c_{H}} \nonumber \\
		\epsilon^{\gamma'} & \lesssim Q(|p_{i}-p_{i}^{0}| \leq 
		\epsilon, \ i=1,\ldots,k_{0})  \label{eqn:bound_prior_distributions}
\end{eqnarray}
where $c_{H}>0$ and $\gamma'>0$ are some positive constants and 
$\epsilon$ is sufficiently small. Assume that claim~\eqref{eqn:bound_prior_distributions} 
is given at the moment. To facilitate the discussion, we further 
divide Step 2 into two small steps. 
\paragraph{Step 2.1:} To obtain the bound for 
$\Pi(B_{K}(\epsilon_{n},M))$, we utilize the result from~\cite{Wong-Shen-95} 
to bound KL divergence and squared KL divergence. In particular, 
from Theorem 5 of~\cite{Wong-Shen-95}, if $p$ and $q$ are two 
densities such that $2h^2(p,q) \leq \epsilon^2$ and 
${\int p^{2}/q \leq M^2}$ then we obtain that $K(p,q) \lesssim 
\epsilon^{2}\log(M/\epsilon)$ and $K_{2}(p,q) \lesssim 
\epsilon^{2}(\log(M/\epsilon))^{2}$ where the constants in these 
bounds are universal. 

Now, since $f$ admits Lipschitz continuity up to the first order, 
we achieve that $h^{2}(p_{G},p_{G_{0}}) \leq C_{1}W_{1}(G,G_{0})$ 
for any $G \in \Ocal_{k_{0}}(\Theta)$ where $C_{1}$ is a positive constant 
depending only on $\Theta$. Now, for any $G \in \Ocal_{k_{0}}(\Theta)$ 
such that $W_{1}(G,G_{0}) \leq C\epsilon_n^2$ where $C < \epsilon_{0}$ is a sufficiently small constant to be chosen later, the previous bound implies that $h^{2}
(p_{G},p_{G_{0}}) \leq C_{1}C\epsilon_n^2$. Since $C 
\epsilon_n^2 \leq \epsilon_{0}$ for all $n$ sufficiently large, we also have that $
{\displaystyle \int (p_{G_{0}}(x))^{2}/p_{G}(x)\mathrm{d}\mu(x) \leq 
M(\epsilon_{0})}$ according to assumption (P.3). Combining all the 
previous results, we achieve that
\begin{eqnarray}
K(p_{G_{0}},p_{G}) 
		\lesssim \epsilon_{n}^2\log(\sqrt{M(\epsilon_0)}/\sqrt{C C_{1}
		}\epsilon_n) , \nonumber \\
K_{2}(p_{G_{0}},p_{G}) 
		\lesssim \epsilon_{n}^2(\log(\sqrt{M(\epsilon_0)}/\sqrt{C C_{1} 
		}\epsilon_n))^{2}  \nonumber 
\end{eqnarray}  
when $\overline{M}$ is sufficiently large. Define $M:=\sqrt{M(\epsilon_0)/C C_{1}}$. Therefore, we have
\begin{eqnarray}
\Pi(B_{K}(\epsilon_{n},M)) 
		\geq \Pi(G \in \Ocal_{k_{0}}(\Theta): W_{1}(G,G_{0}) 
		\leq C \epsilon_n^2). \nonumber
\end{eqnarray}
For any $G=\sum_{i=1}^{k_{0}}{p_{i}\delta_{\theta_{i}}}$ 
such that $\|\theta_{i}-\theta_{i}^{0}\| \leq \overline{\epsilon}$ 
and $|p_{i}-p_{i}^{0}| \leq \overline{\epsilon}/(k_{0}\text{Diam}(\Theta))$ 
for any $1 \leq i \leq k_{0}$ and sufficiently small $
\overline{\epsilon}>0$, we can check that
\begin{eqnarray}
W_{1}(G,G_{0}) 
		\leq \sum \limits_{i=1}^{k_{0}}{p_{i}^{0} \wedge 
		p_{i}\|\theta_{i}-\theta_{i}^{0}\|+\sum \limits_{i=1}
		^{k_{0}}|p_{i}-p_{i}^{0}|\text{Diam}(\Theta)} \leq 
		2\overline{\epsilon}. \nonumber
\end{eqnarray}
Hence, by choosing universal constant $C$ such that 
$C\epsilon_n^2 \leq \overline{\epsilon}$, we would have that
\begin{eqnarray}
& & \Pi(G \in \Ocal_{k_{0}}(\Theta): W_{1}(G,G_{0}) \leq C\epsilon_n^2)  \nonumber \\
& & \hspace{3 em} \geq \Pi(G \in \Ocal_{k_{0}}(\Theta):\ \|\theta_{i}-\theta_{i}^{0}\| \leq 
C \epsilon_n^2, \ |p_{i}-p_{i}^{0}| \leq  C\epsilon_n^2/(k_{0}\text{Diam}
(\Theta)), \ \forall \ 1 \leq i \leq k_{0}) \nonumber \\
& & \hspace{3 em} \gtrsim \epsilon_n^{2(c_{H}+\gamma')} \label{eq:bound_KL_well_specified}
\end{eqnarray}
where the last inequality is due to the results from 
claim~\eqref{eqn:bound_prior_distributions}. 
\paragraph{Step 2.2:} To apply the posterior contraction rate 
result of Theorem~\ref{theorem:posterior_contraction_Nguyen}, we 
choose $\Gcal_{n}=\Ocal_{k_{0}}(\Theta)$ for all $n$. Now, it is clear 
that $\Pi(\overline{\Gcal} \backslash \Gcal_{n}) = 0$. Therefore, 
condition \eqref{eq:well_posterior_contraction_rate_second_condition} 
is obviously satisfied. Additionally, by means of Lemma 4 in~\cite{Nguyen-13}, we can check that condition~\eqref{eq:well_posterior_contraction_rate_first_condition} is 
satisfied with our choice of $\epsilon_{n}$ as $\overline{M}$ is 
sufficiently large. For condition~\eqref{eq:well_posterior_contraction_rate_third_condition}, from 
the bound of KL neighborhood in~\eqref{eq:bound_KL_well_specified}, we find that
\begin{eqnarray} 
\dfrac{\Pi(B_{W_{1}}(G_{0},2j\epsilon_{n}) \backslash B_{W_{1}}
		(G_{0},j\epsilon_{n}))}
		{\Pi(B_{K}(\epsilon_{n},M))} \lesssim \epsilon_{n}^{-2(c_{H}+
		\gamma)}. \nonumber
\end{eqnarray}
Since $f$ is first order identifiable and admits uniform Lipschitz 
property up to the first order, according to~\eqref{eq:lower_bound_first_order_identifiability}, 
we obtain that $\Psi_{\Gcal_{n}}(r) \gtrsim Cr^{2}$ for any $r>0$ 
where $C$ is some positive constant that depends only on $G_{0}$ 
and $\Theta$. Therefore, we have
\begin{eqnarray}
\exp\biggr(n\Psiba_{\Gcal_{n}}(j\epsilon_{n})/16\biggr) 
		\geq \exp(nC(j\epsilon_{n})^{2})/16) \geq n^{CM_{n}^{2}/16} \nonumber
\end{eqnarray}
for any $j \geq M_{n}$. By choosing $M_n$ such that 
$M_{n}^2 \geq 32(c_{H}+\gamma)/C$, it is clear that condition~\eqref{eq:well_posterior_contraction_rate_third_condition} is 
satisfied. 

For condition~\eqref{eq:well_posterior_contraction_rate_fourth_condition}, 
combining with the above bound of $\Psi_{\Gcal_{n}}(r)$, we would 
have that
\begin{eqnarray}
\exp\left(2n\epsilon_{n}^{2}\log 
		\left(\frac{M}{\epsilon_n}\right)\right)\sum \limits_{j 
		\geq M_{n}}{\exp\left(-n\Psiba_{\Gcal_{n}}(j\epsilon_{n})/
		16\right)} & \lesssim &  n^{\overline{M}^2} \sum 
		\limits_{j \geq M_{n}} n^{-C M_{n}^{2}/16} \nonumber \\
		& \lesssim & n^{\overline{M}^2- \frac{C M_{n}^2}{16}} \to 
		0 \nonumber
\end{eqnarray}
as long as $M_{n}$ is chosen such that $M_{n}^2 \geq \frac{32\overline{M}^2}{C}$. 
Therefore, condition~\eqref{eq:well_posterior_contraction_rate_fourth_condition} is 
satisfied. As a consequence, we achieve the conclusion of the 
theorem.
\paragraph{Proof of claim~\eqref{eqn:bound_prior_distributions}} 
According to the formulation of Dirichlet distribution, we obtain that
\begin{eqnarray}
& & Q(|p_{i}-p_{i}^{0}| \leq \epsilon, \ i=1,\ldots,k_{0}) \nonumber \\
& & \hspace{2 em} 
		= \dfrac{\Gamma(\gamma)}{(\Gamma(\gamma/k_{0}))^{k_{0}}} 
		\int \limits_{|p_{i}-p_{i}^{0}| \leq \epsilon, \ 1 \leq i 
		\leq k_{0}} \prod \limits_{i=1}^{k_{0}-1}{p_{i}^{\gamma/
		k_{0}-1}}\biggr(1-\sum \limits_{i=1}^{k_{0}-1}p_{i}
		\biggr)^{\gamma/k_{0}-1}dp_{1}\ldots dp_{k_{0}-1}. \label{eq: dirichlet_probability}
\end{eqnarray}
Now, for $\gamma/k_0 \leq 1$, equation~\eqref{eq: dirichlet_probability} can be re-written as:
\begin{eqnarray}
& & \hspace{-2 em} Q(|p_{i}-p_{i}^{0}| 
		\leq \epsilon, \ i=1,\ldots,k_{0}) \nonumber \\
& & \hspace{2 em} 
		\geq \dfrac{\Gamma(\gamma)}{(\Gamma(\gamma/
		k_{0}))^{k_{0}}} \int \limits_{|p_{i}-p_{i}^{0}| \leq 
		\epsilon, \ 1 \leq i \leq k_{0}} \prod \limits_{i=1}
		^{k_{0}-1}{p_{i}^{\gamma/k_{0}-1}}dp_{1}\ldots 
		dp_{k_{0}-1}\nonumber \\
& & \hspace{2 em} \gtrsim \dfrac{\Gamma(\gamma)}{(\Gamma(\gamma/
		k_{0}))^{k_{0}}}\frac{1}{(\gamma/k_0)^{k_0}} 
		\epsilon^{(k_0-1)(\gamma/k_0)}. \nonumber
\end{eqnarray}
Here, the first inequality in the above display follows from the 
fact that $1-\sum \limits_{i=1}^{k_{0}-1}p_{i} \leq 1$ while the 
second inequality is due to direct integration and the fact that $
\epsilon$ is sufficiently small.\\
On the other hand, for $\gamma/k_0 > 1$ , we can rewrite 
equation~\eqref{eq: dirichlet_probability} as
\begin{eqnarray}
& & Q(|p_{i}-p_{i}^{0}| \leq \epsilon, \ i=1,\ldots,k_{0}) \nonumber \\
& & \hspace{2 em} 
		\geq \dfrac{\Gamma(\gamma)}{(\Gamma(\gamma/
		k_{0}))^{k_{0}}} \left(\frac{1-\sum 
		\limits_{i=1}^{k_{0}-1}p_{i}^0}{2}\right) ^{\gamma/k_0-1}
		\int \limits_{|p_{i}-p_{i}^{0}| \leq \epsilon, \ 1 \leq i 
		\leq k_{0}} \prod \limits_{i=1}^{k_{0}-1}{p_{i}^{\gamma/
		k_{0}-1}}dp_{1}\ldots dp_{k_{0}-1} \nonumber
\end{eqnarray}
where the above inequality follows due to the fact that $p_i^0 >0$ 
for all $i \in \{1,\dots,k_0\}$ and that for sufficiently small $
\epsilon >0$ such that $|p_i -p_i^0|\leq \epsilon$ for all $i \in 
\{1,\dots, k_0-1\}$, we have $p_{k_0} \geq p_{k_0}^0 /2$. 
Therefore, the lower bound for the Dirichlet distribution $Q$ 
follows automatically from the results of these two separate 
conditions of $\gamma/k_0 \leq 1$ and $\gamma/k_0 >1$.

On the other hand, to show the bounds for $H$, we note that 
$\Theta\subset \mathbb{R}^d$.
Suppose $B(\theta,\epsilon):= \{\theta'\subset \Theta : \| \theta-
\theta'\| \leq \epsilon \}$ denotes the $\ell_{2}$ ball in $\Theta \cup 
\mathbb{R}^d$ around $\theta$, with radius $\epsilon$.
Then it can be seen that 
\begin{align*} 
\min_{\theta' \in \Theta}
g(\theta') (\mu(B(\theta',\epsilon))/\mu(\Theta))^{k_0} & \leq H(\|
\theta_{i}-\theta_{i}^{0}\| \leq \epsilon, \ i=1,\ldots,k_{0}) \\
& \hspace{12 em} \leq \max_{\theta' \in \Theta}g(\theta')(\mu(B(\theta',\epsilon))/ 
\mu(\Theta))^{k_0}
\end{align*}
since $B(\theta_i^0,\epsilon)$ are disjoint 
for all $ i \leq k_0$ , for sufficiently small $\epsilon>0$. Here, $
\mu(A)$ denotes the $d$-dimensional Lebesgue measure of the set $A 
\subset \Theta$ and $g$ is the density function of $H$ based on Assumption (P.2). Using the fact that $ \epsilon^d \lesssim
\mu(B(\theta',\epsilon)) \lesssim \epsilon^d$ and the condition that $H$ is approximately uniform in Assumption (P.2), the remainder of 
the claim follows.
\subsection{Posterior consistency of Merge-Truncate-Merge algorithm}
\label{ssub:proof::MTM consistency}
The goal of this section is to both deliver a proof of Theorem~\ref{theorem: merge-truncate-merge consistency} and clarify the role played by each of the steps of the MTM algorithm.

\subsubsection{Probabilistic scheme for merging atoms}
\label{sssection:Merge}
The first step of MTM algorithm comprises of lines from 1 to 3 in Algorithm 1. It describes
a probabilistic scheme for merging atoms from an input measure $G$.
Recall that $G$ is a sample from the posterior distribution of a mixing measure which is assumed to be relatively close to the true $G_0$, per Eq.~\eqref{eq:posterior_results}. To simplify notations within this subsection we shall remove 
subscript $n$ in $\omega_{n}$ in~\eqref{eq:posterior_results}, namely, 
we will not incorporate the randomness of data in the results in this 
subsection.
\comment{
In the literature, identification of the number of true components 
requires the posterior mass for $K$ to contract to $k_0$. For instance, as shown by~\cite{Xu-Xie-Repulsive-17}, one such way to achieve this is by limiting the number of 
components in the prior by using a repulsive-type based measure for 
the component. This approach ensures that the components in the 
posterior are sufficiently well-spaced. We demonstrate that a similar 
effect is also achieved with merge procedure in Algorithm~\ref{algo:merge-truncate-merge}.}

In that regard, suppose that we have a measure $G=\sum_{j } p_j\delta_{\theta_j} \in \overline{\Gcal}(\Theta)$ 
such that $W_r(G,G_0) \leq \delta \omega $ for some $r \geq 1$. Here, $
\delta,\omega$ are sufficiently small such that the following two properties hold:
\begin{itemize}
\label{eq:condition_merge}
    \item[(B.1)] $\omega < \min \{(p_{\min}^{0}/2)^{1/r}, \min_{u \neq v} 
\frac{\|\theta_u^0-\theta_v^0\|}{8} \}$.
    \item[(B.2)] $\sqrt{\delta} < p_{\min}^{0}/(2k_0) $, where $p_{\min}^{0} : = \min_{i = 1}^{k_{0}} p_i^0$.
\end{itemize}
Denote by $\mathcal{A}(G)$ the set of atoms 
corresponding to any mixing measure $G$. For a given $G$ and $\omega$, let $g_{\omega,G}$ be the set of all discrete measures 
which collect the atoms from $G$ such that all their atoms spaced 
apart by a distance at least $\omega$:
\begin{eqnarray}
g_{\omega,G} : = \{\firstmer =\sum_{j} p'_j \delta_{\theta'_j} : \theta'_j \in \mathcal{A}(G),  \min_{u \neq v} 
\|\theta'_u-\theta'_v\| \geq \omega\} \nonumber.
\end{eqnarray}
Note in this definition that any $G' \in g_{\omega,G}$ must have finite number of atoms, because $\Theta$ is compact.

The first merge step in the MTM algorithm is motivated by the following result, which establishes the existence of a probabilistic
procedure that transform $G$ into another measure $G' \in g_{\omega, G}$ that possesses some useful properties, namely, the supporting atoms of $G'$ are well-separated from one another, while $G'$ remains sufficiently close to $G_0$ in the sense of a Wasserstein metric.
\comment{
To ease the presentation of results in this subsection, we state two key properties of mixing measure after the merge scheme from MTM algorithm:
\begin{enumerate}
    \item[(G.1)] $W_r(G,G_0) \leq \left(2k_0^{1/r}+(1+2^r)^{1/r} \right)\delta \omega$.
	\item[(G.2)] There exists a permutation $\tau: \{1,\dots,k\} \rightarrow \{1,\dots,k\}$, such that $|p_{\tau(i)}-p_i^0| \leq (2\delta)^r$ and $\|\theta_{\tau(j)}-\theta_i^0\| \leq \omega/2$ for all $i \leq k_0$.
\end{enumerate}
}
\begin{lemma}
\label{lemma:merge_algorithm_1}
Assume that $W_r(G,G_0) \leq \delta \omega $ for some $r \geq 1$ 
where $\omega, \delta$ satisfy condition (B.1) and (B.2). Then, there exists a probabilistic scheme which transform $G$ into a $\firstmer = \sum_{j=1}^{k} p'_j \delta_{\theta'_j}$ such that $k \geq k_{0}$ and the following holds:
\begin{align*}
    P( \{\firstmer: \firstmer \text{ satisfies (G.1) and (G.2)} \}| G) \geq 1-\delta^{r/2}\sum_{i=1}^{k_0}\frac{1}{p_{i}^{0}}.
\end{align*}
Here, $P$ is the probability measure associated with the probabilistic scheme and the conditions (G.1) and (G.2) stand for
\begin{enumerate}
    \item[(G.1)] $G' \in g_{\omega,G}$ and $W_r(G',G_0) \leq (k_0+2 )\sqrt{\delta} \omega$.
	\item[(G.2)] For each $i=1,\ldots, k_0$ there is an index $j$ for an atom of $G'$ for which $|p_{j}-p_i^0| \leq \delta^{r/2}$ and $\|\theta'_j-\theta_i^0\| \leq \sqrt{\delta}\omega$.
\end{enumerate} 
\end{lemma}
\noindent
\begin{proof}
The probabilistic scheme is the first merge step described in the MTM algorithm. We recall it in the following

\begin{enumerate}
    \item Reorder the indices of components $\{\theta_1,\dots,\theta_{|G|}\}$ by simple random sampling 
without replacement (SRSWOR) with corresponding weights $\{p_1,\dots,p_{|G|}\}$. 
    \item Let $\tau_1,\dots,\tau_{|G|}$ denote the new indices, and set $\mathcal{E} = \{\tau_j \}_{j}$ as the existing set of atoms.
    \item Sequentially for each index $\tau_{j}$, if there exists an index $\tau_{i} < \tau_{j}$ such that 
$\| \theta_{\tau_i} - \theta_{\tau_j} \| \leq \omega$, we perform the following updates
    \begin{itemize}
        \item update $p_{\tau_i} = p_{\tau_i} + p_{\tau_j}$.
        \item update $\mathcal{E}$ by removing index $\tau_j$ from $\mathcal{E}$.
    \end{itemize}
    \item Set $\firstmer = \sum_{j: \ \tau_j \in \mathcal{E}} p_{\tau_j} \delta_{\theta_{\tau_j}}$.
\end{enumerate}
The proof consists of two main steps. First, we shows that every atom of $G_{0}$ lies in a $\sqrt{\delta}\omega$ neighborhood of a unique atom of $G'=\sum_{i=1}^k p'_i\delta_{\theta'_i}$  
having large mass, with high probability. This will allows us to deduce that $\firstmer$ satisfies (G.2) with a high probability. Next, we shall show that 
\begin{eqnarray}
\label{g2g1}
\{\firstmer:\firstmer \text{ satisfies (G.2) }|G\} \subset \{\firstmer:\firstmer \text{ satisfies (G.1) }|G\}.
\end{eqnarray}
to conclude the lemma. Note that by the nature of construction it automatically holds that $\firstmer \in g_{\omega,G}$.
\paragraph{Step 1:} Let $P(B|G)$ be the probability of an event $B$ under the SRSWOR scheme used above, 
conditioned the mixing measure $G$. \comment{Additionally, we choose $n$ to be 
sufficiently large so that we have
\begin{align*}
\omega_n \leq \min \{(p_{\min}^{0})^{1/r}, \min_{u \neq v} 
\frac{\|\theta_u^0-\theta_v^0\|}{2} \}
\end{align*} 
where $p_{\min}^{0} : = \min_i p_i^0$.}
Furthermore, let $G(A)$ denote the mass assigned to the set $A \subset \Theta$ by measure $G$. Thus, for a given $\epsilon > 0$
\begin{align*}
G(\mathbb{B}(\theta, \epsilon\omega))
		= \sum_{i: \| \theta_{i} - \theta \| 
		\leq \epsilon \omega} p_{i}
\end{align*} 
for any $\theta \in \Theta$. 
Now, the amount of mass transfer between $\theta_i^0$ and those atoms of $G$ residing in $\mathbb{B}(\theta_i^0, \epsilon\omega)^{\mathsf{c}}$ is at least $|p_i^0 - G(\mathbb{B}(\theta_i^0, \epsilon\omega))|$. Therefore, as $W_{r}(G, G_{0}) \leq \delta \omega$, 
\begin{align*}
|p_i^0 - G(\mathbb{B}(\theta_i^0, \epsilon\omega))|^{1/r}
		\epsilon \omega \leq \delta \omega
\end{align*}
for any index $i \in \{1,\dots, k_0\}$ and for any $2 \geq \epsilon > 0$. 
The upper bound of $2$ arises from the 
consideration of selecting disjoint balls 
combined with the fact that $\omega < \min_{u \neq v} \frac{\|\theta_u^0-\theta_v^0\|}{4 }$. 
It leads to the following inequalities
\begin{eqnarray} 
\label{eq:delta-bounds}
p_i^0 - \left( \frac{\delta}{\epsilon}\right)^r
		\leq G(\mathbb{B}(\theta_i^0, \epsilon\omega))
		\leq p_i^0 + \left(\frac{\delta}{\epsilon}\right)^r.
\end{eqnarray} 
Since $\omega < \min_{u \neq v} 
\frac{\|\theta_u^0-\theta_v^0\|}{4 }$,  based on the standard union 
bound, the following inequality holds
\begin{eqnarray}
\label{eq:complement_ball_bounds}
G\left(\cup_{ i = 1}^{k_{0}} \mathbb{B}(\theta_i^0, \sqrt{\delta}\omega)\right)
		= \sum_{i = 1}^{k_{0}} G( \mathbb{B}(\theta_i^0, \sqrt{\delta}\omega))
		> 1- k_0(\sqrt{\delta})^r >0.
\end{eqnarray}
The last inequality in the above display holds because 
$1 \geq p_{\min}^{0} > 2 k_0 \sqrt{\delta}$. Now, combining 
Equations~\eqref{eq:complement_ball_bounds} with~\eqref{eq:delta-bounds}, for specific choice of $\epsilon=\sqrt{\delta}$, we get that 
\begin{align*} 
\frac{G(\mathbb{B}(\theta_i^0, \sqrt{\delta}\omega))}{
		G(\mathbb{B}(\theta_i^0, \sqrt{\delta}\omega) \cup (\cup_{ i = 1}^{k_{0}} \mathbb{B}(\theta_i^0, \sqrt{\delta}\omega))^{c})} 
		\geq \frac{p_{i}^{0} - (\sqrt{\delta})^r}{p_{i}^{0} + (k_0 + 1 )(\sqrt{\delta})^r}  \geq 1- \frac{\delta^{r/2}}{p_{i}^{0}} > 0.
\end{align*} 

Divide $\Theta$ into disjoint subsets $\Theta = A_1 \cup \ldots \cup A_{k_0+1}$, where $A_i=\mathbb{B}(\theta_i^0, \sqrt{\delta}\omega)$ for all $i\leq k_0$, and $A_{k_0+1}=(\cup_{ i = 1}^{k_{0}} \mathbb{B}(\theta_i^0, \sqrt{\delta}\omega))^{c}$. For each $i=1,\ldots, k_0$, let $E^i$ denote the event that an atom of $G'$ resides in $A_i$. The probabilistic scheme for the selection of atoms for $G'$ from the those of $G$ (via random sampling without replacement) will pick an atom from $A_i$ and gives it a lower index than one from $A_{k_0+1}$ with probability 
\begin{align*}
    P(E^{i}|G)=\frac{G(A_i)}{G(A_i)+G(A_{k_0+1})}\geq 1- \frac{\delta^{r/2}}{p_{i}^{0}}.
\end{align*}
Moreover, if $\firstmer \in E^i$ and $\theta'_j \in \mathcal{A}(\firstmer)$ such that $\|\theta'_j-\theta_i^0\| \leq \sqrt{\delta}\omega$, 
then 
\begin{align*}
    p_i^0+\delta^{r/2} \geq p_i^0 + \left( \frac{\delta}{2}\right)^r \geq G( \mathbb{B}(\theta_i^0, 2\omega)) 
\geq p'_j  \geq  G( \mathbb{B}(\theta_i^0, \sqrt{\delta}\omega)) \geq p_i^0-\delta^{r/2}.
\end{align*} 
Thus $\firstmer \in E^{i}$ satisfies (G.2). 

This entails that $P( \{\firstmer: \firstmer \text{ satisfies (G.2)} \}| G) 
\geq P(\cap_{i=1}^{i_0} E^i | G) 
\geq 1-\biggr(\sum_{i=1}^{k_0}\frac{\delta^{r/2}}{p_{i}^{0}}\biggr)$, which concludes the first proof step.

An useful fact to be used later is that if $\theta_i^0 \in \mathbb{B}(\theta_{j}', \sqrt{\delta}\omega)$ for some 
$j \leq k$, when $\theta_{j}' \in \mathcal{A}(\firstmer)$, $\firstmer \in E^i$ , then 
$G(\mathbb{B}(\theta_j, \omega)) \geq \omega^r$. 
Indeed, suppose that this claim does not hold, then by the
definition of Wasserstein metric, we find that
\begin{align*}
|p_i^0-\omega^r|^{1/r}\sqrt{\delta}\omega
		\leq W_r(G,G_0) \leq \delta \omega,
\end{align*} 
which is a contradiction as we have $p_{\min}^{0} 
\geq 2 k_0 \sqrt{\delta}$ and $\omega^r < p_i^0/2$.

\paragraph{Step 2:} To establish \eqref{g2g1} it suffices to assume that $G'= \sum_{j} p'_{j}\delta_{\theta'_{j}} \in E^i$ 
satisfies that for every $i =1,\ldots, k_0$,  
$\|\theta'_i -\theta_i^0\| \leq \sqrt{\delta}\omega$, 
and $|p_i^0-G(\mathbb{B}(\theta_i^0, \sqrt{\delta}\omega))| \leq \delta^{r/2}$. Then, we have
\begin{align*}
p_i^0 - \delta^{r/2} 
	\leq G(\mathbb{B}(\theta_i^0, \sqrt{\delta}\omega)) 
	\leq G'(\mathbb{B}(\theta'_{i}, \omega))
	= p'_i.
\end{align*}  
The above result leads to $G'( (\cup_{i=1}^{k_0} 
\mathbb{B}(\theta'_{i}, \omega))^{c})\leq k_0 \delta^{r/2}$. 

Now, we construct an measure $\tilde{G} = \sum_l \tilde{p}_l \delta_{\phi_l}$ 
from $G'$ and $G$, by "de-merging" all atoms of $G'$ except for its first $k_0$ atoms. Specifically, for indices $l \leq k_0$, simply take 
$\tilde{p}_l = p'_l$ and $\phi_l = \theta'_l$. 
Additionally, if index $l > k_{0}$ is such that 
$\|\theta_l-\theta'_i\| > \omega$ for all $i 
\leq k_0$, then $\phi_l=\theta_l$ and $\tilde{p}_l=p_l$. 
Otherwise, let $\tilde{p}_l=0$. By the triangle inequality with Wasserstein metric,
\begin{eqnarray}
\label{proof:lemma_merge_combine}
W_r(G',G_0) \leq W_r(G',\tilde{G})+ W_r(\tilde{G},G_0)
			\leq k_0^{1/r} \sqrt{\delta} \omega + W_r(\tilde{G},G_0).
\end{eqnarray} 
The second inequality above holds because
$\sum_{l=1}^{k_0} \tilde{p}_l 
= \sum_{j=1}^{k_0} p'_j \geq 1- k_0\delta^{r/2}$. So, there exists a coupling of $\firstmer$ and $\tilde{G}$ such that any mass transfer occurs between atoms located at most $\omega$ in distance from each other. Moreover, the coupling can be so obtained that the total mass travelling a non-zero distance is bounded above by $k_0\delta^{r/2}$. 

It remains to obtain a suitable upper bound for $W_{r}^r(\tilde{G},G_0)$. From the definition of Wasserstein metric, we can write
\begin{align*}
W_{r}(\tilde{G},G_0) 
		= \inf_{\vec{q} \in \mathcal{Q}(\vec{\tilde{p}},\vec{p^0})}
		\biggr ({\mathop {\sum }\limits_{i,l}{q_{i l}\|\phi_{l}-
		\theta_{i}^0\|^{r}}}\biggr )^{1/r},
\end{align*} 
where $\mathcal{Q}(\vec{\tilde{p}},\vec{p^0})$ is the set of all 
possible couplings between $\vec{\tilde{p}}=(\tilde{p}_1,\ldots,\tilde{p}_{|G|})$ and $\vec{p^0}=(p_1^0,\ldots,p_{k_0}^0)$. 
Now, we consider a coupling $\vec{q} \in \mathcal{Q}(\vec{\tilde{p}},\vec{p^0})$ such that $q_{ii}=\min\{p_i^0,\tilde{p}_i\}$ for any $i$. Then, the following inequalities hold
\begin{align*}  
W_{r}^r(\tilde{G},G_0) 
		& \leq \sum_{i=1}^{k_0} \tilde{p}_{i} \|\phi_{i}-\theta_i^0\|^r 
		+ \sum_{i,l \neq i} q_{i l}\|\phi_{l}-\theta_{i}^0\|^{r} \\
		& \leq (\sqrt{\delta}\omega)^r+ W_r^r(G,G_0) \leq (\sqrt{\delta} 
		\omega)^r+ (\delta\omega)^r 
		= (1+\delta^{r/2})\delta^{r/2}\omega^r.
\end{align*} 
Therefore, following Equation~\eqref{proof:lemma_merge_combine}, we have,
\begin{align*}
    W_r(G',G_0) \leq (k_0^{1/r}+(1+\delta^{r/2})^{1/r}) \sqrt{\delta}\omega \leq (k_0 +2) \sqrt{\delta} \omega.
\end{align*}
As a consequence, we achieve the conclusion of the lemma.
\end{proof}

\subsubsection{Truncate-merge scheme}
\label{sssection:truncate merge}
\comment{
As demonstrated in Lemma~\ref{lemma:merge_algorithm_1}, the merge procedure in MTM algorithm results in a mixing $\firstmer$ 
having separated atoms and satisfying conditions (G.1) and (G.2). In this subsection, we explore 
possible useful characteristics of any such mixing measure $\firstmer$ 
that possesses these properties. Note that, these properties are useful 
for the analyses of truncate-merge scheme in MTM algorithm later (correspond to Step 11 to Step 21).}

In the previous subsection we studied properties of the first stage of the MTM algorithm, which is applied an arbitrary discrete measure $G$ that is sufficiently close to $G_0$ under Wasserstein metric, namely, $W_r(G,G_0) \leq \delta \omega$ for some small quantities $\delta > 0$ and $\omega > 0$. The next stage of the MTM algorithm comprises of lines 4 to 7 in the algorithm's description. It is applied to a measure $G'$, which is the outcome of the algorithm's first stage. Denote $\firstmer = \sum_{j=1}^{k} p'_j \delta_{\theta'_j}$ where $k \geq k_{0}$. As a consequence of Lemma~\ref{lemma:merge_algorithm_1}, $G'$ satisfies two important properties (G.1) and (G.2), which are to be restated here for the reader's convenience.

\begin{enumerate}
\item[(G.1)] $G' \in g_{\omega,G}$ and $W_{r}(G', G_{0}) \leq \delta' \omega$, where $\delta' = (k_0+2) \sqrt{\delta}$.
\item[(G.2)] For each $i\leq k_0$, there exists an (unique) atom of $G'$, which is relabeled $\theta'_i$ so that $|\theta_i^0 - \theta'_i\| \leq \sqrt{\delta} \omega \leq \omega p_{\min}^{0}/(2k_0)  \leq \omega/(2k_0)$. 
\end{enumerate}
By definition, $G'\in g_{\omega,G}$ implies that its atoms are well-separated, namely, for any $1 \leq i < j \leq k$, $\|\theta_{i}' - \theta_{j}'\| \geq \omega$.
Assuming slightly stronger conditions on the two quantities $\omega$ and $\delta$, we can say more about the structure of $G'$, which turns out to be very useful in identifying the true number of atoms $k_0$ of $G_0$ via a truncation procedure.

\comment{
\begin{lemma}
\label{lemma:closeness}
Assume (T.1) and (T.2).
Furthermore, $\delta'$ is sufficiently small such that condition (B.2) 
holds with $\delta'$. Then, for each $i \leq k_0$, there exists unique $1 \leq j \leq k$ such that $\|\theta_i^0-\theta_j'\| \leq \frac{\omega}{2k_0}$.
\end{lemma}

\begin{proof}
The proof is straight-forward from the properties of $G'$. Indeed, we 
denote $\theta_j'$ the atom of $\firstmer$ which is closest to $\theta_i^0$ for 
any given $1 \leq i \leq k_{0}$. Invoking the definition of Wasserstein 
distances and the continuity property of the function $x^r$, we find that
\begin{eqnarray}
 p_i^0 \|\theta_i^0-\theta_j'\|^r = p_i^0\liminf_{x\to\|\theta_i^0-\theta_j'\|} x^r  \leq W_r^r(G',G_0)\leq (\delta' \omega)^r. \nonumber
\end{eqnarray}
Given the condition (B.2) with $\delta'$, we achieve the desired bound 
with $\|\theta_i^0-\theta_j'\|$ in the conclusion. Furthermore, the 
uniqueness of $j$ follows from the fact that the atoms of $G'$ are 
separated by at least distance $\omega$. As a consequence, we achieve 
the conclusion of the lemma.
\end{proof}
The above lemma states that if the atoms of $G'$ are well-separated and 
$G'$ is sufficiently close to $G_0$, then for each atom of $G_0$, there 
exists a unique atom of $G'$ such that it is also sufficiently close to 
that atom of $G_{0}$. 
}

\begin{enumerate}
    \item[(B.3)] $\omega < \frac{7p_{\min}^{0}\min_{u \neq v}\|\theta_u^0-\theta_v^0\|}{16}$.
    \item[(B.4)] $\sqrt{\delta} < p_{\min}^{0}/(2k_0(k_0+2)) $, where $p_{\min}^{0} : = \min_{i = 1}^{k_{0}} p_i^0$.
\end{enumerate}

\begin{lemma}
\label{lemma:k_0 identification}
Suppose that $\omega$ and $\delta$ satisfy conditions (B.1), (B.3) and (B.4). Then for any $G'$ satisfying properties (G.1) and (G.2), the following hold.
\begin{enumerate}
    \item[(a)] For each $1 \leq i \neq j \leq k_0$, we obtain that
$
(p_j')^{1/r}\|\theta_i' - \theta_j' \| > \omega$.
    \item [(b)] For each $j > k_0$, we find that $
\min_{1 \leq i \leq k_0} (p_j')^{1/r}\|\theta_i' - \theta_j'\| \leq \omega$.
\end{enumerate}
\end{lemma}

\begin{proof}
To show (a), note for any $i,j \leq k_0$
\begin{eqnarray}
\label{eq:lemma 3 theta}
\|\theta'_i -\theta'_j\| \geq  \|\theta_i^0-\theta_j^0\| - \|\theta'_i-\theta_i^0\| -\|\theta'_j-\theta_j^0\|  \geq \|\theta_i^0-\theta_j^0\| - \frac{\omega}{k_0} \geq \frac{7}{8}\|\theta_i^0-\theta_j^0\|
\end{eqnarray}
where the first inequality follows from triangle inequality, the second inequality is due to the hypothesis with $G'$, and the third inequality is due to (B.1).

By the definition of Wasserstein distances, for mass transport to be achieved between $\firstmer$ and $G_0$, an amount of mass at least $|p_i^0-p'_i| $ should be transported from atom $\theta_i^0$ of $G_0$ to an atom of $\firstmer$ other than $\theta'_i$. Hence, for any $i \leq k_0$,
$|p_i^0-p'_i|(\omega-\|\theta'_i-\theta_i^0\|)^r \leq W_r^r(\firstmer,G_0)\leq (\delta'\omega)^r$.
Invoking the hypothesis with $G'$, these inequalities lead to
$|p_i^0-p'_i|^{1/r}\leq 2\delta'$.
Combining with the condition $\sqrt{\delta} < \frac{p_i^0}{2k_0(k_0+2)}$, the above inequality leads to
$p'_i>p_i^0- \frac{(p_i^0)^2}{2k_0^2} \geq \frac{p_i^0}{2}$.
Combining this with Equation~\eqref{eq:lemma 3 theta} and (B.3) to conclude. 

Turning to part (b), 
suppose for some $j>k_0$, we have $
(p'_j)^{1/r}\|\theta'_i-\theta'_j\| > \omega$ for all $i\leq k_0$. Then
by triangle inequality and the properties of $G'$, we find that
\begin{align*}
\|\theta'_j-\theta_i^0\|+ \frac{\omega}{2k_0}
		\geq \|\theta'_j-\theta_i^0\| + \|\theta_{i}^{0} - \theta'_{i}\| 
		\geq \|\theta'_{i} - \theta'_{j}\| 
		> \omega/{p'_j}^{1/r}.
\end{align*}
Applying the triangle inequality again,
$\|\theta_i^0-\theta'_j\| \geq \|\theta'_j-\theta'_i\| - \|\theta_{i}^{0} - \theta'_{i}\| 
\geq \omega(1-\frac{1}{2k_0}) \geq \frac{\omega}{2}$.
Combining the two preceeding bounds, we get
$2\|\theta_j'-\theta_i^0\| > \omega/{p'_j}^{1/r}$ for all $i\leq k_0$.

Now, since $W_r(\firstmer,G_0) \leq \delta' \omega$, we can find a 
coupling $\vec{q} \in \mathcal{Q}(\vec{p'}, \vec{p}^{0})$ between 
$\vec{p'}=(p'_1,\ldots,p'_{k})$ and 
$\vec{p^0}=(p_1,\ldots,p_{k_0}^0)$ 
such that 
$\sum_{ij} q_{ij}\|\theta'_{j}-\theta_{i}^0\|^{r} \leq (\delta'\omega)^r$.
However, based on the previous inequalities, we have for the given index $j$ 
\begin{align*}
\sum_{i=1}^{k_0} q_{ij}  \|\theta'_{j}-\theta_{i}^0\|^{r} 
		> \sum_{i=1}^{k_0} q_{ij} \left(\frac{1}{2}\right)^{r} \omega^r/p'_j
		= \left(\frac{1}{2}\right)^r \omega^r,  
\end{align*} 
which is a contradiction as $\delta' < 1/2$ due to condition (B.4). This concludes the proof of the lemma.
\end{proof}

\subsubsection{Proof of Theorem~\ref{theorem: merge-truncate-merge consistency}}
Now we are ready for the proof of this theorem. 
It suffices to prove for the case constant $c=1$.
\paragraph{Proof of part (a):} Recall that we are given a (random) mixing measure for which the following holds: for each fixed $\epsilon > 0$ and $\delta > 0$, as $n \rightarrow \infty$, there holds
\[P_{G_0}^n \biggr\{\Pi \biggr(W_r(G,G_0) 
\geq \delta\omega_n \big|(X_1,\dots,X_n)\biggr) \geq \epsilon \biggr\} \to 0.\]

Choose $\delta$ sufficiently small, and as $n$ gets large $\omega_n$ also becomes so small that all conditions (B.1--4) in the preceeding sections are satisfied.
Then, we can appeal to Lemma~\ref{lemma:merge_algorithm_1} to obtain that, measure $G'$ as produced in the probablistic merge stage of the MTM algorithm also admits a posterior contraction toward $G_0$, in the sense that as $n\rightarrow \infty$
\begin{eqnarray}
P_{G_0}^n \biggr \{\Pi \biggr(W_r(G',G_0) \leq (k_0+2)\sqrt{\delta} \omega_n \big|(X_1,\dots,X_n)\biggr)
 	\geq (1-\epsilon)\biggr(1-\sum_{i=1}^{k_0}\frac{\delta^{\frac{r}{2}}}{p_{i}^{0}}\biggr)\biggr \} \to 1. \nonumber
\end{eqnarray}
Since this holds for any $\delta > 0$, we deduce that the posterior probability
$\Pi \biggr(W_r(G',G_0) \leq (k_0+2)\sqrt{\delta} \omega_n \big|(X_1,\dots,X_n)\biggr) 
\rightarrow 1$ in $P_{G_0}$ probability.

Suppose that  $G'$ satisfies both conditions (G.1) and (G.2) (per Lemma~\ref{lemma:merge_algorithm_1}), then it can be verified that if the atoms of $\firstmer$ are arranged in descending order of their masses, then each  of the top $k_0$ atoms of $\firstmer$ lie in an $\frac{\omega_n}{2k_0}$- ball around an atom of $G_0$.
Specifically, using the representation $\firstmer=\sum_{i=1}^{k_n} q_i \delta_{\phi_i}$, where $q_1\geq \dots \geq q_{k_n}$, we have that
$\|\theta_i^0-\phi_i\| \leq \frac{\omega_n}{2k_0}$ and $|q_i-p_i^0| \geq \delta^{r/2}$ for all $i \in \{1,\ldots,k_0\}$.

Recall that $G'$ is fed into the second stage, the truncate-merge procedure, of the MTM algorithm. Note that $|q_i-p_i^0| \leq \delta^{r/2}$  implies 
$q_i>p_i^0-\delta^{r/2} >p_i^0/2 >\omega_n^r$ for $n$ sufficiently large. 
By Lemma ~\ref{lemma:k_0 identification} that for each $j > k_0$,  
$\min_{i\leq k_0} (q_j)^{1/r}\|\phi_i-\phi_j\| \leq \omega_n$, but for each $i,j \leq k_0, i \neq j$,
$(q_j)^{1/r}\|\phi_i-\phi_j\| \geq \omega_n$.
Following the definition of $\tilde{k} = |\mathcal{A}|$, we deduce that $\tilde{k} = k_0$.
The final step is to coat this guarantee with a probability statement, due to the fact that $G'$ is random given $G$,
\begin{eqnarray}
\label{eq:posterior mode mass}
P_{G_0}^n \biggr\{\Pi\biggr(\tilde{k} = k_0|X_{1},\ldots,X_{n}\biggr)  
		& \geq & (1-\epsilon)\left(1-\sum_{i=1}^{k_0}\frac{\delta^{r/2}}{p_{i}^{0}}\right)\biggr\} \longrightarrow 1. 
\end{eqnarray}
Let $\delta \rightarrow 0$ to conclude the proof of part (a).

\paragraph{Proof of part (b):} The proof boils down to showing that the reassignment of mass as in the second stage of the MTM Algorithm only increases the Wasserstein distance by a constant factor. Denote by $\vec{p^0}=(p_1^0,\ldots,p_{k_0}^0)$ and 
$\vec{q}=(q_1,\ldots,q_{k_n})$ the weight vectors of $G_{0}$ 
and $G'$ respectively. Suppose that $W_r(G',G_0) \leq (k_0+2)\sqrt{\delta} \omega_n$ as before. So
can find a coupling $\vec{f} \in \mathcal{Q}(\vec{p^0},\vec{q})$ 
such that 
\begin{align*}
\biggr ({\mathop {\sum } \limits_{i,j}
		{f_{ij}\|\theta_{i}^0 - \phi_j\|^{r}}}\biggr )
		^{1/r} \leq 2 (k_0+2)\sqrt{\delta} \omega_n.
\end{align*} 
Define the set $V_{i,n} : = \{\theta \in \Theta :  
\|\theta - \theta_i^0\| \leq \min_{u \neq v}\frac{\|\theta_u^0-
\theta_v^0\|}{2} - \frac{\omega_n}{2k_0} \}$. Furthermore, the following inequalities hold 
\begin{align*}
\|\phi_j - \theta_i^0\| 
		\geq \min_{u \neq v}\frac{\|\theta_u^0- \theta_v^0\|}{2} - \frac{\omega_n}{2k_0} 
		\geq \min_{u \neq v}\frac{\|\theta_u^0- \theta_v^0\|}{4}
\end{align*}
for all $i \leq k_0$ and $j \not\in V_{i,n}$, because $\omega_n$ satisfies assumption (B.1). Therefore, we find that
\begin{align*} 
\sum_{i,j : \phi_j \not\in V_{i,n}} f_{ij} 
		\leq 4\left(\frac{2(k_0+2)\sqrt{\delta} \omega_n}{\min_{u \neq v} \|\theta_u^0- \theta_v^0\|}\right)^r .
\end{align*} 
Notice that if $j \in V_{i,n}$, the second stage of the MTM Algorithm
assigns the mass corresponding 
to atom $j$ of $\firstmer$ to atom $i$ of 
$\secmer$. We can assume henceforth that $\secmer$ is such that $|\mathcal{A}(\secmer)|=k_0$ as a result of the proof of part (a) of this theorem.

Since the sets $V_{i,n}$ are disjoint, this assignment 
is unique. It follows that we can find 
$\vec{f}' \in \mathcal{Q}(\vec{p_0},\vec{q'})$ such that
\begin{eqnarray} 
\label{eq:proof:theorem 3.2 part b}
\sum_{i, j \neq i} {f'_{ij} \|\theta_{i}^0 - \phi_j\|^{r}} 	
	& \leq &  4\left(\frac{2\text{Diam}(\Theta)(k_0+2)\sqrt{\delta} \omega_n}{\min_{u \neq v} \|\theta_u^0- \theta_v^0\|}\right)^r, \nonumber \\
\sum_i f'_{ii} \|\theta_{i}^0-\phi_{i}\|^r 
	& \leq &  ((k_0+2)\sqrt{\delta} \omega_n)^r
\end{eqnarray} 
where $\vec{q'}=(q'_1,\ldots,q'_{k_0})$ is the weight vector of $\secmer$. 
The first inequality above follows, since $\|\phi_i-\theta_i^0\| \leq \|\phi_j-\theta_i^0\|$ for all pairs $i,j$ with $i \leq k_0$, with strict inequality for $i\neq j$. To obtain the conclusion for the second inequality in~\eqref{eq:proof:theorem 3.2 part b}, we note that 
$p_i^0 =\sum_jf_{ij}=\sum_jf'_{ij}$. Therefore, $f'_{ii}=\sum_j f_{ij} -\sum_{j \neq i} f'_{ij}$. Then, if $\phi_j$ is an atom of $\firstmer$, 
for any $j \neq i$, we have $\|\theta_{i}^0-\phi_{i}\| \leq \|\theta_{i}^0-\phi_{j}\|$. Hence, we find that
\begin{eqnarray}
\sum_i f'_{ii} \|\theta_{i}^0-\theta_{i}\|^r \leq \sum_{i,j} f_{ij} \|\theta_{i}^0-\phi_j\|^r \leq W_r^r(\firstmer,G_0)\leq ( (k_0+2)\sqrt{\delta} \omega_n)^r.\nonumber
\end{eqnarray}
By the nature of construction $\mathcal{A}(\secmer) \subset \mathcal{A}(\firstmer)$.
Using the two parts of Equation~\eqref{eq:proof:theorem 3.2 part b}, we obtain that 
\begin{center}
    $W_r(\secmer, G_0) \leq \left(1 +  4 \left(\frac{2\text{Diam}(\Theta)}{\min_{i,l} 
	\|\theta_i^0- \theta_l^0\|}\right)^{r}\right)^{1/r}((k_0+2)\sqrt{\delta}\omega_n)$.
\end{center}
The full probability statement is
 \begin{eqnarray}
 \label{eq:consistency maintainence}
 P_{G_0}^n \biggr\{\Pi \biggr(G \in \overline{\mathcal{G}}(\Theta): W_r(\secmer, G_0) \leq C\delta\omega_n \big|(X_1,\dots,X_n)\biggr)
 	\geq (1-\epsilon)\biggr(1-\sum_{i=1}^{k_0}\frac{\delta^{\frac{r}{2}}}{p_{i}^{0}}\biggr)\biggr \} \to 1,
 \end{eqnarray}
 where $C=\left(1 +  4 \left(\frac{2\text{Diam}(\Theta)}{\min_{i,l} 
	\|\theta_i^0- \theta_l^0\|}\right)^{r}\right)^{1/r}(k_0+2)$ is a constant dependent on $G_0$ and $\Theta$. 
Finally, letting $\delta \rightarrow 0$ we obtain the desired conclusion for part (b).

\subsection{Proof of Lemma~\ref{lemma:Prior_mass_MFM}}
\label{subsection:proof_lemma_prior_mass_MFM}
To simplify the proof argument, we specifically assume that $G_{*}$ is a discrete mixture. The proof argument for other settings of $G_{*}$ is similar and is omitted. Now, we consider an $\epsilon > 0$ maximal packing set of parameter 
space $\Theta$. It leads to a $D-$partition $(S_1,\dots,S_D)$ of $
\Theta$ such that  $\text{Diam}(S_i) \leq  2\epsilon$ for all $1 
\leq i \leq D$. Choose $\epsilon$ to be sufficiently small such that $D>\gamma$.

For mixing measures $G=\sum_i p_i \delta_{\theta_i}$ and $G_{*} = \sum_{i=1}
^{\infty} p_{i}^{*}\delta_{\theta_{i}^{*}}$, we denote 
$G(S_i) := \sum_{i : \theta_i \in S_i} p_i$ and $G_*(S_i)=
\sum_{i:\theta_i^* \in S_i} p_i^*$. Invoking the detailed 
formulation of Wasserstein metric, we can check that
\begin{align*}
W_r^r(G,G_*) \leq (2\epsilon)^r 
		+ \text{Diam}^r(\Theta)\sum_{i=1}^D|G(S_i)-G_*(S_i)|.
\end{align*}
Equipped with the above inequality, the following inequality holds
\begin{align*}
\Pi(W_r^r(G,G_*) 
		\leq (2^r+1)\epsilon^r) \geq \Pi\left(\sum_{i=1}^D|
		G(S_i)-G_*(S_i)| 
		\leq (\epsilon/ \text{Diam}(\Theta))^r\right).
\end{align*}
For any positive constant $A$, we find that
\begin{align*}
\Pi \left(\sum_{i=1}^D|G(S_i)-G_*(S_i)| 
		\leq A \right) \geq q_D\Pi( B \cap \{|G(S_i)-G_*(S_i)| 
		\leq A/D, \text{for each } i\}| K=D )
\end{align*}
where $B$ stands for the event that each $S_i$ contains exactly 
one atom of $G$.

Governed by the above observations, by substituting 
$A = (\epsilon/ \text{Diam}(\Theta))^r$, we obtain that
\begin{eqnarray}
& & \Pi(W_r^r(G,G_*) \leq (2^r+1)\epsilon^r) \nonumber \\ 
& & \hspace{2 em} \gtrsim q_D \left(c_0 \left(\frac{\epsilon}
{\text{Diam}(\Theta)}\right)^{d} \right)^{D}  \underbrace{\Pi(\{|
G(S_i)-G_*(S_i)| \leq A/D, \text{for each } i\}|B \cap  \{K=D\} )}_{:= T}   \nonumber.
\end{eqnarray}

By means of Dirichlet probability model assumption on $\Delta_{D-1}$, we 
have the following evaluations with $T$
\begin{eqnarray}
T 
		& \gtrsim &  D!\frac{\Gamma(\gamma)}{\prod_{i=1}^D 
		\Gamma(\gamma/D)} \underset{\mathcal{U}} {\int} 
		\prod_{i=1}^{D-1}(G(S_i))^{(\gamma/D)-1}(1-\sum_{i=1}
		^{D-1}G(S_i))^{(\gamma/D)-1} \mathrm{d}(G(S_i)) \nonumber  \\
& \geq & D!\frac{\Gamma(\gamma)}{\prod_{i=1}^D \Gamma(\gamma/D)} 
		\prod_{i=1}^{D-1} \int_{\max(G_*(S_i)- (\epsilon/ 
		\text{Diam}(\Theta))^r/D , 0)}^{\min(G_*(S_i)+(\epsilon/ 
		\text{Diam}(\Theta))^r/D , 1)} (G(S_i))^{(\gamma/D)-1}  
		\mathrm{d}(G(S_i)) \nonumber \\
& \geq & D!\frac{\Gamma(\gamma) \gamma/D}{\prod_{i=1}^D (\gamma/D) 
		\Gamma(\gamma/D)} \left( \frac{1}{D} \left(\frac{\epsilon}
		{\text{Diam}(\Theta)}\right)^r\right)^{\gamma(D-1)/D} \nonumber
\end{eqnarray}
where $\mathcal{U} : = \Delta_{D-1} \cap |G(S_i)-G_*(S_i)| 
\leq (\epsilon/ \text{Diam}(\Theta))^r/D$. Here, the second 
inequality in the above display is due to the fact that $(1-
\sum_{i=1}^{D-1}G(S_i))^{(\gamma/D)-1}>1$ as $\gamma < D$. 
Invoking the basic inequality $\alpha \Gamma(\alpha) <1$ for $0<
\alpha <1$, we reach the conclusion of the lemma.
\subsection{Proof of Proposition~\ref{proposition:lower_bound_weighted_Hellinger_Wasserstein_infinite_Gaussian}}
\label{subsection:proof_proposition_lower_bound_hellinger_Wasserstein _infinite_Gaussian}
We denote a sphere of radius $R$ as $S_R := \{x \in \mathbb{R}^d : \|x\|_2 \leq R\}$ 
for any $R > 0$. Direct computations lead to
\begin{align}
2 V(p_G,p_{G'}) & = \int_{\mathbb{R}^d} |p_G(x)-p_{G'}(x)| \mathrm{d}\mu(x) \nonumber \\
& = \int_{S_R^{\mathsf{c}}}|p_G(x)-p_{G'}(x)| \mathrm{d}\mu(x) + 
\int_{S_R}|p_G(x)-p_{G'}(x)|\frac{p_{G_0,f_0}(x)}{p_{G_*}(x)} 
\frac{p_{G_*}(x)}{p_{G_0,f_0}(x)} \mathrm{d}\mu(x) \nonumber \\
& \leq \int_{S_R^{\mathsf{c}}} |p_G(x)-p_{G'}(x)| \mathrm{d}\mu(x) \nonumber \\
& \hspace{8 em} +
\norm{\frac{p_{G_*}(\cdot)}{p_{G_0,f_0}(\cdot)} \mathbbm{1}_{S_R}}
_{\infty} \int_{\mathbb{R}^d}|p_{G}(x)-p_{G'}(x)|\frac{p_{G_0,f_0}(x)}
{p_{G_*}(x)} \mathrm{d}\mu(x)\label{eq:weighted_total_variation_lower_bound}
\end{align}
where the last inequality is an application of Holder's inequality. Now, a direct evaluation yields that
\begin{eqnarray}
\int_{S_R^{\mathsf{c}}}|p_G(x)-p_{G'}(x)| \mathrm{d}\mu(x) 
		& \leq & 2\int_{S_R^{\mathsf{c}}}\max_{G} p_G(x) \mathrm{d}\mu(x) 
		\leq 2\int_{S_R^{\mathsf{c}}} \sup_{\theta} f(x|\theta) d
		\mu(x) \nonumber \\
& \leq & 2\int_{S_R^{\mathsf{c}}} \sup_{\theta} \frac{1}{|2\pi
		\Sigma|^{1/2}} \exp( - \|x - \theta\|_{2}^{2}/
		(2\lambda_{\text{max}})) \mathrm{d}\mu(x). \label{eq: upper_bound_total_variation_complement_ball}
\end{eqnarray}

The last inequality is due to the fact that $(x- \theta)^{\top}
\Sigma^{-1}(x - \theta) \geq \|x - \theta\|_{2}^{2}/
\lambda_{\text{max}}$ for all $x \in \mathbb{R}^{d}$ and $\theta 
\in \Theta$. 

We now assume that $d>2$ as the $d\leq 2$ case can be treated similarly. Since $\Theta \subset \mathbb{R}^d$ is bounded, we can find $r>0$ 
such that $\| \theta \|_2 < r $ for all $\theta \in \Theta$. Now, 
given $R > r$, for any fixed value of $x \in S_{R}^{c}$, we can 
check that $\inf_{\theta} \|x - \theta\|_{2}^{2} \geq (\|x\|_{2} - 
r)^{2}$. Therefore equation~\eqref{eq: upper_bound_total_variation_complement_ball} 
leads to 
\begin{eqnarray}
\int_{S_R^{\mathsf{c}}}|p_G(x)-p_{G'}(x)| \mathrm{d}\mu(x) & \leq & 2 
\int_{S_R^{\mathsf{c}}} \frac{1}{|2\pi\Sigma|^{1/2}} \exp( - (\|x
\|_{2} - r)^{2}/(2\lambda_{\text{max}})) \mathrm{d}\mu(x) \nonumber.
\end{eqnarray}
Invoking spherical coordinates by substituting $z= \|x\|_2$, we get
\begin{eqnarray}
\int_{S_R^{\mathsf{c}}}|p_G(x)-p_{G'}(x)| \mathrm{d}\mu(x) & \leq & 2 
\int_{z>R} \frac{z^{d-1}}{|2\pi\Sigma|^{1/2}} \exp( - (z - r)^{2}/
(2\lambda_{\text{max}})) dz. \nonumber
\end{eqnarray}
We denote $g_R(d) : = \int_{z>R} {z^{d}} \exp( - (z - r)^{2}) dz$. By 
integrating by parts with some basic algebraic manipulation, we find that
\begin{eqnarray}
\label{eq:tailbounds gaussian}
g_R(d-1) = (d/2) g_R(d-3) + (R^{d-2}/2)\exp(-(R-r)^2) + r 
g_R(d-2).
\end{eqnarray}
Observe that $(R^{d-2}/2)\exp(-(R-r)^2) \gtrsim (R^{s}/2)\exp(-(R-r)^2) $ for all $s \leq d-2$.
Also, $\int_x^{\infty} \exp(-t^2/2) \mathrm{d}t \leq \frac{1}{x}\exp(-x^2/2) \lesssim x^{d-2} \exp(-x^2/2)$ and $\int_x^{\infty} t\exp(-t^2/2) \mathrm{d}t \lesssim \exp(-x^2/2) \lesssim x^{d-2} \exp(-x^2/2)$ follows using standard arguments for gaussian tailbounds. Here we use the condition $d>2$. For $d\leq 2$, the tail probability can be directly bounded using standard gaussian tailbounds.

We can expand $g_R(s)$ recursively using Equation~\ref{eq:tailbounds gaussian} for all $s\leq d-2$ as well. Now, equipped with  equation~\eqref{eq:tailbounds gaussian}, and following the discussion in the previous paragraph, we can write 
\begin{eqnarray}
\int_{S_R^{\mathsf{c}}}|p_G(x)-p_{G'}(x)| \mathrm{d}\mu(x)  \lesssim  
R^{d-2}\exp(-(R-r)^2/2\lambda_{\max}) \label{eq: bound_v2}.
\end{eqnarray}

Now, we demonstrate that $\norm{\frac{p_{G_*}(\cdot)}{p_{G_0,f_0}
(\cdot)} \mathbbm{1}_{S_R}}_{\infty}$ is bounded above by $c_2 
\exp(\lambda_{\min}^{-1}R^2)$ for some positive constant $c_2$ 
depending only on $C_1, G_{0}$ and $\lambda_{\min}$. Recall that $G_0 =\sum_{i=1}^{k_0} 
p_i^{0} \delta_{\theta^{0}_i}$. Here $k_0$ can be allowed to be $\infty$. The analysis follows through similar to the finite $k_0$ case. 

The conditions on $p_{G_0,f_0}$ imply that 
\begin{align*}
\norm{\frac{p_{G_0,f}(\cdot)}{p_{G_0,f_0}(\cdot)}\mathbbm{1}_{S_R}}_{\infty} \leq \sup_{x \in \mathbb{R}^d, \theta \in \Theta, \theta_0 \in \textrm{supp}(G_0)} \frac{f(x|\theta)}{f_0(x|\theta_0)} \mathbbm{1}_{\|x\|_2 \leq R} \leq C_1 \exp (C_0 R^2).
\end{align*} 
Then, we have the following inequalities
\begin{eqnarray}
\norm{\frac{p_{G_*}(\cdot)}{p_{G_0,f_0}(\cdot)} \mathbbm{1}_{S_R}}_{\infty} & = & \norm{\frac{p_{G_*}(\cdot)}{p_{G_0,f}(\cdot)} \frac{p_{G_0,f}(\cdot)}{p_{G_0,f_0}(\cdot)}\mathbbm{1}_{S_R}}_{\infty} \nonumber \\ 
&\lesssim &\underset{x\in S_{R}, \ i \in \{1,\dots,k_0\}}{\sup} C_1 \exp (C_0 R^2)\exp
		\left(\frac{1}{2}(x-\theta^0_i)'\Sigma^{-1}(x-\theta^0_i)
		\right) \nonumber \\
& \leq &  C_1 \exp (C_0 R^2) \underset{x\in S_{R}}{\sup} \exp(\lambda_{\min}^{-1} (\|x
		\|^2 + \sup_i \|\theta_i^0\|^2)) \nonumber \\
		& \leq & c_2 \exp((\lambda_{\min}^{-1}+C_0)R^2). \label{eq: bound_v3}
\end{eqnarray} 
The bounds apply uniformly for all $R>r$. Therefore, when $R \geq 4r$, we can bound equation \eqref{eq:weighted_total_variation_lower_bound} according to the 
bounds in~\eqref{eq: bound_v2} and~\eqref{eq: bound_v3} as 
follows:
\begin{eqnarray*}
V(p_G,p_{G'}) 
		& \lesssim & \exp((\lambda_{\min}^{-1}+C_0)R^2) \vba(p_G,p_{G'}) 
		+  R^{d-2}\exp(- \lambda_{\max}^{-1}R^2/4)
\end{eqnarray*}
where $ \vba(p_G,p_{G'}) : = \int_{\mathbb{R}^d}|p_{G}(x)-p_{G'}(x)|\frac{p_{G_0,f_0}(x)}{p_{G_*}(x)} \mu(\mathrm{d}x) $ 
is the weighted variational distance. Now consider $R>0$ 
satisfying $R^{d-2}\leq \exp(-\lambda_{\max}^{-1}R^2/8)$. If $
\vba(p_G,p_{G'})= \exp(-(\lambda_{\min}^{-1}+ C_0+\lambda_{\max}^{-1} /8)R^2)$, we obtain that 
\begin{eqnarray*}
V(p_G,p_{G'}) & \lesssim & \exp(-\lambda_{\max}^{-1}R^2/8) = (\vba(p_G,p_{G'}))^{\frac{1}{1+  8\lambda_{\max}(\lambda_{\min}^{-1}+C_0)}}.  
\end{eqnarray*}
Note that, $\vba(p_G,p_{G'}) \lesssim \hba(p_G,p_{G'})$ by standard application of Holder's 
inequality. Also, since the kernel for 
location Gaussian mixtures is supersmooth~\cite{Fan_1990}, it follows from 
Theorem 2 in~\cite{Nguyen-13} that $V(p_{G},p_{G'}) \gtrsim \exp
\biggr(-1/W_{2}^{2}(G,G')\biggr) $. The proof of the proposition 
now follows from this fact, and the inequality connecting weighted 
Hellinger and variational distances.

\subsection{Proof of Proposition~\ref{proposition:lower_bound_weighted_Hellinger_Wasserstein_infinite_Laplace}}
\label{subsection:proof_proposition_lower_bound_hellinger_Wasserstein _infinite_Laplace}
We denote a sphere of radius $R$ as $S_R := \{x \in \mathbb{R}^d : \|x\|_2 \leq R\}$ 
for any $R > 0$. Assume  $\max_{\theta \in \Theta} \|\theta\|_2 \leq r$ and also $\sup_{i \leq k_0} \|\theta_i^0\|_2 
\leq r$.  We consider $R > 2r$ large enough such that, $\|\theta\|_2 \leq r $ and $\|x\|_2 \geq R$ implies
\begin{align*} 
C_{\text{lower}}\dfrac{\exp\left(-\sqrt{\frac{2}{\lambda}}\|x-\theta\|
_{\Sigma^{-1}}\right)}{(\|x-\theta\|_{\Sigma^{-1}})^{(d - 1)/2}} 
\leq f(x| \theta) \leq  C_{\text{upper}}\dfrac{\exp\left(-\sqrt{\frac{2}
{\lambda}}\|x-\theta\|_{\Sigma^{-1}}\right)}{(\|x-\theta\|
_{\Sigma^{-1}})^{(d - 1)/2}}.
\end{align*}
The above inequalities can always be achieved for $R$ large enough because of the asymptotic formulation of multivariate Laplace 
distributions.

Following equation~\eqref{eq:weighted_total_variation_lower_bound} in the 
proof of Theorem~\ref{proposition:lower_bound_weighted_Hellinger_Wasserstein_infinite_Gaussian}, 
we will prove the proposition by providing upper bounds for $
\norm{\frac{p_{G_*}(\cdot)}{p_{G_0,f_0}(\cdot)} \mathbbm{1}_{S_R}}
_{\infty}$ and 
$ \int_{S_R^{\mathsf{c}}}|p_G(x)-p_{G'}(x)| \mathrm{d}\mu(x)$. Because the 
Laplace density is bounded, $\norm{p_{G_*}(\cdot) \mathbbm{1}
_{S_R}}_{\infty}$ is bounded by a constant.  Similar to the proof of Proposition~\ref{proposition:lower_bound_weighted_Hellinger_Wasserstein_infinite_Gaussian}, we have,
\begin{eqnarray}
\norm{\frac{p_{G_*}(\cdot)}{p_{G_0,f_0}(\cdot)} \mathbbm{1}_{S_R}}_{\infty} 
		&\lesssim & C_1\exp(C_0 R^{\alpha})\max_{x, \theta : \|\theta\|_2 \leq r, \|x\|_2 
		\leq R} {\exp\left(\sqrt{\frac{2}{\lambda\lambda_{\min}}}\|x-\theta\|
		_{2}\right)}{(\|x-\theta\|_{2})^{(d - 1)/2}} \nonumber \\
&\lesssim &  \exp(C_0 R^{\alpha}){\exp\left(\sqrt{\frac{2}{\lambda\lambda_{\min}}}(R+r)\right)}{(R
		+r)^{(d - 1)/2}} \nonumber\\
&\lesssim & {\exp\left(\left(\sqrt{\frac{2}{\lambda\lambda_{\min}}} +C_0\right)R^{\alpha}\right)}{R^{(d - 1)/2}} \nonumber.
\end{eqnarray}
Now, in order to minimize $\int_{S_R^{\mathsf{c}}}|p_G(x)-p_{G'}(x)| \mathrm{d}\mu(x)$, observe that 
\begin{eqnarray}
\int_{S_R^{\mathsf{c}}}|p_G(x)-p_{G'}(x)| \mathrm{d}\mu(x) 
		& \lesssim & \int_{S_R^{\mathsf{c}}}\sup_{G} p_G(x) d
		\mu(x) \lesssim \int_{S_R^{\mathsf{c}}} \sup_{\theta} f(x|
		\theta) \mathrm{d}\mu(x) \nonumber \\
& \lesssim & \int_{S_R^{\mathsf{c}}} \sup_{\theta} \dfrac{\exp
		\left(-\sqrt{\frac{2}{\lambda}}\|x-\theta\|_{\Sigma^{-1}}
		\right)}{(\|x-\theta\|_{\Sigma^{-1}})^{(d - 1)/2}} \nonumber \\
& \lesssim & \int_{S_R^{\mathsf{c}}}  \frac{1}{(\|x\|_2 - r)^{(d - 1)/2}} \exp \left( -\sqrt{2\lambda_{\max}^{-1}}(\|x\|_2-r) \right) d
		\mu(x) \nonumber\\
& \lesssim & \int_{S_R^{\mathsf{c}}}  \frac{1}{(\|x\|_2 )^{(d - 1)/2}} \exp \left( -\sqrt{\frac{2}{\lambda\lambda_{\max}}} \|x\|_2 \right) \mathrm{d}\mu(x). \nonumber
\end{eqnarray}
Substituting $z=\|x\|_2$ in above equation, we get 
\begin{eqnarray}
\int_{S_R^{\mathsf{c}}}|p_G(x)-p_{G'}(x)| \mathrm{d}\mu(x) 
		\lesssim \int_{S_R^{\mathsf{c}}}  \frac{z^{d-1}}{z^{(d - 1)/2}} \exp \left( - z \sqrt{\frac{2}{\lambda\lambda_{\max}}} \right) dz . \nonumber
\end{eqnarray}
 Denote $g_R(s) := \int_{S_R^{\mathsf{c}}}  {z^s} \exp( -z) dz$. Then, we find that
 \begin{eqnarray}
 g_R(s) = R^s \exp(-R) + s g_R(s-1). \nonumber
 \end{eqnarray}
Invoking integration by parts with the above equality for $s=\frac{d-1}{2}$ leads to the following inequality 
 \begin{eqnarray}
\int_{S_R^{\mathsf{c}}}|p_G(x)-p_{G'}(x)| \mathrm{d}\mu(x) \lesssim R^{\frac{d-1}{2}}\exp\left( -\sqrt{\frac{2}{\lambda\lambda_{\max}}}R \right) . \nonumber
\end{eqnarray}
Since the above bounds apply for all $R$ large enough, following the approach with 
equation~\eqref{eq:weighted_total_variation_lower_bound} in the 
proof of Theorem~\ref{proposition:lower_bound_weighted_Hellinger_Wasserstein_infinite_Gaussian}, we can write
\begin{eqnarray}
V(p_G,p_{G'}) 
		& \lesssim &  {\exp\left( \left(\sqrt{\frac{2}{\lambda\lambda_{\min}}} +C_0\right)R^{\alpha}
		\right)}{R^{\frac{d-1}{2}}} \vba(p_G,p_{G'}) + 
		R^{\frac{d-1}{2}}\exp \left(-\sqrt{\frac{2}{\lambda\lambda_{\max}}}R \right)\nonumber.
\end{eqnarray}
Recall that $ \vba(p_G,p_{G'}) : = \int_{\mathbb{R}^d}|p_{G}(x)-p_{G'}(x)|\frac{p_{G_0}(x)}{p_{G_*}(x)} \mu(\mathrm{d}x) $ 
is the weighted variational distance. By setting $\vba(p_G,p_{G'})=\exp\left(-\left[\sqrt{\frac{2}{\lambda\lambda_{\min}}} + 
\sqrt{\frac{2}{\lambda\lambda_{\max}}} +C_0\right]R^{\alpha} \right)$, as $\alpha \geq 1$, we see that 
\begin{eqnarray}
V(p_G,p_{G'}) 
		& \lesssim &  \left(\log \frac{1}{\vba(p_G,p_{G'})} 
		\right)^{d/2\alpha} \exp\left( - \tau(\alpha) \left(\log \frac{1}{\vba(p_G,p_{G'})}\right)^{1/ \alpha}\right) \nonumber,
\end{eqnarray}
where $\tau(\alpha) = \sqrt{\frac{2}{\lambda\lambda_{\max}}}\bigg /\left[\sqrt{\frac{2}{\lambda\lambda_{\max}}} + 
\sqrt{\frac{2}{\lambda\lambda_{\min}}} +C_0\right]^{1/\alpha}$. Now, the 
location family of multivariate Laplace distributions pertains to 
the ordinary smooth likelihood families. Therefore, from part (1) 
of Theorem 2 in~\cite{Nguyen-13}, it follows that for any $m < 
4/(4+5d)$, $W_2^2(G,G') \leq V(p_G,p_{G'})^{m}$. We note in passing that improved rates for other choices of $W_r$ may be possible by utilizing techniques similar to~\cite{Gao-vdV-2016}.  Thus, by means of the inequality $\vba(p_G,p_{G'}) \lesssim \hba(p_G,p_{G'})$, the following inequality holds 
\begin{eqnarray*}
 \left(\log\frac{1}{\hba(p_G,p_{G'})} 
		\right)^{\frac{d}{2\alpha}} \exp\left( - \tau(\alpha) \left(\log \frac{1}{\hba(p_G,p_{G'})}\right)^{1/\alpha}\right) \gtrsim  {W_2^{2/m}(G,G')} .
\end{eqnarray*}
The result now follows by taking logarithms of both sides.

\subsubsection*{Acknowledgements}
This research is supported in part by grants NSF CAREER DMS-1351362, NSF CNS-1409303, a research
gift from Adobe Research and a Margaret and Herman Sokol Faculty Award.

\bibliography{Nhat,NPB,Nguyen,Aritra}
\newpage
\section{Appendix A: Weighted Hellinger and Wasserstein distance}
\label{section:appendix_A}
In this appendix, we will establish several useful bounds between 
weighted Hellinger distance and Wasserstein metric that are employed in the proofs for misspecified settings of Section~\ref{section:Posterior Contraction:mis-specified}. See the formal setup of $G_0,f_0$ and $G_*$ in the beginning of that section. 

First, we start 
with the following lemma regarding an upper bound of weighted 
Hellinger distance in terms of Wasserstein metric when the kernel 
$f$ satisfies first order integral Lipschitz condition. 
\begin{lemma} \label{lemma:upper_bound_weighted_Hellinger_Wasserstein}
Assume that the kernel $f$ is integral Lipschitz up to the 
first order. Then, for any mixing measure $G_1$ and $G_2$ in 
$\Pcal(\Theta)$, there exists a positive constant 
$\overline{C}(\Theta)$ depending only on $\Theta$ such that
\begin{eqnarray}
\hba^{2}(p_{G_{1}},p_{G_{2}}) 
		\leq \overline{C}(\Theta) W_1(G_1,G_2). \nonumber 
\end{eqnarray}
\end{lemma}
\begin{proof}
Denote the weighted total variation distance as follows
\begin{eqnarray}
\overline{V}(p_{G_{1}},p_{G_{2}}) 
		= \dfrac{1}{2} \int |p_{G_{1}}(x) - p_{G_{2}}(x)|
		\dfrac{p_{G_0,f_0}(x)}{p_{G_{*}}(x)}\mathrm{d}\mu(x) \nonumber
\end{eqnarray}
for any $G_{1}, G_{2} \in \Pcal(\Theta)$. 
It is clear that $\hba^{2}(p_{G_{1}},p_{G_{2}}) \leq \overline{V}
(p_{G_{1}},p_{G_{2}})$ for any $G_{1},G_{2} \in \Pcal(\Theta)$. 

For any coupling $\vec{q}$ of the weight vectors of 
$G_{1} = \sum \limits_{i=1}^{k_{1}}{p_{i,1}\delta_{\theta_{i,1}}}$ 
and $G_{2} = \sum \limits_{i=1}^{k_{2}}{p_{i,2}\delta_{\theta_{i,2}}}$, 
we can check via triangle inequality that
\begin{eqnarray}
\overline{V}(p_{G_{1}},p_{G_{2}}) 
		& \leq & \dfrac{1}{2}\int \sum \limits_{i,j} q_{ij}|f(x|
		\theta_{i,1})-f(x|\theta_{j,2})|\dfrac{p_{G_0,f_0}(x)}
		{p_{G_{*}}(x)}\mathrm{d}\mu(x) \nonumber \\
		& \leq & \overline{C}(\Theta)\sum \limits_{i,j} q_{ij}\|
		\theta_{i,1}-\theta_{j,2}\| \nonumber
\end{eqnarray}
where the existence of positive constant $\overline{C}(\Theta)$ in 
the second inequality is due to the first order integral 
Lipschitz property of $f$. The above result implies that
\begin{eqnarray}
\hba^{2}(p_{G_{1}},p_{G_{2}}) 
		\leq \overline{V}(p_{G_{1}},p_{G_{2}}) \leq \overline{C}
		(\Theta) W_{1}(G_{1},G_{2}) \nonumber
\end{eqnarray}
for any $G_{1}, G_{2} \in \Pcal(\Theta)$. We achieve the 
conclusion of the lemma.
\end{proof}
\subsection{Proof of Lemma~\ref{lemma:misspecified_optimal_unique}}
\label{subsection:proof_lemma_misspecified_optimal}
The proof is a straightforward application of Lemma \ref{lemma:misspecified_optimal}. In fact, from that lemma, we have
\begin{eqnarray}
2 \leq {\displaystyle \int \biggr(\dfrac{p_{G_{1,*}}(x)}{p_{G_{2,*}}(x)}+\dfrac{p_{G_{2,*}}(x)}{p_{G_{1,*}}(x)}\biggr) p_{G_0,f_0}(x) \textrm{d}\mu(x) \leq 2} \nonumber
\end{eqnarray}
where the first inequality is due to Cauchy inequality. The above inequality holds only when $p_{G_{1,*}}(x)=p_{G_{2,*}}(x)$ for almost all $x \in \mathcal{X}$, which concludes our lemma. 
\subsection{Proof of Proposition~\ref{proposition:weighted_Hellinger_Wasserstein_metric}}
\label{subsection:proof_proposition_weighted_hellinger}
Denote the weighted total variation distance as follows
\begin{eqnarray}
\overline{V}(p_{G_{1}},p_{G_{2}}) = \dfrac{1}{2} \int |p_{G_{1}}(x) - p_{G_{2}}(x)|\dfrac{p_{G_0,f_0}(x)}{p_{G_{*}}(x)}\mathrm{d}\mu(x) \nonumber
\end{eqnarray}
for any $G_{1}, G_{2} \in \Gcal$. Then, by means of Holder's inequality, we can 
verify that
\begin{eqnarray}
\overline{V}(p_{G_{1}},p_{G_{2}}) & \leq & \sqrt{2} \overline{h}(p_{G_{1}},p_{G_{2}}) \biggr(\int (\sqrt{p_{G_{1}}(x)} + \sqrt{p_{G_{2}}(x)})^{2}\dfrac{p_{G_0,f_0}(x)}{p_{G_{*}}(x)}\mathrm{d}\mu(x)\biggr)^{1/2} \nonumber \\
& \leq & 2\sqrt{2} \overline{h}(p_{G_{1}},p_{G_{2}})  \nonumber
\end{eqnarray}
where the last inequality is due to Lemma \ref{lemma:misspecified_optimal}. 
Therefore, to obtain the conclusion of the proposition, it is sufficient to 
demonstrate that 
\begin{eqnarray}
\inf \limits_{G \in \Ocal_{\overline{k}}} \overline{V}(p_{G},p_{G_{*}}) / 
W_{2}^{2}(G,G_{*}) > 0 \label{eq:proof_proposition_weighted_Hellinger_Wasserstein_first}
\end{eqnarray}
where $\overline{k} > k_{*}$.
Firstly, we will show that
\begin{eqnarray}
\lim \limits_{\epsilon \to 0} \inf \limits_{G_{*} 
		\in \Ocal_{\overline{k}}} \left\{\dfrac{\overline{V}
		(p_{G},p_{G_{*}})}{W_{2}^{2}(G,G_{*})}: W_{2}(G,G_{*}) 
		\leq \epsilon \right\} >0. \nonumber
\end{eqnarray}
Assume that the above inequality does not hold. It implies that there exists a sequence of $G_{n} \in \Ocal_{\overline{k}}(\Theta)$ such that $\overline{V}
(p_{G_{n}},p_{G_{*}})/W_{2}^{2}(G_{n},G_{*}) \to 0$ as $n \to \infty$. By means 
of Fatou's lemma, we have
\begin{eqnarray}
 0 & = & \mathop {\lim \inf} \limits_{n \to \infty} 
 \dfrac{\overline{V}(p_{G_{n}},p_{G_{*}})}{W_{2}^{2}(G_{n},G_{*})}  
 \geq  \dfrac{1}{2} \int \mathop {\lim \inf} \limits_{n \to 
 \infty} \dfrac{|p_{G_{n}}(x) -  p_{G_{*}}(x)|\dfrac{p_{G_0,f_0}(x)}
 {p_{G_{*}}(x)}}{W_{2}^{2}(G_{n},G_{*})}\mathrm{d}\mu(x). \nonumber
\end{eqnarray}
Hence, for almost surely $x \in \mathcal{X}$, we obtain that
\begin{eqnarray}
\mathop {\lim \inf} \limits_{n \to \infty} \dfrac{|p_{G_{n}}(x) -  p_{G_{*}}(x)|\dfrac{p_{G_0,f_0}(x)}{p_{G_{*}}(x)}}{W_{2}^{2}(G_{n},G_{*})} = 0. \nonumber
\end{eqnarray}
The above equality is equivalent to
\begin{align}
\mathop {\lim \inf} \limits_{n \to \infty} \dfrac{|p_{G_{n}}(x) -  p_{G_{*}}(x)|}{W_{2}^{2}(G_{n},G_{*})} = 0 \nonumber
\end{align}
for almost surely $x \in \mathcal{X}$. However, using the same argument as that of Theorem 3.2 in~\cite{Ho-Nguyen-EJS-16}, the above equality cannot hold due to 
the second order identifiability of $f$, which is a contradiction. Therefore, we can 
find positive constant $\epsilon_{0}>0$ such that as long as $W_{2}(G,G_{*}) 
\leq \epsilon_{0}$, we achieve that $\overline{V}(p_{G},p_{G_{*}}) \gtrsim 
W_{2}^{2}(G,G_{*})$. As a consequence, to obtain the conclusion of 
\eqref{eq:proof_proposition_weighted_Hellinger_Wasserstein_first}, we only 
need to verify that
\begin{eqnarray}
\inf \limits_{G \in \Ocal_{\overline{k}}: \ W_{2}(G,G_{*})>\epsilon_0} \overline{V}(p_{G},p_{G_{*}})/W_{2}^{2}(G,G_{*}) >0. \nonumber
\end{eqnarray}
Assume that the above result does not hold. It implies that we can find a 
sequence of $G_{n} \in \Ocal_{\overline{k}}$ such that $W_{2}(G_{n},G_{*}) >
\epsilon_{0}$ and $\overline{V}(p_{G_{n}},p_{G_{*}})/W_{2}^{2}(G_{n},G_{*}) 
\to 0$ as $n \to \infty$. Since $\Theta$ is a bounded subset of $\mathbb{R}
^{d}$, we can find a subsequence of $G_{n}$ such that $W_{1}(G_{n},
\overline{G}) \to 0$ for some $\overline{G} \in \Ocal_{\overline{k}}$ such that 
$W_{2}(\overline{G},G_{*}) \geq \epsilon_{0}$. From our hypothesis, we will 
have that $\overline{V}(p_{G_{n}},p_{G_{*}}) \to 0$. However, by virtue of 
Fatou's lemma, we obtain that
\begin{eqnarray}
0 = \mathop {\lim \inf} \limits_{n \to \infty} \overline{V}(p_{G_{n}},p_{G_{*}}) \geq \overline{V}(p_{\overline{G}},p_{G_{*}}). \nonumber
\end{eqnarray}
The above equation leads to $p_{\overline{G}}(x) = p_{G_{*}}(x)$ 
for almost surely $x \in \mathcal{X}$. Due to the identifiability 
of $f$, the previous equation leads to $\overline{G} \equiv G_{*}$, 
which is a contradiction to the assumption that $W_{2}
(\overline{G},G_{*}) \geq \epsilon_{0}$. We obtain the conclusion 
of the proposition.

\section{Appendix B: Posterior contraction under misspecification}
\label{section:appendix_B}
This appendix is devoted to the description of a general method for establishing posterior convergence rates of mixing measures under misspecified settings, extending the methods of~\cite{Kleijn-2006} and~\cite{Nguyen-13}. 
 Once the general method is fully developed we shall be ready to complete the proofs of the main theorems of Section~\ref{section:Posterior Contraction:mis-specified}, which are given in Section~\ref{section:appendix_C}. Recall the weighted Hellinger distance defined in~\eqref{definition:weighted_Hellinger}, which leads to the following definition. 
\begin{definition} \label{definition:hellinger_to_Wasserstein}
For any set $\Scal \subset \overline{\Gcal}$, define a real-valued function $\Psiba_{\Scal}: \mathbb{R} \to \mathbb{R}^{+}$ as follows
\begin{eqnarray}
\Psiba_{\Scal}(r)= \inf \limits_{G \in \Scal: \ 
		W_{2}(G,G_{*}) \geq r/2} \hba^{2}(p_{G},p_{G_{*}}) \nonumber
\end{eqnarray}
for any $r \in \mathbb{R}$.
\end{definition}
A key ingredient to establishing the posterior contraction bounds is through the existence of tests for subsets of parameters of interest. In the model misspecification setting, it is no longer appropriate to test any mixing measure $G$ against true measure $G_{0}$. Instead, following~\cite{Kleijn-2006}, it is appropriate to test any mixing measure $G$ against $G_{*}$, which is ultimately achieved by testing $\dfrac{p_{G_{0},f_0}}{p_{G_{*}}}p_{G}$ against $p_{G_{0},f_0}$.
This insight leads us to the following crucial result regarding 
the existence of test for discriminating $G_{*}$ against a closed 
Wasserstein metric ball centered at $G_{1}$ for any fixed pair of mixing measures $(G_{*},G_{1})$.
\begin{lemma} \label{lemma:existence_test_misspecified_kernel}
Consider $\Scal \in \overline{\Gcal}$ such that $G_{*} \in \Scal$. 
Given $G_{1} \in \Scal$ such that $W_{2}(G_{1},G_{*}) \geq r$ for some 
$r > 0$. Assume that either one of the following two sets of 
conditions holds:
\begin{itemize}
\item[(1)] $\mathcal{S}$ is a convex set, in which case, let 
$\overline{M}(\Scal,G_{1},r)  = 1$.
\item[(2)] $\mathcal{S}$ is a nonconvex set. In addition, $f$ has 
first order integral Lipschitz property. In this case, we 
define that
\begin{eqnarray}
\overline{M}(\Scal,G_{1},r) 
		= D\biggr( \dfrac{\Psiba_{\Scal}(r)}{8\overline{C}
		(\Theta)}, \Scal \cap B_{W_{2}}(G_{1},r/2), W_{2}\biggr). \nonumber
\end{eqnarray}
\end{itemize}
\noindent
Then, there exists tests $\phi_{n}$ such that
\begin{eqnarray}
P_{G_{0},f_0} \phi_{n} 
		& \leq & \overline{M}(\Scal,G_{1},r) \exp(- n 
		\Psiba_{\Scal}(r)/8), \label{eq:existence_test_equ_first} \\
\sup \limits_{G \in \Scal \cap B_{W_{2}}(G_{1},r/2)} \dfrac{P_{G_{0},f_0}}{P_{G_{*}}}P_{G} (1-\phi_{n}) 
		& \leq & \exp(- n \Psiba_{\Scal}(r)/8). \label{eq:existence_test_equ_second}
\end{eqnarray}
\end{lemma}
By means of the existence of tests in Lemma~\ref{lemma:existence_test_misspecified_kernel}, 
we have the following result regarding testing $G_{*}$ versus a 
complement of a closed Wasserstein ball.
\begin{lemma} \label{lemma:test_complement_ball}
Assume that all the conditions in Lemma~\ref{lemma:existence_test_misspecified_kernel} hold. 
Let $D(\epsilon)$ be a non-decreasing function such that , for 
some $\epsilon_{n} \geq 0$ and every $\epsilon>\epsilon_{n}$, 
\begin{eqnarray}
\sup \limits_{G \in \Scal} \overline{M}(\Scal,G,r) 
		D(\epsilon/2, \Scal \cap B_{W_{2}}(G_{*},2\epsilon) 
		\backslash B_{W_{2}}(G_{*},\epsilon), W_{2}) \leq 
		D(\epsilon). \nonumber
\end{eqnarray}
Then, for every $\epsilon>\epsilon_{n}$ there exist tests 
$\phi_{n}$ (depending on $\epsilon>0$) such that, for every $J \in 
\mathbb{N}$,
\begin{eqnarray}
P_{G_{0},f_0} \phi_{n} 
		& \leq & D(\epsilon) \sum \limits_{t=J}^{[\text{Diam}
		(\Theta)/\epsilon]}\exp(- n \Psiba_{\Scal}(t\epsilon)/8), \label{eq:test_complement_ball_first} \\
\sup \limits_{G \in \Scal: W_{2}(G,G_{*})> J\epsilon} 
		\dfrac{P_{G_{0},f_0}}{P_{G_{*}}}P_{G} (1-\phi_{n}) & \leq & 
		\exp(- n \Psiba_{\Scal}(J\epsilon)/8). \label{eq:test_complement_ball_second}
\end{eqnarray}
\end{lemma}

For any $\epsilon>0, M>0$, we define a generalized Kullback-Leibler neighborhood of $G_{*}$ by
\begin{eqnarray}
B_{K}^{*}(\epsilon, G_{*},P_{G_{0},f_0},M) & : = & \biggr\{G \in \overline{\Gcal}: 
		\ -P_{G_{0},f_0} \log \dfrac{p_{G}}{p_{G_{*}}} \leq 
		\epsilon^{2}\log{M/\epsilon} +\epsilon, \nonumber \\
		& & \hspace{ 6 em} P_{G_{0},f_0}
		\biggr(\log \dfrac{p_{G}}{p_{G_{*}}}\biggr)^{2} \leq 
		\epsilon^{2} \left( \log(M/\epsilon)\right)^2 \biggr\}. \label{eq:KL_neighborhood_misspecified}
\end{eqnarray}
Invoking the results in Lemma~\ref{lemma:existence_test_misspecified_kernel} and Lemma~\ref{lemma:test_complement_ball}, we have the following theorem establishing posterior 
contraction rate for $G_{*}$. This theorem generalizes
Theorem 3 in~\cite{Nguyen-13} to the misspecified setting.
\begin{theorem} \label{theorem:posterior_contraction_rate_misspecified_convex}
Suppose that for a sequence of $\left\{\epsilon_{n}\right\}_{n \geq 1}$ 
that tends to a constant (or $0$) such that $n\epsilon_{n}^{2} \to 
\infty$, and constants $C,M>0$, and convex sets $\Gcal_{n} \subset 
\overline{\Gcal}$, we have
\begin{eqnarray}
\log D(\epsilon_{n},\Gcal_{n},W_{2}) \leq n \epsilon_{n}^2, \label{eq:posterior_contraction_rate_convex_first_condition} \\
\Pi(\overline{\Gcal} \backslash \Gcal_{n}) \leq \exp(-n(\epsilon_n^{2}\log(M/\epsilon_n) + \epsilon_n)(C+4)) , \label{eq:posterior_contraction_rate_convex_second_condition} \\
\Pi(B_{K}^{*}(\epsilon_{n}, G_{*},P_{G_{0},f_0},M)) \geq \exp(-n(\epsilon_n^{2}\log(M/\epsilon_n) + \epsilon_n) C), \label{eq:posterior_contraction_rate_convex_third_condition}
\end{eqnarray}
Additionally, $M_{n}$ is a sequence such that
\begin{eqnarray}
\Psiba_{\Gcal_{n}}(M_{n}\epsilon_{n}) \geq 8(\epsilon_n^{2}\log(M/\epsilon_n) + \epsilon_n)(C+4), \label{eq:posterior_contraction_rate_convex_fourth_condition} \\
\exp(2n(\epsilon_n^{2}\log(M/\epsilon_n) + \epsilon_n))\sum \limits_{j \geq M_{n}} \exp(-n\Psiba_{\Gcal_{n}}(j\epsilon_{n})/8) \to 0. \label{eq:posterior_contraction_rate_convex_fifth_condition}
\end{eqnarray}
Then, $\Pi(G \in \overline{\Gcal}: \ W_{2}(G,G_{*}) \geq M_{n}\epsilon_{n}|X_{1},\ldots,X_{n}) \to 0$ in $P_{G_{0},f_0}$-probability. 
\end{theorem}
The above theorem is particularly useful for establishing the 
convergence rate of $G_{*} \in \Pcal(\Theta)$ for which the suitable sieves $\Gcal_n$ of mixture densities are convex classes of functions. 
In the situation where $\Gcal_n$ are non-convex, we need the following
result, which is the generalization of Theorem 4 in~\cite{Nguyen-13}. 

\begin{theorem} \label{theorem:posterior_contraction_rate_misspecified}
Assume that $f$ admits the first order integral Lipschitz 
property. Additionally, there is a sequence $\epsilon_{n}$ with $
\epsilon_{n} \to 0$ such that $n\log(1/\epsilon_n)^{-2}$ is 
bounded away from 0, a sequence $M_{n}$, a constant $M>0$, and a sequence of sets $
\Gcal_{n} \subset \overline{\Gcal}$ such that the following 
conditions hold
\begin{eqnarray}
\log D(\epsilon/2, \Gcal_{n} \cap B_{W_{2}}(G_{*},2\epsilon) 
		\backslash B_{W_{2}}(G_{*},\epsilon), W_{2}) + \sup 	
		\limits_{G \in \Gcal_{n}} \log \overline{M}(\Gcal_{n},G,r) 
		\leq n\epsilon_{n}^2 \ \forall \ \epsilon \geq \epsilon_{n}, \label{eq:posterior_contraction_rate_first_condition} \\
\dfrac{\Pi(\overline{\Gcal} \backslash \Gcal_{n})}{\Pi(B_{K}^{*}(\epsilon_{n},G_{*},P_{G_{0},f_0},M))}=o\left(\exp \left(-2n \left( 
		\epsilon_{n}^{2}\log \left(\frac{M}{\epsilon_n}\right) + 
		\epsilon_n \right) \right)\right), \label{eq:posterior_contraction_rate_second_condition} \\
\dfrac{\Pi(B_{W_{2}}(G_{*},2j\epsilon_{n}) \backslash B_{W_{2}}
		(G_{*},j\epsilon_{n}))}
		{\Pi(B_{K}^{*}(\epsilon_{n},G_{*},P_{G_{0},f_0},M))} \leq \exp
		\biggr(n\Psiba_{\Gcal_{n}}(j\epsilon_{n})/16\biggr), \ 
		\forall j \geq M_{n} \label{eq:posterior_contraction_rate_third_condition} \\
\exp \left(2n \left( \epsilon_{n}^{2}\log \left(\frac{M}
		{\epsilon_n}\right) + \epsilon_n \right) \right)\sum 
		\limits_{j \geq M_{n}}{\exp\left(-n\Psiba_{\Gcal_{n}}(j
		\epsilon_{n})/16\right)} \to 0. \label{eq:posterior_contraction_rate_fourth_condition}
\end{eqnarray}
Then, we have that $\Pi(G \in \overline{\Gcal}: W_{2}(G,G_{*}) 
\geq M_{n}\epsilon_{n}|X_{1},\ldots,X_{n}) \to 0$ in $P_{G_{0},f_0}$-
probability.
\end{theorem}

\section{Appendix C}
\label{section:appendix_C}
We are now ready to complete the proof of the main posterior contraction theorems stated in Section~\ref{section:Posterior Contraction:mis-specified}.


\subsection{Proof of Theorem~\ref{theorem:posterior_rate_Gaussian}}
\label{ssub:proof_theorem_posterior_rate_Gaussian}

Note that the MFM prior places full mass on discrete measure with finite support, it is enough to show $\Pi\biggr(G \in \Gcal(\Theta): W_{2}(G,G_*) 
		\gtrsim  \left(\frac{\log \log n}{\log n}\right)^{1/2}\biggr|
		X_{1},\ldots,X_{n}\biggr) \to 0$.

The proof of this result is a straightforward application of 
Proposition~\ref{proposition:lower_bound_weighted_Hellinger_Wasserstein_infinite_Gaussian}, Lemma~\ref{lemma:Prior_mass_MFM}, 
and Theorem~\ref{theorem:posterior_contraction_rate_misspecified_convex}; 
therefore, we will only provide a sketch of this proof. Similar to 
the proof of Theorem~\ref{theorem:convergence_rate_well_specified} (for the well-specified setting), we proceed by constructing a 
sequence $\epsilon_n$ and sieves $\Gcal_n$ that satisfy all the 
conditions specified in Theorem~\ref{theorem:posterior_contraction_rate_misspecified_convex}.

\paragraph{Step 1:}  First, we choose $\epsilon_n$ to satisfy condition~\eqref{eq:posterior_contraction_rate_convex_third_condition} 
in Theorem~\ref{theorem:posterior_contraction_rate_misspecified_convex}. 
To that effect, we proceed by making  use of the results from 
Lemma 8.1 in~\cite{Kleijn-2006}. 
In particular, from Lemma 8.1 of~\cite{Kleijn-2006}, as long as $P$ 
is a probability measure and $Q$ is a finite measure (with 
densities $p$ and $q$ respectively, with respect to Lebesgue 
measure on $\mathbb{R}^d$) such that $h(p,q) \leq \epsilon$ and $
\int p^{2}/q \leq M$, we obtain that 
\begin{eqnarray}
P\log(p/q) & \lesssim & \epsilon^{2}\log(M/\epsilon)+\|p-q\|_{1}, \nonumber \\
P(\log(p/q))^{2} & \lesssim & \epsilon^{2}(\log(M/\epsilon))^{2},  
\end{eqnarray} 
where the constants in these bounds are universal. For the purpose 
of our proof, we will choose $p = p_{G_{0},f_0}$ and 
$q = p_{G}p_{G_{0},f_0}/p_{G_{*}}$.

Since the Gaussian kernel satisfies the integral 
Lipschitz property up to the first order, by invoking the result of 
Lemma~\ref{lemma:upper_bound_weighted_Hellinger_Wasserstein}, we have
\begin{eqnarray}
\|p-q\|_1= \biggr\|\frac{p_{G_0,f_0}}{p_{G_*}}(p_G-p_{G_*}) \biggr\|_{1} \lesssim W_1(G,G_*) \leq W_2(G,G_*).
\end{eqnarray}
Additionally, for the Gaussian kernel $f$ given by~\eqref{eqn:location_Gaussian_formulation}, we find that 
\begin{eqnarray}
\overline{h}^2(p_{G},p_{G_{*}}) \leq C_{1}W_{2}^2(G,G^{*})
\end{eqnarray} 
for any $G \in \overline{\mathcal{G}}(\Theta)$. For Gaussian location mixtures as long as there exists 
$\epsilon_{0}>0$ such that $W_{2}(G,G_{*}) \leq \epsilon_{0}$, 
we also can check that $\int p_{G_0,f_0}(x)p_{G_{*}}(x)/p_{G}(x)\mathrm{d}\mu(x) \leq M^{*}(\epsilon_{0})$ for some positive constant $M^{*}(\epsilon_{0})$ depending only on $\epsilon_{0}, G_{0}$, and $\Theta$. Therefore, as long as $W_{2}(G,G_{*}) \leq \epsilon \leq \epsilon_0$, 
we have  for mixing measure $G$, 
\begin{eqnarray}
 - P_{G_0,f_0} \log(p_{G}/p_{G_{*}}) 
 		& \leq & \epsilon^2(\log(M/\epsilon))  + \epsilon, \nonumber \\
   P_{G_0,f_0} [\log(p_{G}/p_{G_{*}})]^2 
		& \leq & \epsilon^2(\log(M/\epsilon))^2 \nonumber,
\end{eqnarray} 
 where $M := M^{*}(\epsilon_0)$. The constants in the bounds are all universal.

Governed by this result, we can write 
\begin{eqnarray}
\Pi(B_{K}^{*}(\epsilon_{n},G_*,P_{G_{0},f_0},M)) 
		\geq \Pi(W_2(G,G_*)\lesssim \epsilon_n) \nonumber
\end{eqnarray} 
for any sequence $\epsilon_n \leq \epsilon_0$.
Now, the packing number $D=D(\epsilon_n)$ (with packing radius $\epsilon_n$) in Lemma~\ref{lemma:Prior_mass_MFM} 
satisfies $D(\epsilon_n) \asymp \left(\frac{\text{Diam}(\Theta)}
{\epsilon_n}\right)^d$. Following the result in Lemma~\ref{lemma:Prior_mass_MFM}, we have with $r=2$ that 
\begin{eqnarray}
\log(\Pi(B_{K}^{*}(\epsilon_{n},G_*,P_{G_{0},f_0}.M))) 
		& \gtrsim & D(\epsilon_n) (\log c_0 + \log(\epsilon_n/
		D(\epsilon_n))) \nonumber \\
		& & \hspace{ 2 em} + \log(\epsilon_n/
		D(\epsilon_n))( 1 + (1+(2/d))\gamma(D(\epsilon_n)-1)/
		D(\epsilon_n) ) \nonumber.
\end{eqnarray} 

With $\epsilon_n \asymp (\frac{\log n}{n})^{1/(d+2)} $, one can check that condition \eqref{eq:posterior_contraction_rate_convex_third_condition} and ~\eqref{eq:posterior_contraction_rate_convex_first_condition} hold. 

\paragraph{Step 2:} Note that condition \eqref{eq:posterior_contraction_rate_convex_second_condition} holds automatically since we take $\mathcal{G}_n=\overline{\mathcal{G}}$, while condition \eqref{eq:posterior_contraction_rate_convex_first_condition} follows from Lemma 4 in~\cite{Nguyen-13}. 

\paragraph{Step 3:} Next we will show condition \eqref{eq:posterior_contraction_rate_convex_fourth_condition} and condition \eqref{eq:posterior_contraction_rate_convex_fifth_condition} for some appropriate choice of $M_n$ for the $\epsilon_n$ considered in Step 1. 

Following proposition \ref{proposition:lower_bound_weighted_Hellinger_Wasserstein_infinite_Gaussian} 
we know that $\Psiba_{\Gcal_{n}}(r) \gtrsim \exp\biggr(-(1 + 
(2\lambda_{max}/\lambda_{min}))/r^{2}\biggr)$. Using this fact, we 
can check to see that $M_n$ such that $M_n \epsilon_n \approx 
\left(\frac{\log \log(n)}{\log n}\right)^{1/2}$ works, with $
\epsilon_n \approx (\frac{\log n}{n})^{1/(d+2)}$.
\subsection{Proof of Lemma~\ref{lemma:test_complement_ball}}
\label{subsection:proof_lemma_test_complement}
Consider a $ t\epsilon/2$ maximal-packing of the set 
$B_{W_{1}}(G_{*},2t\epsilon) \backslash B_{W_{1}}(G_{*},t\epsilon)$. 
Let $S_t$ be the corresponding set of $D(t\epsilon/2,\Scal \cap 
B_{W_{1}}(G_{*},2t\epsilon) \backslash B_{W_{1}}(G_{*},t\epsilon)$  
points obtained there in. Then as in Lemma~\ref{lemma:existence_test_misspecified_kernel} 
corresponding to  each point $G_1$ in $S_t$,  there exist $
\omega_{n,t}$ which satisfies~\eqref{eq:existence_test_equ_first} 
and~\eqref{eq:existence_test_equ_second}. 
Then taking  $\phi_n$ as the supremum over all these tests $
\omega_{n,t}$ over all points in $S_t$ , all $t\geq J$  we see 
that by the union bound,
\begin{eqnarray}
P_{G_0,f_0} \phi_n 
		& \leq & \sum_{t>J} \sum_{G_1 \in S_t}\overline{M}
		(\Scal,G_{1},t\epsilon) \exp(-n\Psiba_{\Scal}(t\epsilon)/8), \nonumber \\
		&\leq & D(\epsilon)\sum \limits_{t=J}^{[\text{Diam}
		(\Theta)/\epsilon]}\exp(- n \Psiba_{\Scal}(t\epsilon)/8) \nonumber \\
		\sup \limits_{G \in  \cup_{t \geq J }\{B_{W_{1}}(G_{*},2t
		\epsilon) \backslash B_{W_{1}}(G_{*},t\epsilon)\} } 
		\dfrac{P_{G_0,f_0}}{P_{G_{*}}}P_{G} (1-\phi_n) & \leq & \sup 
		\limits_{t \geq J}\exp(-n\Psiba_{\Scal}(t\epsilon)/8)  
		\leq \exp(-n\Psiba_{\Scal}(J\epsilon)/8). \nonumber
\end{eqnarray}
The last inequality follows from the fact that $\Psiba_{\Scal}
(\cdot)$ is an increasing function in its argument.

\subsection{Proof of Theorem~\ref{theorem:posterior_contraction_rate_misspecified}}
\label{subsection:proof_theorem_posterior_contraction_misspecified}
The following lemma is analogous to Lemma 7.1 in~\cite{Kleijn-2006} and Lemma 8.1 in~\cite{Ghosal-Ghosh-vanderVaart-00} and can be similarly proved.
\begin{lemma}
For every $M,\epsilon>0$, $C>0$, and probability measure $\Pi$ on $G$, we obtain that
\begin{eqnarray}
& & \hspace{- 1 em} P_{G_0,f_0}\biggr(\int \prod \limits_{i=1}^{n}{\dfrac{p_{G}(X_{i})}{p_{G_{*}}(X_{i})}d\Pi(G)} 
		\leq \Pi(B_{K}^{*}(\epsilon, G_{*},P_{G_{0},f_0},M))\exp(-
		(1+C)n(\epsilon^{2}\log{(M/\epsilon)} +\epsilon)\biggr) \nonumber \\
		& & \hspace{26 em} \leq \dfrac{\log^2(M/\epsilon)}{C^2n\left(1+ \epsilon\log(M/\epsilon)\right)^2}. \nonumber
\end{eqnarray}
\end{lemma}

Equipped with this lemma we can now prove the theorem as follows.
Denote $A_{n}$ the event such that 
\begin{align*}
{\displaystyle \int \prod \limits_{i=1}^{n}{\dfrac{p_{G}(X_{i})}
{p_{G_{*}}(X_{i})}d\Pi(G)} \leq \Pi(B_{K}^{*}
(\epsilon,G_{*},P_{G_{0},f_0},M))\exp(-n(\epsilon^{2}\log(M/\epsilon) +
\epsilon)(1+C))}.
\end{align*}
The above result indicates that $P_{G_0,f_0} 
\mathbbm{1}_{A_{n}}  \leq (C^{2}n)^{-1}\log(M/\epsilon)^{2}$ for 
any $\epsilon>0$ and $C>0$.

For any sequence $\epsilon_{n}$,  we denote $\mathcal{U}_{n}=\left
\{G \in \overline{\Gcal}: W_{2}(G,G_{*}) \geq M_{n}\epsilon_{n}\right\}$ 
and $S_{n,j}=\left\{G \in \Gcal_{n}: \right. \\ \left. 
W_{2}(G,G_{*}) \in [j\epsilon_{n},(j+1)\epsilon_{n})\right\}$ for 
any $j \geq 1$. From the result of Lemma~\ref{lemma:existence_test_misspecified_kernel} 
and condition~\eqref{eq:posterior_contraction_rate_misspecified_first}, 
there exists a test $\phi_{n}$ such that inequality ~\eqref{eq:test_complement_ball_first} and~\eqref{eq:test_complement_ball_second} 
hold when $D(\epsilon_{n})=\exp(n\epsilon_{n}^{2})$. Now, we have
\begin{eqnarray}
& & \hspace{-2 em} P_{G_0,f_0} \Pi(G \in \mathcal{U}_{n}|X_{1},
		\ldots,X_{n}) \nonumber \\
& & \hspace{1 em} = P_{G_0,f_0}\phi_{n}\Pi(G \in \mathcal{U}_{n}|
		X_{1},\ldots,X_{n}) +P_{G_0,f_0}(1-\phi_{n})\mathbbm{1}_{A_{n}}
		\Pi(G \in \mathcal{U}_{n}|X_{1},\ldots,X_{n}) \nonumber \\
& & \hspace{1 em} + P_{G_0,f_0}(1-\phi_{n})\mathbbm{1}_{A_{n}^{c}}			\Pi(G \in \mathcal{U}_{n}|X_{1},\ldots,X_{n}) \nonumber \\
& & \hspace{1 em} \leq P_{G_0,f_0}\phi_{n}+P_{G_0,f_0}\mathbbm{1}_{A_{n}}
		+P_{G_0,f_0}(1-\phi_{n})\mathbbm{1}_{A_{n}^{c}}\Pi(G \in 
		\mathcal{U}_{n}|X_{1},\ldots,X_{n}). \label{eq:posterior_contraction_rate_misspecified_zero}
\end{eqnarray}
According to Lemma~\ref{lemma:test_complement_ball}, 
we have $P_{G_0,f_0}\phi_{n} \leq \exp(n\epsilon_{n}^{2})
\sum \limits_{j \geq M_{n}}{\exp\left(-n\Psiba_{\Gcal_{n}}(j
\epsilon_{n})/8\right)} \to 0$, 
which is due to condition~\eqref{eq:posterior_contraction_rate_fourth_condition}. 
Additionally, from the formation of $A_{n}$, we also obtain that 
$P_{G_0,f_0}\mathbbm{1}_{A_{n}} \leq (C^{2}n)^{-1}\log(M/\epsilon)^{2}$. 
If $n\log(1/\epsilon)^{-2}\to \infty$, it is clear that $P_{G_0,f_0}
\mathbbm{1}_{A_{n}} \to 0$ for any $C \geq 1$. If $n\log(1/
\epsilon)^{-2}$ does not tend to $\infty$ but is bounded away from 0, 
then we can choose $C>0$ large enough such that 
$P_{G_{0}}\mathbbm{1}_{A_{n}}$ is sufficiently close to 0. 
Therefore, the first two terms in~\eqref{eq:posterior_contraction_rate_misspecified_first} 
can always be made to vanish to 0. To achieve the conclusion of 
the theorem, it is sufficient to demonstrate that the third term 
in~\eqref{eq:posterior_contraction_rate_misspecified_zero} 
goes to 0. In fact, we have the following equation
\begin{eqnarray}
\Pi(G \in \mathcal{U}_{n}|X_{1},\ldots,X_{n}) 
		= \biggr(\int \limits_{\mathcal{U}_{n}}{\prod 
		\limits_{i=1}^{n}{\dfrac{p_{G}(X_{i})}{p_{G_{*}}(X_{i})}d
		\Pi(G)}}\biggr)/\biggr(\int {\prod \limits_{i=1}^{n}
		{\dfrac{p_{G}(X_{i})}{p_{G_{*}}(X_{i})}d\Pi(G)}}\biggr). \nonumber
\end{eqnarray}
From the formulation of $A_{n}$, we have
\begin{eqnarray}
\label{eq:posterior_contraction_rate_misspecified_first}
& & \hspace{-4em} P_{G_0,f_0}(1-\phi_{n})\mathbbm{1}_{A_{n}^{c}}	
		\Pi(G \in \mathcal{U}_{n}|X_{1},\ldots,X_{n}) \\
& & \hspace{-4em} \leq \biggr\{P_{G_0,f_0}(1-\phi_{n}) \biggr(\int 
		\limits_{\mathcal{U}_{n}}{\prod \limits_{i=1}^{n}
		{\dfrac{p_{G}(X_{i})}{p_{G_{*}}(X_{i})}d\Pi(G)}}\biggr)
		\biggr\} \nonumber \\
& & \hspace{ 4 em} \bigg/\biggr\{ \Pi(B_{K}^{*}(\epsilon, G_{*},P_{G_{0},f_0},M))
		\exp(-(1+C)n(\epsilon^{2}\log(M/\epsilon) +\epsilon)\biggr
		\}\nonumber. 
\end{eqnarray}
By means of Fubini's theorem, we obtain that
\begin{eqnarray}
\label{eq:posterior_contraction_rate_misspecified_second}
P_{G_0,f_0}(1-\phi_{n})\biggr(\int \limits_{\mathcal{U}_{n} 
		\cap \Gcal_{n}}{\prod \limits_{i=1}^{n}{\dfrac{p_{G}
		(X_{i})}{p_{G_{*}}(X_{i})}d\Pi(G)}}\biggr) & = &  \int
		\limits_{\mathcal{U}_{n} \cap \Gcal_{n}} \dfrac{P_{G_0,f_0}}
		{p_{G_{*}}}p_{G} (1-\phi_{n}) d\Pi(G)  \nonumber \\
& \leq & \sum \limits_{j \geq M_{n}}{\Pi(S_{n,j})\exp(-n
		\Psiba_{\Gcal_{n}}(j\epsilon)/8)} 
\end{eqnarray}
where the last inequality is due to inequality~\eqref{eq:test_complement_ball_second} and condition~\eqref{eq:posterior_contraction_rate_misspecified_first}. 
Furthermore, by means of Fubini's theorem
\begin{eqnarray}
\label{eq:posterior_contraction_rate_misspecified_third}
P_{G_0,f_0}(1-\phi_{n})\biggr(\int \limits_{\mathcal{U}_{n}
		 \backslash \Gcal_{n}}{\prod \limits_{i=1}^{n}
		 {\dfrac{p_{G}(X_{i})}{p_{G_{*}}(X_{i})}d\Pi(G)}}\biggr) & 
		 \leq & P_{G_0,f_0} \int \limits_{\mathcal{U}_{n} \backslash 
		 \Gcal_{n}}{\prod \limits_{i=1}^{n}{\dfrac{p_{G}(X_{i})}
		 {p_{G_{*}}(X_{i})}d\Pi(G)}} \nonumber \\
		 & = & \int \limits_{\mathcal{U}_{n} \backslash \Gcal_{n}}
		 \biggr(\prod \limits_{i=1}^{n} \int \dfrac{p_{G}(x_{i})}
		 {p_{G_{*}}(x_{i})}p_{G_0,f_0}(x_{i})dx_{i} \biggr) \Pi(G) \nonumber \\
		& \leq & \int \limits_{\mathcal{U}_{n} \backslash 
		\Gcal_{n}} \Pi(G) = \Pi(\mathcal{U}_{n} \backslash 
		\Gcal_{n}) \leq \Pi(\overline{\Gcal} \backslash \Gcal_{n}) 
\end{eqnarray}
where the second inequality in the above result is due to Lemma~\ref{lemma:misspecified_optimal}. 
By combining the results of~\eqref{eq:posterior_contraction_rate_misspecified_first},~\eqref{eq:posterior_contraction_rate_misspecified_second}, and~\eqref{eq:posterior_contraction_rate_misspecified_third}, we obtain
\begin{eqnarray}
& & P_{G_0,f_0}(1-\phi_{n})\mathbbm{1}_{A_{n}^{c}}\Pi(G \in 
		\mathcal{U}_{n}|X_{1},\ldots,X_{n}) \nonumber \\
& & \hspace{5 em} \leq  \dfrac{\sum \limits_{j \geq M_{n}}
		{\Pi(S_{n,j})\exp(-n\Psiba_{\Gcal_{n}}(j\epsilon)/8)}+
		\Pi(\overline{\Gcal} \backslash \Gcal_{n})}{ \Pi(B_{K}
		^{*}(\epsilon, G_{*},P_{G_{0},f_0},M))\exp(-(1+C)n(\epsilon^{2}
		\log(M/\epsilon) +\epsilon)} \nonumber \\
& & \hspace{5 em} \leq \exp((1+C)n(\epsilon^{2}\log(M/\epsilon) +
		\epsilon)\sum \limits_{j \geq M_{n}}{\exp\left(-n
		\Psiba_{\Gcal_{n}}(j\epsilon_{n})/16\right)}\nonumber \\ & 
		& \hspace{20 em} + o(\exp((C-1)n(\epsilon^{2}\log{M/
		\epsilon} +\epsilon)) \nonumber 
\end{eqnarray}
where the last inequality is due to condition~\eqref{eq:posterior_contraction_rate_misspecified_second} and~\eqref{eq:posterior_contraction_rate_misspecified_third}. 
If $n\log(1/\epsilon_n)^{-2} $ is bounded away from 0 by choosing 
$C\geq 1$, the right hand side term of the above display will go 
to 0 due to condition~\eqref{eq:posterior_contraction_rate_fourth_condition}. Therefore, 
we have $P_{G_0,f_0}(1-\phi_{n})\mathbbm{1}_{A_{n}^{c}}\Pi(G \in 
\mathcal{U}_{n}|X_{1},\ldots,X_{n}) \to 0$ as $n \to \infty$.

As a consequence, $P_{G_0,f_0} \Pi(G \in \mathcal{U}_{n}|X_{1
},\ldots,X_{n}) \to 0$.  We achieve the conclusion of the theorem.
\subsection{Proof of Lemma~\ref{lemma:existence_test_misspecified_kernel}}
For the setting (1) when $\mathcal{S}$ is convex, 
since $B_{W_{2}}(G_{1},r/2)$ is a convex set, we also have $
\mathcal{S} \cap B_{W_{2}}(G_{1},r/2)$ is a convex set. By means 
of the result of Theorem 6.1 in~\cite{Kleijn-2006}, there exist 
tests $\phi_{n}$ such that
\begin{eqnarray}
P_{G_0,f_0} \phi_{n} 
		& \leq & \left[1-\dfrac{1}{2}\inf \limits_{G \in 
		\mathcal{S} \cap B_{W_{2}}(G_{1},r/2)}\hba^{2}
		(p_{G},p_{G_{*}})\right]^{n}, \nonumber \\
		\sup \limits_{G \in \mathcal{S} \cap B_{W_{2}}(G_{1},r/2)} 
		\dfrac{P_{G_0,f_0}}{p_{G_{*}}}p_{G} (1-\phi_{n}) & \leq & 
		\left[1-\dfrac{1}{2}\inf \limits_{G \in \mathcal{S} \cap 
		B_{W_{2}}(G_{1},r/2)}\hba^{2}(p_{G},p_{G_{*}})\right]^{n}. \nonumber
\end{eqnarray} 
Due to inequality $(1-x)^{n} \leq \exp(-n x)$ for all $0<x<1$ 
and $n \geq 1$, the above inequalities become
\begin{eqnarray}
P_{G_0,f_0} \phi_{n} 
		& \leq & \exp\biggr(-\dfrac{n}{2}\inf \limits_{G \in 
		\mathcal{S} \cap B_{W_{2}}(G_{1},r/2)}\hba^{2}
		(p_{G},p_{G_{*}})\biggr), \nonumber \\
		\sup \limits_{G \in \mathcal{S} \cap B_{W_{2}}(G_{1},r/2)} 
		\dfrac{P_{G_0,f_0}}{p_{G_{*}}}p_{G} (1-\phi_{n}) & \leq & 
		\exp\biggr(-\dfrac{n}{2}\inf \limits_{G \in \mathcal{S} 
		\cap B_{W_{2}}(G_{1},r/2)}\hba^{2}(p_{G},p_{G_{*}})
		\biggr). \nonumber
\end{eqnarray}
Now, since $W_{2}(G_{1},G_{*}) = r$ and $W_{2}(G,G_{1}) \leq r/2$ 
as long as $G \in \mathcal{S} \cap B_{W_{2}}(G_{1},r/2)$, it 
implies that $W_{2}(G_{1},G_{*}) \geq r/2$. Therefore, according 
to Definition~\ref{definition:hellinger_to_Wasserstein}, we will 
obtain that
\begin{eqnarray}
\Psiba_{\Scal}(r)= \inf \limits_{G \in \Scal: \ 
		d_{W_{2}}(G,G_{*}) \geq r/2} \hba^{2}(p_{G},p_{G_{*}}) 
		\leq \inf \limits_{G \in \mathcal{S} \cap B_{W_{2}}
		(G_{1},r/2)} \hba^{2}(p_{G},p_{G_{*}}). \nonumber
\end{eqnarray}
With the above inequality, we reach the conclusion of part (1). 

Regarding part (2), we consider a maximal $c_{0}r$-packing of $
\Scal \cap B_{W_{2}}(G_{1},r/2)$ under $W_{2}$ metric. It gives us 
a set of $\overline{M}=\overline{M}(\Scal,G_{1},r) = D(c_{0}r,
\Scal \cap B_{W_{2}}(G_{1},r/2), W_{2})$ points 
$\widetilde{G}_{1},\ldots,\widetilde{G}_{\overline{M}}$ in $\Scal 
\cap B_{W_{2}}(G_{1},r/2)$. 

Now, for any $G \in \Scal \cap B_{W_{2}}(G_{1},r/2)$, we can find 
$t \in \left\{1,\ldots,\overline{M}\right\}$ such that $W_{1}(G,
\widetilde{G}_{t}) \leq c_{0}r$. Due to the triangle inequality, 
we achieve that
\begin{eqnarray}
\hba(p_{G},p_{G_{*}}) 
		& \geq & \hba(p_{G_{*}},p_{\widetilde{G}_{t}})-
		\hba(p_{G},p_{\widetilde{G}_{t}}) \nonumber \\
		& \geq & \biggr(\Psiba_{\Scal}(r)\biggr)^{1/2} - 
		\biggr(\overline{C}(\Theta)c_{0}r \biggr)^{1/2} \nonumber
\end{eqnarray}
where the second inequality is due to Definition~\ref{definition:hellinger_to_Wasserstein} 
and Lemma~\ref{lemma:upper_bound_weighted_Hellinger_Wasserstein}. 
By choosing the positive number $c_{0}=\Psiba_{\Scal}(r)/(4\overline{C}(\Theta)r)$, we obtain 
\begin{eqnarray}
\hba(p_{G},p_{\widetilde{G}_{t}}) \leq \Psiba_{\Scal}(r)/2 \leq 
\hba(p_{G_{*}},p_{\widetilde{G}_{t}})/2. \nonumber
\end{eqnarray}
It eventually leads to $\hba(p_{G},p_{G_{*}}) \geq 
\hba(p_{G_{*}},p_{\widetilde{G}_{t}})/2$. According to the result 
of Theorem 6.1 in~\cite{Kleijn-2006} and inequality $(1-x)^{n} 
\leq \exp(-n x)$ for all $0<x<1$ and $n \geq 1$, by denoting 
\begin{align*}
A_{t} : = \left\{G \in \overline{\Gcal}: \ 
\hba(p_{G},p_{\widetilde{G}_{t}}) \leq 
\hba(p_{G_{*}},p_{\widetilde{G}_{t}})/2 \right\},
\end{align*} 
there exists 
test $\psi_{n}^{(t)}$ such that
\begin{eqnarray}
P_{G_0,f_0} \psi_{n}^{(t)} 
		& \leq & \exp\biggr(-\dfrac{n}{2}\inf \limits_{G \in 
		A_{t}}\hba^{2}(p_{G},p_{G_{*}})\biggr), \nonumber \\
		\sup \limits_{G \in A_{t}} \dfrac{P_{G_0,f_0}}{p_{G_{*}}}
		p_{G} (1-\psi_{n}^{(t)}) & \leq & \exp\biggr(-\dfrac{n}{2}
		\inf \limits_{G \in A_{t}}\hba^{2}(p_{G},p_{G_{*}})\biggr). \nonumber
\end{eqnarray}
Since $\hba^{2}(p_{G},p_{G_{*}}) 
\geq \hba(p_{G_{*}},p_{\widetilde{G}_{t}})/2 \geq \Psiba_{\Scal}(r)/2$ 
for all $G \in A_{t}$, the above inequalities can be rewritten as
\begin{eqnarray}
P_{G_0,f_0} \psi_{n}^{(t)} 
		& \leq & \exp(-n\Psiba_{\Scal}(r)/8), \nonumber \\
		\sup \limits_{G \in A_{t}} \dfrac{P_{G_0,f_0}}{p_{G_{*}}}		
		p_{G} (1-\psi_{n}^{(t)}) 
		& \leq & \exp(-n\Psiba_{\Scal}(r)/8). \nonumber
\end{eqnarray}
By choosing $\phi_{n} = \max \limits_{1 \leq t 
\leq \overline{M}} \psi_{n}^{(t)}$, we quickly achieve that
\begin{eqnarray}
P_{G_0,f_0} \phi_{n} 
		& \leq & \overline{M}(\Scal,G_{1},r) \exp(-n\Psiba_{\Scal}	
		(r)/8), \nonumber \\
\sup \limits_{G \in \Scal \cap B_{W}(G_{1},r/2)} \dfrac{P_{G_0,f_0}}
		{p_{G_{*}}}p_{G} (1-\phi_{n}) 
		& \leq & \exp(-n\Psiba_{\Scal}(r)/8). \nonumber
\end{eqnarray}
As a consequence, we obtain the conclusion of the lemma.
\subsection{Proof of Theorem~\ref{theorem:convergence_rate_misspecified}}
\label{subs:proof_misspecified_finite_k }

Due to the assumption on prior $p_{K}$, it is sufficient to demonstrate that
\begin{eqnarray}
\Pi\biggr(G \in \mathcal{O}_{\overline{k}}(\Theta): 
		W_{2}(G,G_{*}) \gtrsim \dfrac{(\log n)^{1/4}}{n^{1/4}}|
		X_{1},\ldots,X_{n}\biggr) \to 0 \nonumber
\end{eqnarray}
in $P_{G_0,f_0}$- probability. 
We divide our proof for the above result into the following steps
\paragraph{Step 1:} To obtain the bound for 
$\Pi(B_{K}^{*}(\epsilon_{n},G_{*},P_{G_{0},f_0},M))$, we use Lemma 8.1 from~\cite{Kleijn-2006} 
to obtain a bound for weighted KL divergence and squared weighted 
KL divergence. Similar to the proof of Theorem~\ref{theorem:posterior_rate_Gaussian}, 
with the choice of $p=P_{G_0,f_0}$ and of finite measure 
$q=p_{G}p_{G_0,f_0}/p_{G_{*}}$, as long as 
$G \in \Ocal_{\overline{k}}$ such that $\hba(p_{G},p_{G_{*}}) \leq \epsilon$ and 
$\int p_{G_0,f_0}p_{G_{*}}/p_{G} \leq M$ then we obtain that
\begin{eqnarray}
&-P_{G_0,f_0} \log \dfrac{p_{G}}{p_{G_{*}}} 
		\lesssim \epsilon^{2}\log(M/\epsilon) + \epsilon \nonumber \\
& P_{G_0,f_0}\biggr(\log \dfrac{p_{G}}{p_{G_{*}}}\biggr)^{2} 
		\leq \epsilon^{2}(\log(M/\epsilon))^{2}. \nonumber
\end{eqnarray}
For the purpose of this proof we use $M=M^{*}(\epsilon_0)$, where $M^{*}(\epsilon_0)$ is as in condition (M.2) in Section~\ref{ssection:finite k}. Now, according to Lemma~\ref{lemma:upper_bound_weighted_Hellinger_Wasserstein}, 
as $f$ admits integral Lipschitz property up to the first 
order, we obtain that $\hba^{2}(p_{G},p_{G_{*}}) \leq \overline{C}W_{1}(G,G_{*})$ 
for any $G \in \Ocal_{\overline{k}}$ where $\overline{C}$ is a 
positive constant depending only on $\Theta$. Now, from the 
discussion in the above paragraph we have 
\begin{eqnarray}
\Pi(B_{K}^{*}(\epsilon_{n},G_*,P_{G_{0},f_0},M)) 
		& \geq & \Pi(G \in \Ocal_{\overline{k}}: W_{1}(G,G_{*}) 
		\leq \overline{C}\epsilon_n^2) \nonumber \\
& \gtrsim & \Pi(G \in \Ecal_{\overline{k}}: W_{1}(G,G_{*}) 
		\leq \overline{C}\epsilon_n^2) \nonumber \\
& \gtrsim & \epsilon_{n}^{2(c_{H}+\gamma)} \nonumber.
\end{eqnarray}
where the last inequality can be obtained similar to equation~\eqref{eq:bound_KL_well_specified} 
based on the assumption $\gamma<\overline{k}$. Now, we note that $
\Psiba_{\Gcal_n(r)} \gtrsim r^4$ , since $f$ is assumed to be 
second order identifiable and to satisfy the integral 
Lipschitz property of second order. Then, by choosing 
$\epsilon_n = n^{-1}$ and $M_n= \overline{A}\left(\frac{\log(n)}
{n^3}\right)^{1/4}$ for some sufficiently large $\overline{A}$, we 
can see that condition~\eqref{eq:posterior_contraction_rate_third_condition} is 
satisfied.
\paragraph{Step 2:} We choose the sieves 
$\Gcal_{n} = \Ocal_{\overline{k}}$ for all $n \geq 1$. With these 
choices, it is clear that $\Pi(\overline{\Gcal}\backslash 
\Gcal_{n})=0$. Therefore, condition~\eqref{eq:posterior_contraction_rate_second_condition} is satisfied. 

\paragraph{Step 3:}  For condition~\eqref{eq:posterior_contraction_rate_fourth_condition} 
to be satisfied,
\begin{eqnarray}
\exp\left(2n \left( \epsilon_{n}^{2}\log 
		\left(\frac{M}{\epsilon_n}\right) + \epsilon_n \right) 
		\right)\sum \limits_{j \geq M_{n}}{\exp\left(-n
		\Psiba_{\Gcal_{n}}(j\epsilon_{n})/16\right)} & \lesssim &  
		\sum \limits_{j \geq M_{n}} \exp({-\overline{C} {n}^{-3} 
		M_{n}^{4}/16}) \nonumber \\
& \lesssim & 2\exp({-\overline{C} {n}^{-3} M_{n}^{4}/16}) \nonumber \\
& \lesssim & 2 n^{- \frac{\overline{A}\overline{C}}{16}} \to 0. \nonumber
\end{eqnarray}

\end{document}